\documentclass[11pt,letterpaper]{amsart} 
\usepackage[left=3.18cm,right=3.18cm,top=3cm,bottom=3.1cm]{geometry}
\usepackage{amsfonts}
\usepackage{amsmath}
\usepackage{amssymb}
\usepackage{mathtools}
\usepackage{amsthm}
\usepackage{pgfplots}
\pgfplotsset{compat=newest}
\usepackage[shortlabels]{enumitem}
\usepackage{tikz}   
\usepackage{tikz-cd}  
\usepackage{changepage}
\usetikzlibrary{arrows}
\usepackage{bm}
\usepackage{graphicx}
\usepackage{cases}
\usepackage{appendix}
\usepackage{esint}
\usepackage{cite}
\usepackage{hyperref}
\usepackage{mathrsfs}
\numberwithin{equation}{section}

\theoremstyle{definition}
\newtheorem{thm}{Theorem}[section]
\newtheorem{claim}[thm]{Claim}
\newtheorem{defn}[thm]{Definition}
\newtheorem{lem}[thm]{Lemma}
\newtheorem{prop}[thm]{Proposition}
\newtheorem{cor}[thm]{Corollary}
\newtheorem*{rmk}{Remark}

\newcommand{\R}{\mathbb{R}}  
\newcommand{\Sp}{\mathbb{S}}
\newcommand{\p}{\partial}  

\newcommand{\dif}{\textup{d}} 

\newcommand{\Hau}{\mathcal{H}} 
\newcommand{\sff}{\mathrm{I\!I}} 
\newcommand{\diam}{\textup{diam}} 
\newcommand{\dist}{\textup{dist}}

\newcommand{\sing}{\textup{sing }}
\newcommand{\spt}{\textup{supp }}

\begin{document}
\title[Energy Identity of Ginzburg-Landau approximation]{Energy Identity for Ginzburg-Landau approximation of harmonic maps}
\author{Xuanyu Li}
\address{Department of Mathematics, Cornell University, Ithaca, NY 14853, USA}
\email{xl896@cornell.edu}

\begin{abstract}
Given two Riemannian manifolds $M$ and $N\subset\R^J$, we consider the energy concentration phenomena of the penalized energy functional $$E_{\epsilon}(u)=\int_M\frac{\vert \nabla u\vert^2}{2}+\frac{F(u)}{\epsilon^2},u\in W^{1,2}(M,\R^J),$$where $F(x)=\dist(x,N)^2$ in a small tubular neighborhood of $N$ and is constant away from $N$. It was shown by Chen-Struwe that as $\epsilon\rightarrow0$, the critical points $u_{\epsilon}$ of $E_{\epsilon}$ with energy bound $E_{\epsilon}(u_{\epsilon})\leqslant\Lambda$ subsequentially converge weakly in $W^{1,2}$ to a weak harmonic map $u:M\rightarrow N$. In addition, we have the convergence of the energy density $$\left(\frac{\vert \nabla u_{\epsilon}\vert^2}{2}+\frac{F(u_{\epsilon})}{\epsilon^2}\right)\dif x\rightarrow\frac{1}{2}\vert \nabla u\vert^2\dif x+\nu,$$ and the defect measure $\nu$ above is $(\dim M-2)$-rectifiable. Lin-Wang showed that if $N$ is a sphere or $\dim M=2$, then the density of $\nu$ can be expressed by the sum of energies of harmonic spheres. In this paper, we prove this result for an arbitrary $M$ using the idea introduced by Naber-Valtorta.
\end{abstract}
\maketitle
\tableofcontents
\section{Introduction}\label{s: intro}
Fix smooth compact manifolds $M^m$ and $N^n$. For convenience, we assume $N$ is isometrically embedded in $\R^J$ for some large $J$ and $n\geqslant2$. Then there exists a small $\gamma>0$ such that $\dist(x,N)^2$ is smooth on $\lbrace x:\dist(x,N)<\gamma\rbrace$. In this paper, we are interested in the limiting behavior of following penalized energy functional
$$E_{\epsilon}(u)=\int_Me_{\epsilon}(u),e_{\epsilon}(u)=\frac{\vert \nabla u\vert^2}{2}+\frac{F(u)}{\epsilon^2},u\in W^{1,2}(M,\R^J),$$where 
$$F(x)=F(d^2(x,N))=\begin{cases}
        d^2(x,N),&d(x,N)<\gamma;\\
        4\gamma^2,&d(x,N)\geqslant2\gamma.
\end{cases}$$
This energy functional was first introduced by Chen-Struwe \cite{ChenStruweexistenceharmonicmapflow} to study the existence of harmonic map flow. By standard elliptic estimates, the critical points $u_{\epsilon}\in W^{1,2}(M)$ are smooth and satisfies the following equation 
\begin{equation}
    \Delta u_{\epsilon}=\frac{f(u_{\epsilon})}{\epsilon^2},\quad \text{where } f=\nabla F.\tag{*}\label{*}
\end{equation}
Let $u_{\epsilon_i}$ be a sequence of solutions to (\ref{*}) with energy bound $E_{\epsilon_i}(u_{\epsilon_i})\leqslant\Lambda$. As shown in \cite{ChenStruweexistenceharmonicmapflow}, if $\epsilon_i\rightarrow0$, then $u_{\epsilon_i}$ subconverges weakly in $W^{1,2}$ to a weakly harmonic map $u$. The energy loss in this convergence is quantitatively characterized by the convergence of energy measure
$$e_{\epsilon_i}(u_{\epsilon_i})\dif x\rightarrow\frac{1}{2}\vert \nabla u\vert^2\dif x+\nu,$$
where the defect measure $\nu=\theta(x)\Hau^{m-2}\llcorner\Sigma$ is $(m-2)$-rectifiable. This is called \emph{energy concentration} in the literature. There are two known facts shown by Lin-Wang in \cite{Linwangharmonicsphere99,Linwangharmonicsphere02}. First, for a general target $N$, the partial energy identity is true:$$\theta(x)\geqslant\sum_i \frac{1}{2}\int_{\Sp^2}\vert \nabla u_i\vert^2,$$
where $u_i$ are harmonic maps from $\Sp^2$ to $N$, usually referred as \emph{harmonic spheres}. Second, in the case of $\dim M=2$ or $N$ is a standard sphere, we have a full identity
$$\theta(x)=\sum_i \frac{1}{2}\int_{\Sp^2}\vert \nabla u_i\vert^2.$$
In this paper, we are going to prove such an \emph{energy identity} for general target $N$.
\begin{thm}\label{main theorem}
    Let $u_{\epsilon_i}:M\rightarrow \R^J$ be a sequence of solutions to (\ref{*}) with $\int_Me_{\epsilon_i}(u_{\epsilon_i})\leqslant\Lambda$. Suppose that as $\epsilon_i\rightarrow0$, $u_{\epsilon_i}$ converges weakly to a harmonic map $u:M\rightarrow N$ with the defect measure $\nu$. For $\Hau^{m-2}$ a.e. $x\in\spt\nu$, there exists harmonic spheres $\lbrace u_j\rbrace_{j=1}^N$ such that the density of $\nu$ at $x$ satisfies 
    $$\theta_{\nu}(x)=\sum_{j=1}^N\frac{1}{2}\int_{\Sp^2}\vert \nabla u_j\vert^2.$$
\end{thm}
Since the result is local in nature, we only present the proof for the special case where $M$ is a domain in $\R^m$. \par

\vspace{1em}

The phenomenon of energy concentration plays an important role in the study of harmonic maps and minimal surfaces. In dimension two, this phenomenon was first deployed by Sacks-Uhlenbeck for their $\alpha$-harmonic maps \cite{SUminimalsphere}. Later, Jost \cite{Jostvariationalproblems} and Parker \cite{Parkerbubbletree} proved the full energy identity, as well as the bubble tree convergence for 2-dimensional harmonic maps. See also \cite{BreinerLakzianbubbletreemetricspace,DingTianenergyidentity,LiZhuSUenergyidentity,LWharmonicmapflow98,WWZapproxharmonicmapenergyidentity} for related work on harmonic map flows, approximate harmonic maps. We also refer the reader to \cite{ColdingMinicozziwidthofricciflow,Zhouminmaxtorus,Zhouhighgenusminmaxsurface} for the applications of energy identity on the existence of minimal surfaces.\par

In higher dimensions, the analysis of such concentration phenomenon becomes much more complicated. This is practically because the concentration set is no longer a discrete set, and non-isolated singularities may occur. The first breakthrough was made by Lin \cite{Linharmonic99}. 
It was proved that the defect measure $\nu$ is $(m-2)$-rectifiable and the partial energy identity was established. The next progress was made by Lin-Riviere \cite{LinRiviereenergyidentity}. They showed that the energy identity is true if a sequence of harmonic map is uniformly bounded in $W^{2,1}$ and verified this condition for harmonic maps to a standard sphere. Riviere \cite{Riviereyangmillsenergyidentity} showed a similar result for the Yang-Mills connections \cite{Riviereyangmillsenergyidentity}. If the domain of a harmonic map is 2-dimensional, the $W^{2,1}$ estimate was established by Lamm-Sharp \cite{LammSharpW12conjecture}, but it remains open in higher dimensions.\par

It turns out that proving a full energy identity in high dimensions is extremely difficult. For Yang-Mills connections, this result was established by Naber-Valtorta \cite{NVyangmills}. Very recently, Naber-Valtorta proved the full energy identity for harmonic maps \cite{naber2024energyidentitystationaryharmonic}.\par

\vspace{1em}

The Ginzburg-Landau approximation is a useful tool in the study, especially the existence theorem, of harmonic maps. Chen-Struwe \cite{ChenStruweexistenceharmonicmapflow} first used the Ginzburg-Landau approximation to prove the existence of harmonic map flow starting from an arbitrary map. Very recently, Karpukhin-Stern \cite{KSminmaxharmonic24,KSharmonic24} introduced a novel min-max procedure for the penalized energy to derive the existence of harmonic maps for high dimensional domains. The disadvantage of Ginzburg-Landau approximation is that, in order to derive the existence of stationary harmonic maps, one must assume the target manifold do not admit certain harmonic spheres to avoid the energy concentration in the convergence scheme. On the other hand, a general target manifold may admit harmonic sphere. We wish to understand further what is happening in the general case. In particular, if the energy concentration does exist, if it is possible to still derive the existence of stationary harmonic sphere as in \cite{KSminmaxharmonic24,KSharmonic24}. Proving a full energy identity is always a first step in the study of fine structure of defect measure and will help us understand better the behavior of energy functional.

\subsection{Sketch of the proof}
Suppose we have a sequence of solutions to (\ref{*}) $u_{\epsilon}$ with $E_{\epsilon}(u_{\epsilon})\leqslant\Lambda$ which converges weakly to a harmonic map $u$ with 
$$e_{\epsilon}(u_{\epsilon})\dif x\rightarrow\frac{1}{2}\vert\nabla u\vert^2\dif x+\nu,\textup{ as }\epsilon\rightarrow0.$$ 
The defect measure $\nu$ is $(m-2)$-rectifiable by Lin-Wang \cite{Linharmonic99,Linwangharmonicsphere02}. 
Moreover, by the $\varepsilon$-regularity of solutions to (\ref{*}), $\Hau^{m-2}(\sing u)=0$. Hence, for $\Hau^{m-2}$ a.e. $x\in \spt(\nu)$, there exists a unique tangent plane $L$ of $\nu$ and $u$ is smooth at $x$. Rescaling each $u_{\epsilon}$ at $x$, we may assume $u_{\epsilon}$ converges weakly to a constant together with the convergence of energy measure 
$$e_{\epsilon}(u_{\epsilon})\dif x\rightarrow\theta\Hau^{m-2}\llcorner L.$$
We want to show that $\theta$ is a sum of energies of harmonic spheres.
There are two parts of the proof. We mainly make use of the ideas and arguments of Naber-Valtorta in their energy identity papers for harmonic maps and Yang-Mills connections \cite{NVyangmills,naber2024energyidentitystationaryharmonic}. To corporate with the Ginzburg-Landau setting, We have the following major difference and improvement.

\begin{itemize}
    \item The annular region needed in this paper is slightly different from that in \cite{NVyangmills}. In particular, the annular region in \cite{NVyangmills} is the annular neighborhood of a discrete set $\mathcal{C}$ while the annular region introduced in \cite{naber2024energyidentitystationaryharmonic} and in this paper is the annular neighborhood of an $(m-2)$-submanifold $\mathcal{T}$. We use the argument in \cite{NVyangmills} together with a gluing argument to provide full details on the construction of annular region we need.
    \item There are certain differences between the convergence scheme of harmonic maps and Ginzburg-Landau approximation. In the construction of bubble region, if we have a sequence of harmonic maps $u_i$ and we blow up at points $x_i$ and scale $r_i$, then $v_i=(u_i)_{x_i,r_i}$ still converges to a harmonic map. On the other hand, if we have a sequence of Ginzburg-Landau approximation $u_{\epsilon_i}$, the blow up $v_{\epsilon_i/r_i}=(u_{\epsilon_i})_{x_i,r_i}$ may not converge to a harmonic map. Instead, it may converge to a solution to (\ref{*}) if $\epsilon_i/r_i\rightarrow\epsilon>0$, or a harmonic function on $\R^m$ if $\epsilon_i/r_i\rightarrow\infty$. We rule out the latter two situations by observing that there is no nontrivial finite energy solution to (\ref{*}) and harmonic function. Similar observations was first developed by Lin-Wang \cite{Linwangharmonicsphere99}.
    \item There is a penalized term $F(u_{\epsilon})/\epsilon^2$ in the energy which may create extra difficulty when we are proving the energy in annular region is small. However, we will see in Section \ref{s: pre} that this term together with $L$-energy forms the $(m-2)$-symmetric of Ginzburg-Landau approximation. When the solution to (\ref{*}) exhibits high $(m-2)$-symmetry, this term is naturally small. 
\end{itemize}
We will delve into these in details in the following.

\vspace{1em}

The first part is a quantitative bubble decomposition Theorem. Near the plane $L$, $u_{\epsilon}$ finally exhibits almost $(m-2)$-symmetries. As a result, if we look at small regions around $L$, then $u_{\epsilon}$ looks like a 2-dimensional harmonic map on $L^{\perp}$ in the strong sense. Those regions are called bubble regions. The energy of $u_{\epsilon}$ in the bubble region is captured by the energy of the harmonic spheres. The regions between bubble regions are called annular regions. \par 

In Theorem \ref{bubble region existence}, \ref{annular region existence} and \ref{Quantitative bubble decomposition}, we use the ideas from \cite{NVyangmills} to decompose $B_1\cap L$ into bubble regions and annular regions effectively. In Theorem \ref{annular region existence}, we construct the annular regions in the maximal sense, which means that these annular regions almost capture all the spaces between bubbles where the map $u_{\epsilon}$ looks trivial. Moreover, the annular regions built here are of the form $\mathcal{A}=B_1(p)\setminus\overline{B_{\mathfrak{r}_x}(\mathcal{T})}$, where $\mathcal{T}$ is a submanifold with small curvature and $\mathfrak{r}_x$ is a radial function with small derivative. The formal definition we use here was introduced in \cite{naber2024energyidentitystationaryharmonic} and will provide much convenience in further discussion. \par

The second part is to show that the energy in annular regions is small enough. After constructing the bubble regions and annular regions, it is not hard to see the well-known energy inequality 
$$\theta\geqslant\sum_j\frac{1}{2}\int\vert\nabla u_j\vert^2,$$
where $u_j$ are harmonic spheres.  The energy in annular region $\mathcal{A}$ is naturally decomposed into 3 parts:
\begin{equation}
    \begin{aligned}
        \int_{\mathcal{A}}e_{\epsilon}(u_{\epsilon})&=\frac{1}{2}\left(\int_{\mathcal{A}}\left\vert\Pi_{L^{\perp}}\nabla u_{\epsilon}\right\vert^2+\left\vert\Pi_{L}\nabla u_{\epsilon}\right\vert^2+2F(u_{\epsilon})/\epsilon^2\right)\\&=\frac{1}{2}\left(\int_{\mathcal{A}}\left\vert\nabla_{n_{L^{\perp}}}u_{\epsilon}\right\vert^2+\left\vert\nabla_{\alpha_{L^{\perp}}}u_{\epsilon}\right\vert^2+\left\vert\Pi_{L}\nabla u_{\epsilon}\right\vert^2+2F(u_{\epsilon})/\epsilon^2\right).
    \end{aligned}\nonumber
\end{equation}
Here $n_{L^{\perp}}$ is the radial direction and $\alpha_{L^{\perp}}$ is the angular direction in $L^{\perp}$ respectively. The tangential energy $$E_L(u_{\epsilon})=\int_{\mathcal{A}}\left\vert\Pi_{L}\nabla u_{\epsilon}\right\vert^2+2F(u_{\epsilon})/\epsilon^2,$$vanishes naturally due to the $(m-2)$-symmetric of $u_{\epsilon_i}$. The angular energy is also not hard to control. To see this, if we set
$$e_{\alpha}(x)=s_{L^{\perp}}^2\int_{\alpha_{L^{\perp}}\in S^1}\left\vert\nabla_{\alpha_{L^{\perp}}} u_{\epsilon_i}(y)\right\vert^2\dif\alpha_{L^{\perp}},$$where $y\in x+L^{\perp}$ and we write the polar coordinate in $x+L^{\perp}$ as $s_{L^{\perp}}$ and $\alpha_{L^{\perp}}$, then we can verify the following is true in the annular region
$$\bar{\Delta} e_{\alpha}(x)\geqslant e_{\alpha}(x),\textup{ for conformal Laplacian }\bar{\Delta}=\left(s_{L^{\perp}}\frac{\p}{\p s_{L^{\perp}}}\right)^2+s_{L^{\perp}}^2\Delta_L.$$
Integrating above yields the desired estimate for angular energy. In dimension 2, this fact is well known and widely used to prove the energy identity. It was first used by Parker in \cite{Parkerbubbletree}. Very recently, Naber-Valtorta \cite[section 9]{naber2024energyidentitystationaryharmonic} generalized this fact to high dimensions.\par

The radial energy $$E_n(u_{\epsilon})=\int_{\mathcal{A}}\left\vert\nabla_{n_{L^{\perp}}}u_{\epsilon}\right\vert^2$$ is the most challenging part to control. Fix a positive cut-off function $\phi$, suppose $\Pi_{L^{\perp}}\nabla\phi(y)=\varphi(y)\Pi_{L^{\perp}}(y),\varphi\leqslant0$, multiply (\ref{*}) by $\phi(y-x)\nabla_{\Pi_L^{\perp}(y-x)}u_{\epsilon}$ and integrate by parts, we get
\begin{equation}
    \begin{aligned}
        &\int\phi(y-x)\left(\vert\Pi_L\nabla u_{\epsilon}\vert^2+2\frac{F(u_{\epsilon})}{\epsilon^2}\right)-\frac{1}{2}\varphi(y-x)\left\vert\Pi_{L^{\perp}}(y-x)\right\vert^2\vert\nabla_{n_{L^{\perp}}}u_{\epsilon}\vert^2\\=&\frac{1}{2}\int-\varphi(y-x)\vert\Pi_{L^{\perp}}(y-x)\vert^2\left(\left(\vert\Pi_L\nabla u_{\epsilon}\vert^2+2\frac{F(u_{\epsilon})}{\epsilon^2}\right)+\vert\nabla_{\alpha_{L^{\perp}}}u_{\epsilon}\vert^2\right)\\&+\langle\nabla_{\Pi_L\nabla\phi}u_{\epsilon},\nabla_{\Pi_{L^{\perp}}(y-x)}u_{\epsilon}\rangle.
    \end{aligned}\label{almost conformality}
\end{equation}
This equation will provide control of the radial energy by the angular energy. 
In dimension 2, this equation is simply the conformality
$$\int_{\p B_r(x)}\left(\vert\nabla_\alpha u_{\epsilon}\vert^2+2F(u_{\epsilon})/\epsilon^2\right)=\int_{\p B_r(x)}\vert\nabla_n u_{\epsilon}\vert^2+\frac{1}{r}\int_{B_r(x)}F(u_{\epsilon})/\epsilon^2.$$
One can immediately control the radial energy from this identity and the estimate of angular energy.\par

In higher dimensions, however, there are many extra terms in (\ref{almost conformality}). The first term on the right side is a good term. As we would see in section \ref{s:radial energy}, this is a derivative term and can be easily controlled. The difficult part is the cross derivative term
\begin{equation}
    \int\langle\nabla_{\Pi_L\nabla\phi}u_{\epsilon},\nabla_{\Pi_{L^{\perp}}(y-x)}u_{\epsilon}\rangle.\label{sketch1}
\end{equation}
A priori, there is no smallness condition on this part. In particular, we cannot make $\Pi_L\nabla\phi$ supported away from $\mathcal{T}$. Consequently, it is impossible to directly use the Cauchy inequality to break down this term into $L$ and $L^{\perp}$ parts so as to control it. This is essentially because the $L^{\perp}$ part of energy is concentrated near $\mathcal{T}$. \par

We make use of the new techniques introduced in \cite{naber2024energyidentitystationaryharmonic} to handle this term. The basic tool is the heat mollified energy introduced in \cite{naber2024energyidentitystationaryharmonic}, which is defined as $$\bar{\Theta}(x,r)=\int r^2\rho_r(y-x)e_{\epsilon}(u_{\epsilon})(y)\dif y.$$
Here $\rho_r(x)$ is almost the heat kernel $r^{-m}e^{-\vert y\vert^2/(2r^2)}$ although it is modified to compactly supported in a large ball $B_{R+2}$. The strength of heat mollified energy is that it can provide much convenience when we want to exchange derivatives due to the fact that $\rho_r\approx \dot{\rho}_r$. The heat mollified energy is widely used in the study of flows, but this version for stationary maps was first introduced by Naber-Valtorta in \cite{naber2024energyidentitystationaryharmonic}.\par

The first step towards the solution is the construction of local best planes $\mathcal{L}_{x,r}$. For $x\in\mathcal{A}$ and $r\geqslant2\dist(x,\mathcal{T})$, this plane is defined as $$\mathcal{L}_{x,r}=\arg\min_{L}\int r^2\rho_r(y-x)\left\vert\Pi_L\nabla u_{\epsilon}\right\vert^2,$$
which satisfies Euler-Lagrange equation $$\int\rho_r(y-x)\langle\nabla_vu_{\epsilon},\nabla_wu_{\epsilon}\rangle=0,\textup{ for all }v\in\mathcal{L}_{x,r},w\in\mathcal{L}_{x,r}^{\perp}.$$Moreover, $\mathcal{L}_{x,r}$ is unique, close to $L$ and has good estimates as a function of $x$ and $r$. By gluing together $\mathcal{L}_{x,r}$ for different $x$, we build the best approximating submanifolds $\mathcal{T}_r$. The construction here guarantees that $\mathcal{T}_r$ are all close to $\mathcal{T}$ the union of all $B_1(p)\setminus B_r(\mathcal{T}_r)$ for $r\geqslant\mathfrak{r}_x$ contains $\mathcal{A}$. The most important property for $\mathcal{T}_r$ is its Euler-Lagrange equation
$$\int-\dot{\rho}_r(y-x)\langle\nabla_{y-x}u_{\epsilon},\nabla_vu_{\epsilon}\rangle=0,\textup{ for all }v\in\mathcal{L}_{x,r}^{\perp}.$$
This equation finally allows us to turn an $\mathcal{L}_{x,r}^{\perp}$ derivative into a $\mathcal{L}_{x,r}$ derivative. On the other hand, since the tangent space of $\mathcal{T}_r$ at $x$ is almost $\mathcal{L}_{x,r}$, we can integrate by parts in (\ref{sketch1}) to show this is in fact a negligible term. See Section \ref{s:radial energy} for details.

\subsection{Outline of the paper}

In Section \ref{ss:harmonic map intro} and $\ref{ss:Ginzburg-Landau intro}$ we introduce the concepts and basic properties for solutions of (\ref{*}). In Section \ref{ss:heat mollified energy} we introduce the heat mollified energy for solutions of (\ref{*}). As already seen in \cite{naber2024energyidentitystationaryharmonic}, this is a powerful tool in the study of geometric variational problems. In section \ref{ss: cone splitting}, we establish various versions of cone splitting results. Roughly speaking, those results assert that the 0-symmetries of a map at independent points add up to $k$-symmetry of the map. Finally, in section \ref{ss:symmetries} we use compactness argument to derive the corresponding symmetries of defect measures.\par
Section \ref{s:annular and bubble regions} introduces and proves the existence of annular regions and bubble regions. Also, we show how to decompose $L\cap B_1$ into annular regions, bubble regions and small negligible sets. Finally, in this section we show how to use the fact that annular regions only admit small energy to prove the quantitative energy identity \ref{Quantitative energy identity}. Theorem \ref{main theorem} is an easy corollary of quantitative energy identity.\par
Section \ref{s:best planes}-\ref{s:radial energy} are devoted to prove there is no energy in annular regions. In Section \ref{s:best planes} and \ref{s:best submanifold} we build the local best planes and best approximating submanifolds and prove the basic estimates for them. In Section \ref{s:energy decomposition} we formally introduce the tree parts of energy: tangential, angular and radial part of energy as discussed in the sketch of proof. We show that those three parts of energy add up to give the full energy in the annular region.\par
Finally, we show three parts of energy are all small in the annular region. In Section \ref{s:angular energy} we establish the desired super-convexity of angular energy and use this property to show the smallness of angular energy. In Section \ref{s:radial energy} we prove the smallness of the tangential and radial parts of the energy. In particular, we prove how to use the equations of local best planes and best approximating submanifolds to write (\ref{sketch1}) as good terms.

\subsection*{Acknowledgment}
The author would like to thank Xin Zhou and Daniel Stern for their many valuable discussions and constant support. The author would like to thank Aaron Naber for a detailed discussion on his work of harmonic maps. The author would also like to thank Zhihan Wang for his interest in this work and discussion on related topics. 

\section{preliminaries}\label{s: pre}
In this section we introduce the basic definitions in this paper and fix some notations. Throughout this paper, we assume $N^n$ with $n\geqslant2$ is isometrically embedded in $\R^J$ with curvature bound
$$\vert\sec_N\vert\leqslant K_N^2,\textup{inj}(N)\geqslant K_N^{-1},\diam(N)\leqslant K_N.$$
The domain of our map will always be in the open subset of Euclidean space $\R^m$ and the covariant derivative on $\R^m$ will be denoted by $\nabla$. 

\subsection{Harmonic maps}\label{ss:harmonic map intro}
The usual energy functional for maps in $W^{1,2}(M,N)$ is
$$E(u)=\frac{1}{2}\int_M\vert \nabla u\vert^2,u\in W^{1,2}(M,N).$$
It is natural to consider the variational problem associated to $E$. If $u$ is a critical point to $E$ with respect to compact variations, we say $u$ is a (stationary) \emph{harmonic map}.\par 
In general, a harmonic map may not be smooth but can have $(m-2)$-rectifiable singular set. In the case $m=2$, the harmonic maps are well studied and have the following nice estimate.
\begin{thm}\label{2d estimate}
    Let $u: B_1^2\rightarrow N$ be a harmonic map with $\int_{B_1}\vert \nabla u\vert^2\leqslant\Lambda$. Then $u$ is smooth. Moreover, on each annulus $B_{2R}(p)\setminus B_r(p)$ with $R\geqslant r$, set 
    $$\delta=\max_{B_{2R}(p)\setminus B_R(p)}\vert x-p\vert^2\vert \nabla u\vert^2+\max_{B_{2r}(p)\setminus B_r(p)}\vert x-p\vert^2\vert \nabla u\vert^2.$$There exists $\varepsilon_0,\alpha>0$ depending only on $K_N$ such that if 
    $$\vert x-p\vert^2\vert \nabla u\vert^2\leqslant\varepsilon_0\textup{ on }B_{2R}(p)\setminus B_r(p),$$ then we have the improved estimate: $$\vert x-p\vert^2\vert \nabla u\vert^2\leqslant C(K_N)\left(\left(\frac{\vert x-p\vert}{2R}\right)^{\alpha}+\left(\frac{r}{\vert x-p\vert}\right)^{\alpha}\right)\delta.$$
\end{thm}
We refer the reader to \cite[Section III]{dalio2023morseindexstabilitycritical} for a proof.

\subsection{Ginzburg-Landau approximation for harmonic maps}\label{ss:Ginzburg-Landau intro}
Throughout this paper, we are interested in the solution to (\ref{*}). We also need to frequently use a corollary of (\ref{*}). Suppose $u_{\epsilon}$ is the critical point of $E_{\epsilon}$, for each $\xi\in C^{\infty}_0(B_1,\R^m)$, if we consider the variation $t\mapsto u(\cdot+t\xi)$, or equivalently, multiply (\ref{*}) by $\nabla_{\xi}u_{\epsilon}$ and then integrate by parts, we get the following equation 
\begin{equation}
    \int e_{\epsilon}(u_\epsilon)\textup{div}\xi=\sum_{\alpha,\beta=1}^m\int \langle \p_{\alpha}u_{\epsilon},\p_{\alpha}u_{\epsilon}\rangle \p_{\alpha}\xi^\beta,\forall\xi\in C_0^{\infty}(B_1,\R^m).\tag{**}\label{**}
\end{equation}
This is the \emph{stationary equation} satisfied by $u_{\epsilon}$.\par
It turns out that the solutions to (\ref{*}) are good approximations for harmonic maps in the sense that if $u_{\epsilon_j}$ is a sequence of solutions to (\ref{*}) with $\epsilon_j\rightarrow0$ and $E_{\epsilon_j}(u_{\epsilon_j})\leqslant\Lambda$. After passing to a subsequence, we have that $u_{\epsilon_j}$ converges weakly in $W^{1,2}$ to a (weakly) harmonic map $u$. The strength of studying solutions to (\ref{*}) is that they are all smooth, although the elliptic estimates may depend on $\epsilon$. Moreover, we have the following $\varepsilon_0$-regularity that is independent of $\epsilon$.
\begin{thm}\label{e-reg thm}
    There exists $\varepsilon_0=\varepsilon_0(m,K_N)>0$ such that if $u_{\epsilon}$ solves (\ref{*}) with $$(2r)^{2-m}\int_{B_{2r}(p)}e_{\epsilon}(u_{\epsilon})\leqslant\varepsilon_0,$$ then there holds $$r^2\sup_{B_r(p)}e_{\epsilon}(u_{\epsilon})\leqslant C(m,K_N)\frac{1}{(2r)^{m-2}}\int_{B_{2r}(p)}e_{\epsilon}(u_{\epsilon}).$$
    Moreover, if $\epsilon/r\leqslant c(m,K_N)$ is sufficiently small, we have
    \begin{equation}
        \begin{aligned}
            &r^2\sup_{B_r(p)}\left\vert\frac{f(u_{\epsilon})}{\epsilon^2}\right\vert\\\leqslant &C(m,K_N)\left(\frac{1}{(2r)^{m-2}}\int_{B_{2r}(p)}e_{\epsilon}(u_{\epsilon})+\frac{r}{\epsilon}e^{-C(m,K_N)\frac{r}{\epsilon}}\left(\frac{1}{(2r)^{m-2}}\int_{B_{2r}(p)}e_{\epsilon}(u_{\epsilon})\right)^\frac{1}{2}\right).
        \end{aligned}\nonumber
    \end{equation}
\end{thm}
The proof for the first part of Theorem \ref{e-reg thm} can be found in \cite[Section 5]{KSharmonic24} and second part can be found in \cite[Section 2]{Linwangharmonicsphere02}. As a corollary, we have the estimate on the lower bound of derivatives of $f(u_{\epsilon})$.
\begin{cor}\label{second derivative est}
    If $u_{\epsilon}$ and $\epsilon/r$ satisfies the condition in Theorem \ref{e-reg thm}, then we have for any unit vector $v$, the following holds on $B_r(p)$ $$r^2\left\langle\nabla_v u_{\epsilon},\nabla_v\frac{f(u_{\epsilon})}{\epsilon^2}\right\rangle\geqslant-C(m,K_n)\left(\frac{1}{(2r)^{m-2}}\int_{B_{2r}(p)}e_{\epsilon}(u_{\epsilon})\right)^\frac{1}{2}\vert\nabla_vu_{\epsilon}\vert^2.$$
\end{cor}
\begin{proof}
Given the Theorem \ref{e-reg thm}, we see that $f(u_{\epsilon})=2\textup{dist}(u_{\epsilon},N)\nabla\dist(y,N)|_{y=u_{\epsilon}}$. We then calculate 
    \begin{equation}
        \begin{aligned}
            &r^2\left\langle\nabla_v u_{\epsilon},\nabla_v\frac{f(u_{\epsilon})}{\epsilon^2}\right\rangle\\=&2\left(\langle\nabla\textup{dist}(y,N)|_{y=u_{\epsilon}},\nabla_vu_{\epsilon}\rangle^2+\textup{dist}(u_{\epsilon},N)\nabla^2\textup{dist}(y,N)|_{y=u_{\epsilon}}(\nabla_vu_{\epsilon},\nabla_vu_{\epsilon})\right)\frac{r^2}{\epsilon^2}\\\geqslant&-C(m,K_N)r^2\frac{\textup{dist}(u_{\epsilon},N)}{\epsilon^2}\vert\nabla_vu_{\epsilon}\vert^2.
        \end{aligned}\nonumber
    \end{equation}
    Theorem \ref{e-reg thm} gives the desired result.
\end{proof}

\subsection{Heat mollified energy}\label{ss:heat mollified energy}
Given a solution $u_{\epsilon}:B_1\rightarrow\R^J$ to (\ref{*}), the usual energy density of Ginzburg-Landau approximation is defined by
$$\Theta(x,r)=\frac{1}{r^{m-2}}\int_{B_r(x)}e_{\epsilon}(u_{\epsilon}).$$
$\Theta(x,r)$ is monotone increasing in $r$. In this paper, instead of considering $\Theta$, we will consider a mollified energy density introduced in \cite{naber2024energyidentitystationaryharmonic}, which has better behavior in the space and time variable. To be specific, let us introduce the mollified heat kernel:
\begin{defn}[Mollified heat kernel,{cf. \cite[Theorem 2.4]{NVyangmills}}]
    In the following discussion, $\rho:[0,\infty)\rightarrow[0,\infty)$ will be fixed a smooth function depending on large $R>0$ satisfies
    \begin{enumerate}[(1)]
        \item $\rho(t)=c_me^{-t},0\leqslant t\leqslant R$;
        \item $\textup{supp}\rho\subset[0,R+2]$;
        \item $0\leqslant\Ddot{\rho}(t)\leqslant c_me^{-R}$.
        \item $\int_{\R^m}\rho\left(\frac{\vert y\vert^2}{2}\right)\dif y=1$.
    \end{enumerate}
    Here $\dot{\rho}$ denotes the derivative of $\rho$. We will also set the rescaling of $\rho$ by 
    $$\rho_r(y)=\frac{1}{r^m}\rho\left(\frac{\vert y\vert^2}{2r^2}\right),\dot{\rho}_r(y)=\frac{1}{r^m}\dot{\rho}\left(\frac{\vert y\vert^2}{2r^2}\right).$$
\end{defn}
We will select $R$ in a later stage of proof. The selection of $\rho$ gives us much convenience in the proof. The most important property of a heat kernel is that its derivative is exactly itself. Our mollified heat kernel inherits this property as shown in the following lemma. The proof is by direct calculation.
\begin{lem}\label{properties of heat mollifier}[cf. \cite[Lemma 2.8]{naber2024energyidentitystationaryharmonic}]
    For all $R$ and $t\geqslant0$:
    \begin{enumerate}[(1)]
        \item $-\dot{\rho}\geqslant0,-\dot{\rho}(t)-t\dot{\rho}(t)+t^2\Ddot{\rho}(t)\leqslant C\min\lbrace\rho(t/1.1^2),-\dot{\rho}(t/1.1^2)\rbrace$;
        \item $\rho(t)\leqslant-C(m)\dot{\rho}(t)$;
        \item $\vert\rho(t)+\dot{\rho}(t)\vert\leqslant2c_me^{-R}\chi_{[R,R+2]}\leqslant20e^{-R/2}\rho(t/4)$;
        \item If $R\geqslant20,r\leqslant s\leqslant10r$, we have $$B_r(x)\subset B_s(x')\implies\rho_r(y-x)\leqslant C(m)\rho_s(y-x').$$
    \end{enumerate}
\end{lem}
Following \cite[Section 2]{naber2024energyidentitystationaryharmonic}, we define the heat mollified energies as follows.
\begin{defn}[Heat mollified energies,{cf. \cite[Definition 2.10]{NVyangmills}}]
    Let $u_{\epsilon}: B_{10R}\rightarrow\R^J$ be a solution to (\ref{*}). For each $B_r(x)\subset B_1$, define $$\bar{\Theta}(x,r)=r^2\int\rho_r(y-x)e_{\epsilon}(u_{\epsilon})(y)\dif y.$$Moreover, for each $k$-dimensional subspace $L$ of $\R^m$, denote $\Pi_L$ the orthogonal projection to $L$, we define the partial energies by 
    $$\bar{\Theta}(x,r;L)=r^2\int\rho_r(y-x)\vert\Pi_L\nabla u_{\epsilon}\vert^2+2\frac{F(u_{\epsilon})}{\epsilon^2},$$ $$\bar{\Theta}(x,r;n_{L^{\perp}})=\int\rho_r(y-x)\vert\nabla_{n_{L^{\perp}}}u_{\epsilon}\vert^2\vert\Pi_{L^{\perp}}(y-x)\vert^2,$$ $$\bar{\Theta}(x,r;\alpha_{L^{\perp}})=\int\rho_r(y-x)\vert\nabla_{\alpha_{L^{\perp}}}u_{\epsilon}\vert^2\vert\Pi_{L^{\perp}}(y-x)\vert^2,$$where $n_{L^{\perp}}(y)=\Pi_{L^{\perp}}(y-x)/\vert\Pi_{L^{\perp}}(y-x)\vert$ is the radial direction in $L^{\perp}$ while $\alpha_{L^{\perp}}$ are the rest angular directions. Finally, define the heat mollified partial energy in $L^{\perp}$ by
    \begin{equation}
        \begin{aligned}
            \bar{\Theta}(x,r;L^{\perp})&=\bar{\Theta}(x,r;n_{L^{\perp}})+\bar{\Theta}(x,r;\alpha_{L^{\perp}})\\&=\int\rho_r(y-x)\vert\Pi_{L^{\perp}}(y-x)\vert^2\vert\Pi_{L^{\perp}}\nabla u\vert^2.
        \end{aligned}\nonumber
    \end{equation}
\end{defn}
It is easy to see that we have simple comparison between $\bar{\Theta}$ and $\Theta$:
$$\Theta(x,r)\leqslant c(m)\bar{\Theta}(x,r)\leqslant C(m)\Theta(x,2Rr).$$
As classic energy density, the heat mollified one enjoys good monotonicity formula. If we take $\xi=\rho_r(y-x)(y-x)$ in stationary equation (\ref{**}), we get 
\begin{equation}
    r\frac{\p}{\p r}\bar{\Theta}(x,r)=\int-\dot{\rho}_r(y-x)\vert\nabla_{x-y} u_{\epsilon}\vert^2+r^2\frac{F(u_{\epsilon})}{\epsilon^2}\rho_r(y-x)\dif y\geqslant0.\label{monotonicity}
\end{equation}
After mollifying using the heat kernels, the energy functional get better behavior both on $x$ and $r$ variables: we will see in the following quantitative cone splitting lemma that we can control those energies effectively by the radial derivative of $\bar{\Theta}$. This calculation allows us to do quantitative estimates.\par
We define the stratification of a solution to (\ref{*}) related to monotone quantity $\bar{\Theta}$.
\begin{defn}[{cf. \cite[Definition 2.13]{naber2024energyidentitystationaryharmonic}}]
    Let $u_{\epsilon}:B_{10R}\rightarrow\R^J$ be a solution to (\ref{*}). For each $\delta>0$, we say $u_{\epsilon}$ is $(k,\delta)$-symmetric on $B_r(x)\subset B_1$ with respect to $k$-dimensional subspace $L\subset\R^m$ if $$\bar{\Theta}(x,r;L)+\bar{\Theta}(x,r;n_{L^{\perp}})\leqslant\delta.$$
\end{defn}
In this paper, we only consider the case where $k=m-2$ or $m-1$ in above definition.

\subsection{Quantitative cone splitting results}\label{ss: cone splitting}
In this section we prove the cone splitting results associated to the monotone quantity $\bar{\Theta}$. Namely, if a map looks like 0-homogeneous at $(k+1)$ different points near a $k$-subspace, then it indeed looks like a $k$-symmetric cone in the center of those points. In other words, the $(0,\gamma)$ symmetric at those $(k+1)$ points can be lifted to the $(k,\gamma')$ symmetric at the center point. The idea of stratification goes back to Federer \cite{FedererGMT} to study the singular sets of minimal currents. The quantitative version of it was first observed by Cheeger-Naber \cite{CheegerNaberStrac13} to derive the refined estimates on singular sets of geometric variational problems. The heat mollified energies allow those cone splitting results to be even more effective. To state the results, we first need a definition.
\begin{defn}[{\cite[Definition 2.18]{naber2024energyidentitystationaryharmonic}}]\label{effective independent}
    Given $x_0,\dots,x_k\in B_r(x_0)$, we say that they are $\alpha$-linearly independent at scale $r$ if $$x_{i+1}\notin B_{\alpha r}(x_0+L_i),L_i=\textup{span}\lbrace x_1-x_0,\dots,x_i-x_0\rbrace.$$
\end{defn}
In other words, $\lbrace x_i-x_0\rbrace_{i=1}^k$ spans a $k$-subspace effectively. Hence, we have the following corollary of this definition.
\begin{lem}[{\cite[Lemma 2.19]{naber2024energyidentitystationaryharmonic}}]\label{est for independ points}
    If $x_0,\dots,x_k$ are $\alpha$-linear independent at scale $r$, then for every $q\in L=\textup{span}\lbrace x_1-x_0,\dots,x_k-x_0\rbrace+x_0$, there exists unique $q_1,\dots,q_k\in\R$ such that $$q=x_0+\sum_{i=1}^kq_i(x_i-x_0)\textup{ with }\vert q_i\vert\leqslant C(m,\alpha)\frac{\vert q-x_0\vert}{r}.$$
\end{lem}
\begin{prop}[{Quantitative cone splitting, cf. \cite[Theorem 2.20]{naber2024energyidentitystationaryharmonic}}]\label{quantitative cone splitting}
    Let $u_{\epsilon}:B_{10R}(p)\rightarrow\R^J$ be a solution to (\ref{*}) . Let $x_0,\dots,x_k\in B_r(x_0)\subset B_1(p)$ be $\alpha$-linear independent at scale $r$. Then 
    \begin{equation}
        \begin{aligned}
            &\bar{\Theta}(x_0,r;L)+\bar{\Theta}(x_0,r;n_{L^{\perp}})\\=&\int\rho_r(y-x_0)(r^2(\vert\nabla_L u_{\epsilon}\vert^2+F(u_{\epsilon})/\epsilon^2)+\vert\Pi_{L^{\perp}}(y-x_0)\vert^2\vert\nabla_{n_{L^{\perp}}}u_{\epsilon}\vert^2)\\\leqslant &C(m,\alpha)r\sum_{i=0}^k\frac{\p}{\p r}\bar{\Theta}(x_i,1.4r).
        \end{aligned}\nonumber
    \end{equation}
\end{prop}
\begin{proof}
    Scaling, we may assume $r=1$. By Lemma \ref{est for independ points}, we have that for all $l\in L$, 
    $$\vert\nabla_l u_{\epsilon}\vert^2\leqslant2\sum_{i=1}^kl_i^2\vert\nabla_{x_i-x_0}u_{\epsilon}\vert^2\leqslant C(m,\alpha)\sum_{i=0}^k\vert\nabla_{y-x_i}u_{\epsilon}\vert^2.$$
    Also, using Lemma \ref{properties of heat mollifier}(1), we have $\rho_1(y-x_0)\leqslant-C(m)\dot{\rho}_{1.1}(y-x_i)$. Hence \begin{equation}
        \begin{aligned}
            \bar{\Theta}(x_0,1;L)&=\int\rho_1(y-x_0)(\vert\nabla_Lu_{\epsilon}\vert^2+F(u_{\epsilon})/\epsilon^2)\\&\leqslant C(m,\alpha)\sum_{i=0}^k\int-\dot{\rho}_{1.1}(y-x_i)\vert\nabla_{y-x_i}u_{\epsilon}\vert^2+\rho_{1.1}(y-x_i)F(u_{\epsilon})/\epsilon^2\\&=\sum_{i=0}^k\frac{\p}{\p r}\bar{\Theta}(x_i,1.1).
        \end{aligned}\nonumber
    \end{equation}
    On the other hand,
    \begin{equation}
        \begin{aligned}
            \bar{\Theta}(x_0,1;n_{L^{\perp}})&=\int\rho_1(y-x_0)\vert\Pi_{L^{\perp}}(y-x)\vert^2\vert\nabla_{n_{L^{\perp}}}u_{\epsilon}\vert^2\\&\leqslant2\int\rho_1(y-x_0)(\vert\nabla_{y-x_0}u_{\epsilon}\vert^2+\vert\nabla_{\Pi_L(y-x_0)}u_{\epsilon}\vert^2)\\&\leqslant\frac{\p}{\p r}\bar{\Theta}(x_0,1.1)+2\int\vert\nabla_{\Pi_L(y-x_0)}u_{\epsilon}\vert^2\rho_1(y-x_0).
        \end{aligned}\nonumber
    \end{equation}
    Using $-t\dot{\rho}(t)\leqslant-\dot{\rho}(t/1.1)$ we have $$\int\vert\nabla_{\Pi_L(y-x_0)}u_{\epsilon}\vert^2\rho_1(y-x_0)\leqslant\int\rho_1(y-x_0)\vert\nabla_Lu_{\epsilon}\vert^2\vert y-x_0\vert^2\leqslant C(m)\frac{\p}{\p r}\bar{\Theta}(x_0,1.1).$$
\end{proof}
When $k=m-2$ is the top dimension, it is possible to capture the full energy in $L^{\perp}$ using stationary equation (\ref{**}). Note that when $m=2$, Proposition \ref{Top dimension quantitative cone splitting} can be viewed as a rephrase of the fact 
$$\int_{\p B_r(x)}\left(\vert\nabla_\alpha u_{\epsilon}\vert^2+2F(u_{\epsilon})/\epsilon^2\right)=\int_{\p B_r(x)}\vert\nabla_n u_{\epsilon}\vert^2+\frac{1}{r}\int_{B_r(x)}F(u_{\epsilon})/\epsilon^2.$$
\begin{prop}[{Top dimension quantitative cone splitting, cf. \cite[Theorem 2.21]{naber2024energyidentitystationaryharmonic}}]\label{Top dimension quantitative cone splitting}
    Let $u_{\epsilon}:B_{10R}(p)\rightarrow\R^J$ be a solution to (\ref{*}). Let $x_0,\dots,x_{m-2}\in B_r(x_0)\subset B_1(p)$ be $\alpha$-linear independent at scale $r$. Then 
    \begin{equation}
        \begin{aligned}
            &\bar{\Theta}(x_0,r;L)+\bar{\Theta}(x_0,r;L^{\perp})\\=&\int\rho_r(y-x_0)(r^2(\vert\nabla_L u_{\epsilon}\vert^2+F(u_{\epsilon})/\epsilon^2)+\vert\Pi_{L^{\perp}}(y-x_0)\vert^2\vert\nabla_{L^{\perp}}u_{\epsilon}\vert^2)\\\leqslant& C(m,\alpha)r\sum_{i=0}^{m-2}\frac{\p}{\p r}\bar{\Theta}(x_i,1.4r).
        \end{aligned}\nonumber
    \end{equation}
\end{prop}
\begin{proof}
    Assume $r=1$. We only need to show the estimate of $\bar{\Theta}(x,r;\alpha_{L^{\perp}})$. Let us choose $\xi(y)=P(\vert y-x_0\vert^2/2)\Pi_{L^{\perp}}(y-x_0)$ in stationary equation (\ref{**}), where $P$ is the only compactly supported function on $[0,\infty)$ with $\dot{P}=-\rho$. Then
    \begin{equation}
        \begin{aligned}
            &\int e_{\epsilon}(u_{\epsilon})(2P(\vert y-x_0\vert^2/2)-\rho_1(y-x_0)\vert\Pi_{L^{\perp}}(y-x_0)\vert^2)\\=&\int\vert\nabla_{L^{\perp}} u_{\epsilon}\vert^2P(\vert y-x_0\vert^2/2)-\rho_1(y-x_0)(\vert\nabla_{\Pi_{L^{\perp}}(y-x_0)}u_{\epsilon}\vert^2+\langle\nabla_{\Pi_{L^{\perp}}(y-x_0)}u_{\epsilon},\nabla_{\Pi_L(y-x_0)}u_{\epsilon}\rangle).
        \end{aligned}\nonumber
    \end{equation}
    Hence 
    \begin{equation}
        \begin{aligned}
            &\int P(\vert y-x_0\vert^2/2)(\vert\Pi_L\nabla u_{\epsilon}\vert^2+2F(u_{\epsilon})/\epsilon^2)-\rho_1(y-x_0)F(u_{\epsilon})/\epsilon^2\\=&\frac{1}{2}\int\rho_1(y-x_0)(\vert\Pi_{L^{\perp}}(y-x_0)\vert^2(\vert\nabla_L u_{\epsilon}\vert^2+\vert\nabla_{\alpha_{L^{\perp}}}u_{\epsilon}\vert^2-\vert\nabla_{n_{L^{\perp}}}u_{\epsilon}\vert^2)\\-&\int\rho_1(y-x_0)\langle\nabla_{\Pi_{L^{\perp}}(y-x_0)}u_{\epsilon},\nabla_{\Pi_L(y-x_0)}u_{\epsilon}\rangle).
        \end{aligned}\nonumber
    \end{equation}
    As a result, 
    \begin{equation}
        \begin{aligned}
            &\int\rho_1(y-x_0)\vert\nabla_{\alpha_{L^{\perp}}}u_{\epsilon}\vert^2\vert\Pi_{L^{\perp}}(y-x_0)\vert^2\\ \leqslant&C(m,\alpha)\int\rho_{1.1}(y-x_0)(\vert\nabla_L u_{\epsilon}\vert^2+F(u_{\epsilon})/\epsilon^2+\vert\Pi_{L^{\perp}}(y-x_0)\vert^2\vert\nabla_{n_{L^{\perp}}}u_{\epsilon}\vert^2).
        \end{aligned}\nonumber
    \end{equation}
    Proposition \ref{quantitative cone splitting} applies to give the desired result.
\end{proof}
After mollifying using the heat kernel, the energy density is also smooth in $x$ variable. It is clear that a map $u_{\epsilon}$ with zero $L$-energy should be $L$ invariant. We characterize this property by quantitative estimate on the spacial derivative of $\bar{\Theta}$.
\begin{prop}[{cf. \cite[Theorem 2.27, 2.29]{naber2024energyidentitystationaryharmonic}}]\label{spacial derivative est}
    Let $u_{\epsilon}:B_{10R}(p)\rightarrow\R^J$ be a solution to (\ref{*}). Then for any $L^k\subset\R^m,B_{2r}(x)\subset B_{1}(p)$ we have$$r^2\vert\nabla_L\bar{\Theta}(x,r)\vert^2\leqslant C(m)r\frac{\p}{\p r}\bar{\Theta}(x,r)\bar{\Theta}(x,2r;L)$$ $$r^4\vert\nabla^2_L\bar{\Theta}(x,r)\vert^2\leqslant C(m)\left(r\frac{\p}{\p r}\bar{\Theta}(x,r)\bar{\Theta}(x,2r;L)+\bar{\Theta}(x,2r;L)^2\right)$$
\end{prop}
\begin{proof}
    Take any $l\in L$, then $$r\nabla_l\bar{\Theta}(x,r)=r^2\int\dot{\rho}_r(y-x)\left\langle \frac{x-y}{r},l\right\rangle e_{\epsilon}(u_{\epsilon}).$$ Taking $\xi=\rho_r(y-x)l$ in stationary equation (\ref{**}) yields \begin{equation}
        r^2\int\dot{\rho}_r(y-x)\left\langle \frac{x-y}{r},l\right\rangle e_{\epsilon}(u_{\epsilon})=r^2\int\dot{\rho}_r(y-x)\left\langle\nabla_{\frac{x-y}{r}}u_{\epsilon},\nabla_lu_{\epsilon}\right\rangle.\label{spacial gradient}
    \end{equation}
    Taking derivative to both sides, we have  for all $v,w\in L$ 
    \begin{equation}
    \begin{aligned}
         &r^2\nabla^2_L\bar{\Theta}(x,r)[v,w]\\=&r^2\int\ddot{\rho}_r(y-x)\left\langle\nabla_{\frac{x-y}{r}}u_{\epsilon},\nabla_wu_{\epsilon}\right\rangle\left\langle\frac{x-y}{r},v\right\rangle+\dot{\rho}_r(y-x)\left\langle\nabla_vu_{\epsilon},\nabla_wu_{\epsilon}\right\rangle.\label{spcial hessian}
    \end{aligned}
    \end{equation}
    Hence $$r^2\vert\nabla_l\bar{\Theta}(x,r)\vert^2\leqslant \left(r^2\int-\dot{\rho}_r(y-x)\vert\nabla_l u_{\epsilon}\vert^2\right)\left(\int-\dot{\rho}_r(y-x)\vert\nabla_{y-x}u_\epsilon\vert^2\right).$$Using Lemma \ref{properties of heat mollifier}, we get the estimate for the $\nabla_L\bar{\Theta}$. Similarly we have the estimate for the Hessian of $\bar{\Theta}$.
\end{proof}

\subsection{Symmetries of maps}\label{ss:symmetries}
Let us see some corollaries of the quantitative cone splitting results in Section \ref{ss: cone splitting}. First, by Proposition \ref{spacial derivative est}, we have that if a sequence of solutions to (\ref{*}) has vanishing $L$-energy, then both the limiting map and defect measure is invariant with respect to $L$. 
\begin{lem}[{cf. \cite[Theorem 2.7]{NVyangmills}}]\label{Invariant with plane}
    Let $u_{\epsilon_i}$ be a sequence of solutions to (\ref{*}) with $\epsilon_i\rightarrow0$. Suppose $u_{\epsilon_i}\rightharpoonup u$ and $e_{\epsilon_i}(u_{\epsilon_i})\dif x\rightarrow\vert\nabla u\vert^2\dif x+\nu$. If there exists a $k$-dimensional subspace $L$ in $\R^m$ such that $\bar{\Theta}_{u_{\epsilon_i}}(p,1;L)\rightarrow0$, then both $u$ and $\nu$ are invariant with respect to $L$. Moreover,
    \begin{enumerate}[(1)]
        \item If $k=m-2$, then $u$ is smooth and $\nu=\sum_k\theta_k\Hau^{m-2}\llcorner(p_k+L)$ with $\theta_k\geqslant\varepsilon_0$;
        \item If $u_{\epsilon_i}$ is $(m-2,i^{-1})$ symmetric with respect to $L$ on $B_1(p)$ for large $i$, then $u$ is constant and $\nu=\theta\Hau^{m-2}\llcorner(p+L)$;
    \end{enumerate}
\end{lem}
\begin{proof}
    For (1), by Proposition \ref{spacial derivative est}, $\bar{\Theta}(p,1;L)\rightarrow0$ implies that $\nu$ is invariant with respect to $L$. Since $\nu$ is $(m-2)$-rectifiable, we must have $\nu=\sum_k\theta_k\Hau^{m-2}\llcorner(p_k+L)$. As a result, $u_{\epsilon_i}$ converges to $u$ smoothly away from $\cup_k(p_k+L)$. $\bar{\Theta}(p,1;L)\rightarrow0$ then implies that $u$ is invariant with respect to $L$.\\
    For (2), the top dimension cone splitting \ref{Top dimension quantitative cone splitting} and $\varepsilon$-regularity \ref{e-reg thm} implies that $u_{\epsilon_i}$ converges to a constant smoothly away from $L+p$. The conclusion then follows.
\end{proof}
By a contradiction argument, we also have the following corollaries.
\begin{lem}[{cf. \cite[Theorem 3.2]{NVyangmills}}]\label{m-2 symmetry alternative}
    Let $u_{\epsilon}:B_{10R}(p)\rightarrow\R^J$ be a solution to (\ref{*}) with $(10R)^{2-n}\int_{B_{10R}(p)}e_{\epsilon}(u_{\epsilon})\leqslant\Lambda$. For each $0<\sigma<\sigma(m,K_N)$, if $u_{\epsilon}$ is $(m-2,\delta)$ symmetric on $B_2(p)$ with $\delta<\delta(m,\Lambda,K_N,R,\sigma)$, then either 
    \begin{enumerate}[(1)]
        \item $\sup_{B_{3/2}(p)}e_{\epsilon}(u_{\epsilon})\leqslant\sigma$, or
        \item $\bar{\Theta}(x,1)\geqslant\varepsilon_0(m,K_N)$ and $r\frac{\p}{\p r}\bar{\Theta}(x,r)\leqslant\sigma$ for all $r\in[\sigma,2]$ and $x\in p+L\cap B_2$.
    \end{enumerate}
\end{lem}
\begin{proof}
    Suppose two conclusions above are both false for $u_{\epsilon_i}$ with $u_{\epsilon_i}$ $(m-2,1/i)$ symmetric. After passing to a subsequence we may assume, by Lemma \ref{Invariant with plane}, that $u_{\epsilon_i}$ converges weakly to a constant with defect measure $\nu=\theta\Hau^{m-2}\llcorner(p+L)$. Note that $r\frac{\p}{\p r}\bar{\Theta}_{u_{\epsilon_i}}(x,r)\rightarrow 0$ is automatic for $x\in p+L\cap B_2$ because $\bar{\Theta}_{u_{\epsilon_i}}(x,2r)-\bar{\Theta}_{u_{\epsilon_i}}(x,r)\rightarrow\theta-\theta=0$. On the other hand, Lemma \ref{properties of heat mollifier} implies that $\bar{\Theta}_{u_{\epsilon_i}}(x,2r)-\bar{\Theta}_{u_{\epsilon_i}}(x,r)\geqslant c(m)r\frac{\p}{\p r}\bar{\Theta}_{u_{\epsilon_i}}(x,r)$. Hence if (2) is false, we must have $u_{\epsilon_i}$ converges smoothly to a constant, which contradicts (1).
\end{proof}
\begin{lem}[{cf. \cite[Theorem 2.6]{NVyangmills}}]\label{m-1 flatness}
    Let $u_{\epsilon}:B_{10R}(p)\rightarrow\R^J$ be a solution to (\ref{*}) with $(10R)^{2-n}\int_{B_{10R}(p)}e_{\epsilon}(u_{\epsilon})\leqslant\Lambda$. For each $\sigma>0$, if $u_{\epsilon}$ is $(m-1,\delta)$ symmetric wrt $L$ on $B_2(p)$ with $\delta<\delta(m,\Lambda,K_N,R,\sigma)$, then 
    $$\sup_{B_{3/2}(p)}e_{\epsilon}(u_{\epsilon})<\sigma.$$ 
\end{lem}
\begin{proof}
    If there exists $u_{\epsilon_i}$ which is $(m-1,1/i)$ symmetric but does not satisfy the conclusion, after passing to a subsequence, we must have $u_{\epsilon_i}$ converges to a constant with defect measure $\nu$. $(m-1,1/i)$ symmetric of $u_{\epsilon_i}$ implies $\nu$ has $(m-1)$-symmetric. Then $\nu$ must be 0 otherwise it cannot be $(m-2)$-locally finite. Hence, $u_{\epsilon_i}$ converges to a constant strongly, which contradicts the assumption.
\end{proof}

\section{Annular regions and bubble regions}\label{s:annular and bubble regions}
To establish the energy identity, the first question to ask is how we can see the bubbles near the support of defect measure. Traditionally, bubbles come from the blow up near the top concentration points. The neighborhood of those points form regions, called \emph{bubble regions}, where we can see the energies of harmonic spheres. If we capture the energy in all bubble regions, we then get the energy inequality: $$\Theta_\nu(x)\geqslant\sum_iE(u_i),$$where $u_i$ are harmonic spheres. In order to get full energy identity, we need to show the regions between bubbles, called \emph{annular regions}, admit no energy. The formal definition we use here was first introduced by Naber-Valtorta in their work on the energy identity of Yang-Mills functional \cite{NVyangmills}.
\subsection{Bubble regions}
The first one we want to define is the bubble regions where we can see harmonic spheres.
\begin{defn}[{cf. \cite[Definition 3.19]{naber2024energyidentitystationaryharmonic}}]\label{bubble regions}
    Let $u_{\epsilon}: B_{10R}(p)\rightarrow\R^J$ be a solution to (\ref{*}) and an $m-2$-subspace $L$ in $\R^J$. Consider $\lbrace x_j\rbrace\subset p+L^{\perp}$ such that $\lbrace B_{r_j}(x_j+L)\rbrace$ are disjoint. We say $\mathcal{B}=B_r(p)\setminus\cup B_{r_j}(x_j+L)$ is a $\delta$-bubble region if there exists a $L$ invariant harmonic map $b:B_{Rr}(p)\rightarrow N$ such that 
    \begin{enumerate}
        \item[(b1)] $r^2(\vert\Pi_L\nabla u_{\epsilon}\vert^2+2F(u_{\epsilon})/\epsilon^2)<\delta$ on $\mathcal{B}$;
        \item[(b2)] $\vert b-u_{\epsilon}\vert^2+r^2\vert\nabla b-\nabla u_{\epsilon}\vert^2\leqslant\delta$ on $\mathcal{B}_R=B_{rR}(p)\setminus\cup B_{R^{-1}r_j}(x_i+L)$;
        \item[(b3)] $\Theta(p,1)-\Theta(x_j,r_j)\geqslant\varepsilon_0$ for all $j$. Here $\Theta(x,r)=r^{2-m}\int_{B_r(x)}e_{\epsilon}(u_{\epsilon})$ is the usual energy density; 
        \item[(b4)] For all $x\in x_j+L\cap B_1,\delta^2 r_j\leqslant r\leqslant 2r_j$ and some $\delta''<\delta$, $r\frac{\p}{\p r}\bar{\Theta}(x,r)\leqslant\delta^2$, we have $\bar{\Theta}(x,2r_j;L)\leqslant\delta''$ and $\bar{\Theta}(x,r_j)\geqslant\varepsilon_0$. 
    \end{enumerate}
\end{defn}
Given that $u_{\epsilon}$ is sufficiently symmetric, we can prove the following existence theorem for bubble regions.
\begin{thm}[{cf. \cite[Theorem 3.23]{naber2024energyidentitystationaryharmonic}}]\label{bubble region existence}
    Let $u_{\epsilon}:B_{10R\delta^{-1}}(p)\rightarrow\R^J$ be a solution to (\ref{*}) with $(2R)^{2-n}\int_{B_{2R}(p)}e_{\epsilon}(u_{\epsilon})\leqslant\Lambda$. For each $0<\delta\leqslant\delta(m,\Lambda,R,K_N)$ and $0<\delta''\leqslant\delta''(m,R,\Lambda,\delta,K_N)$, there exists $\delta'(m,\Lambda,R,K_N,\delta,,\delta''), \epsilon(m,\Lambda,R,K_N,\delta)>0$ such that if $\epsilon\leqslant\epsilon(m,\Lambda,R,K_N,\delta),\bar{\Theta}(x,2;L)\leqslant\delta'$ and $r\frac{\p}{\p r}\bar{\Theta}(p,r)\leqslant\delta$ for $r\in[\delta^{1/2},\delta^{-1/2}]$ but $$\lambda\frac{\p}{\p r}\bar{\Theta}(x,\lambda)\geqslant c(m,\Lambda,K_N)\delta\textup{ for }\lambda=c(m,K_N,\Lambda)\delta\textup{ and all }x\in B_{C(m)\delta}(p+L)\cap (p+L^{\perp}),$$ then there exists a $\delta$-bubble region $\mathcal{B}=B_1(p)\setminus\cup B_{r_j}(x_j+L)$ with $r_j\geqslant\bar{r}(m,\Lambda,K_N,\delta)$.
\end{thm}
\begin{rmk}
    By the assumption of Theorem \ref{bubble region existence} and quantitative cone splitting \ref{Top dimension quantitative cone splitting}, all bubble points $x_j\in B_{C(m,K_N,\Lambda)\delta}(p+L)$.
\end{rmk}
\begin{proof}
    We prove by contradiction. Assume that there exist $u_{\epsilon_i}$ solutions to (\ref{*}) with $\epsilon_i\rightarrow0$ and $\bar{\Theta}_{u_{\epsilon_i}}(x,1;L)\rightarrow0$ but $B_1(p)$ does not admit bubble regions for $u_{\epsilon_i}$ for sufficiently large $i$. Moreover, $r\frac{\p}{\p r}\bar{\Theta}_{u_{\epsilon_i}}(p,r)\leqslant\delta$ for all $r\in[1,\delta]$ and $r\frac{\p}{\p r}\bar{\Theta}_{u_{\epsilon_i}}(x,r)\geqslant\delta$ for some $r\in[\delta,1]$. After passing to a subsequence, we may assume $u_{\epsilon_i}\rightharpoonup b$ an $L$-invariant harmonic map into $N$ with defect measure $\nu=\sum_{j=1}^J\theta_j\Hau^{m-2}\llcorner(x_j+L)$ for $x_j\in p+L^{\perp}$.\par 
    Let us define bubble regions for large $i$. Pick $0<\delta''<\delta$ to be determined. For each $k$, since 
    \begin{equation}
        \begin{aligned}
            &\sum_{k=1}^K(\bar{\Theta}_{b}(x_j,2^{1-k}\delta^2)+\nu(B_{2^{1-k}\delta^2}(x_j))-\bar{\Theta}_{b}(x_j,2^{-k}\delta^2)-\nu(B_{2^{-k}\delta^2}(x_j)))\\=&\bar{\Theta}(x_j,\delta^2)+\nu(B_{\delta^2}(x_j))-\bar{\Theta}(x_j,2^{-K}\delta^2)-\nu(B_{2^{-K}\delta^2}(x_j))\leqslant\Lambda.
        \end{aligned}\nonumber
    \end{equation}
    We see that if $K=10\Lambda/\delta''$, then for some $k_j\in\lbrace1,\dots,K\rbrace$ we have $$\bar{\Theta}(x_j,2^{1-k_j}\delta^2)+\nu(B_{2^{1-k_j}\delta^2}(x_j))-\bar{\Theta}(x_j,2^{-k_j}\delta^2)-\nu(B_{2^{-k_j}\delta^2}(x_j))\leqslant\delta''/2.$$ 
    This property propagates to $u_{\epsilon_i}$ for large $i$. This implies for some $2r_j\in[2^{-k_j}\delta^2,2^{1-k_j}\delta^2]$, we have 
    $$2r_j\frac{\p}{\p r}\bar{\Theta}_{u_{\epsilon_i}}(x_j,2r_j)\leqslant\delta''.$$
    Moreover, if $i$ is large enough, we have$$\bar{\Theta}_{u_{\epsilon_i}}(x_j,2r_j;L)\leqslant C(m)r_j^{2-m}\bar{\Theta}(p,2;L)\leqslant\delta''.$$Hence, $u_{\epsilon_i}$ is $(m-2,\delta'')$ on $B_{2r_j}(x)$ and we conclude, using Lemma \ref{m-2 symmetry alternative}, that if $\delta''\leqslant\delta''(m,\Lambda,K_N,R,\delta)$,
    \begin{equation}
        r\frac{\p}{\p r}\bar{\Theta}_{u_{\epsilon_i}}(x_j,r)\leqslant\delta^2\textup{ for all }r\in[\delta^2 r_j,r_j].\label{3.22}
    \end{equation}
    Let us choose a subcovering (without relabeling) $\lbrace B_{2r_j}(x_j)\rbrace$ such that $\lbrace B_{r_j}(x_j)\rbrace$ is pairwise disjoint and verify that $\mathcal{B}=B_1(p)\setminus\cup B_{r_j}(x_j)$ defines a bubble region for $u_{\epsilon_i}$ if $i$ is large enough. The (b1),(b2),(b4) in Definition \ref{bubble regions} is clearly satisfied by our choice of $x_j$ and $r_j^{(i)}$.
    Next, we prove another part $$\int_{B_1(p)}e_{\epsilon_i}(u_{\epsilon_i})-r_j^{2-m}\int_{B_{r_j}(x_j)}e_{\epsilon_i}(u_{\epsilon_i})\geqslant\varepsilon_0$$ by contradiction. Suppose this is not true for some $x_j$, then we have that $u_{\epsilon_i}$ converges smoothly on $B_1(p)\setminus B_{r_j}(x_j+L)$. Hence, $\spt\nu\subset B_{r_j}(x_j+L)$. Moreover, using Definition \ref{bubble regions} (b4), the assumption $\delta^{-1/2}\frac{\p}{\p r}\bar{\Theta}_{u_{\epsilon_i}}(p,\delta^{-1/2})\leqslant c\delta$ and Theorem \ref{2d estimate}, we deduce that 
    $$\vert\Pi_{L^{\perp}}(y-x_j)\vert^2\vert\nabla b\vert^2\leqslant C\delta\left(\frac{\vert\Pi_{L^{\perp}}(y-x_j)\vert^{\alpha}}{\delta^{-\alpha/2}}+\frac{r_j^{\alpha}}{\vert\Pi_{L^{\perp}}(y-x_j)\vert^{\alpha}}\right).$$
    Goes back to $u_{\epsilon_i}$, we get that 
    $$\vert\Pi_{L^{\perp}}(y-x_j)\vert^2\vert\nabla u_{\epsilon_i}\vert^2\leqslant C\delta\left(\frac{\vert\Pi_{L^{\perp}}(y-x_j)\vert^{\alpha}}{\delta^{-\alpha/2}}+\frac{\delta^{2\alpha}}{\vert\Pi_{L^{\perp}}(y-x_j)\vert^{\alpha}}\right),$$
    for all $x\in B_{\delta^{-1}/2}(x_j)\setminus B_{\delta^2}(x_j)$ by our choice $r_j\leqslant\delta^2$. Hence,
    \begin{equation}
        \begin{aligned}
            r\frac{\p}{\p r}\bar{\Theta}_{u_{\epsilon_i}}(x_j,r)&=2\int-\dot{\rho}_r(y-x_j)\vert\nabla_{y-x_j}\nabla u_{\epsilon_i}\vert^2+\rho_1(y-x_j)\frac{F(u_{\epsilon_i})}{\epsilon_i^2}\\&\leqslant2\int-\dot\rho_r(y-x_j)\vert\nabla_{\Pi_{L^{\perp}}(y-x_j)}u_{\epsilon_i}\vert^2+C\bar{\Theta}_{u_{\epsilon_i}}(p,2)\\&=2\int_{B_{\delta^2}(x_j+L)}+2\int_{B_{\delta^2}(x_j+L)^c}-\dot{\rho}_r(y-x_j)\vert\nabla_{\Pi_{L^{\perp}}(y-x_j)}u_{\epsilon_i}\vert^2+C\bar{\Theta}_{u_{\epsilon_i}}(p,2)\\&\leqslant C\delta^2\Lambda+C\delta^{1+\alpha/2}+C\delta'\leqslant c(m,K_N,\Lambda)\delta,
        \end{aligned}\nonumber
    \end{equation}
    if we choose $\delta$ and $\delta'$ small. Here $c(m,K_N,\Lambda)$ is the same small constant as in our assumption. This contradicts our assumption since $x_j\in B_{C(m)\sqrt{\delta}}(p)$ and hence proves our result.
\end{proof}
Next, we discuss the structure of bubble regions. By definition, we know that on bubble regions our map looks like a 2-dimensional harmonic in the smooth sense. We can also capture the energy of solution to (\ref{*}) using harmonic spheres on bubble regions. 
\begin{thm}[{cf. \cite[Theorem 3.24]{naber2024energyidentitystationaryharmonic}}]\label{bubble region properties}
    Let $u_{\epsilon}:B_{10R\delta^{-1}}(p)\rightarrow\R^J$ be a solution to (\ref{*}) with $(2R)^{2-n}\int_{B_{2R}(p)}e_{\epsilon}(u_{\epsilon})\leqslant\Lambda$ and $\mathcal{B}=B_1(p)\setminus\cup B_{r_j}(x_j+L)$ be a $\delta$-bubble region. Suppose $\bar{\Theta}(p,2;L)\leqslant\delta'$. For each $\sigma>0$, if $\delta<\delta(m,\Lambda,K_N,\sigma)$ and $\delta'<\delta'(m,K_N,\Lambda,\sigma)$, the following hold \begin{enumerate}
        \item For all $y\in (L+p)\cap B_1(p)$, if we set $\mathcal{B}_y=\mathcal{B}\cap(y+L^{\perp})$ and $y_j=(x_j+L)\cap(y+L^{\perp})$, we have
        $$\left\vert\Theta(p,1)-\omega_{m-2}\int_{\mathcal{B}_y}e_{\epsilon}(u_{\epsilon})-\sum_j\Theta(y_j,r_j)\right\vert\leqslant\sigma,$$
        where we denote the usual energy density by $\Theta(x,r)=r^{2-m}\int_{B_r(x)}e_{\epsilon}(u_{\epsilon})$;
        \item For all $y\in (L+p)\cap B_1(p),\mathcal{B}_y=\mathcal{B}\cap(y+L^{\perp})$, there exists at most $K\leqslant K(m,K_N,\Lambda)$ points $\lbrace z_k\rbrace_{k=1}^K\subset\mathcal{B}_y$ and $s_j>0$ such that if $S\geqslant S(m,K_N,\Lambda,\sigma)$, then $$\int_{\mathcal{B}_y\setminus\cup B_{Ss_k}(z_k)}e_{\epsilon}(u_{\epsilon})\leqslant\sigma.$$
    \end{enumerate}
\end{thm}
\begin{proof}
    For the 2 dimensional harmonic maps, the results are simply the classical energy identity. For higher dimensions, if the results do not hold then by letting $\delta'$ goes to 0 then we immediately obtain the contradiction and hence proves the results. To be specific, when $\delta'\rightarrow0$, we immediately have 
    $$\Theta_{u_{\epsilon_i}}(p,1)-\omega_{m-2}\int_{\mathcal{B}_y}e_{\epsilon_i}(u_{\epsilon_i})-\sum_j\Theta_{u_{\epsilon_i}}(x_j,r_j)\rightarrow0.$$
    This proves (1). For (2), as in the proof \ref{bubble region existence}, $u_{\epsilon_i}$ converges smoothly to $u$ on $\mathcal{B}_y$. Since $u$ is $L$ invariant, we can regard it as a 2-dimensional harmonic map. Hence, we can select bubble points $\lbrace z_k\rbrace\subset\mathcal{B}_y,s_k>0$, such that for $S>S(m,K_N,\Lambda,\sigma)$, $$\int_{\mathcal{B}_y\setminus B_{Ss_k(z_k)}}\vert\nabla u\vert^2<\sigma/2.$$
    Since $u_{\epsilon_i}$ converges smoothly to $u$ on $\mathcal{B}_y$, same conclusion holds for $u_{\epsilon_i}$ for $i$ large. This proves (2).
\end{proof}
\subsection{Annular regions}

In the bubble region parts, we have seen that the energy of Ginzburg-Landau equation is captured by 2-dimensional bubbles. The next part of energy identity is to show that there is no energy in the regions between bubbles, called annular regions. In the bubble regions, the solution to (\ref{*}) looks like a 2-dimensional harmonic map. The blow-up scheme ensures that the solution is really flat between bubbles.  The precise properties we can require those annular regions to have is the following.
\begin{defn}[[{cf. \cite[Definition 3.1]{naber2024energyidentitystationaryharmonic}}]]\label{annular regions}
    Let $\delta>0,\mathcal{T}\subset B_2(p)$ be an $(m-2)$-submanifold and $\mathfrak{r}:\mathcal{T}\rightarrow[0,\infty)$ satisfies $\mathfrak{r}_x+\vert\nabla\mathfrak{r}_x\vert+\mathfrak{r}_x\vert\nabla^2\mathfrak{r}_x\vert<\delta$. We say that $$\mathcal{A}=B_2(p)\setminus\overline{B_{\mathfrak{r}_x}(\mathcal{T})}=B_2(p)\setminus\overline{\bigcup_{x\in\mathcal{T}}B_{\mathfrak{r}_x}(x)}$$is a $\delta$-annular region if there exists an $(m-2)$ plane $L_{\mathcal{A}}$ such that
    \begin{enumerate}
        \item[(a1)] There exists a $\mathfrak{t}_{\mathcal{T}}:L_{\mathcal{A}}\cap B_2\rightarrow L_{\mathcal{A}}^{\perp}$ with $\mathcal{T}=p+$Graph$\mathfrak{t}_{\mathcal{T}}$ and $\vert\mathfrak{t}_{\mathcal{T}}\vert+\vert\nabla\mathfrak{t}_{\mathcal{T}}\vert+\mathfrak{r}_x\vert\nabla^2\mathfrak{t}_{\mathcal{T}}\vert<\delta$;
        \item[(a2)] For each $x\in\mathcal{T}$ and $\mathfrak{r}_x\leqslant r\leqslant2$ we have $r\frac{\p}{\p r}\bar{\Theta}(x,r)\leqslant\delta$;
        \item[(a3)] For each $x\in\mathcal{T}$ we have $\bar{\Theta}(x,\mathfrak{r}_x)\geqslant\varepsilon_0(m,K_N)/2$.
    \end{enumerate}
\end{defn}

If we have a sequence $\lbrace u_{\epsilon_i}\rbrace$ which converges to a constant map together with defect measure $\nu=\theta\Hau^{m-2}\llcorner L$, where $L$ is an $(m-2)$-subspace of $\R^m$. Then it is not hard to verify that $B_1(p)\setminus B_{r_0}(L)$ is an annular region for $u_{\epsilon_i}$ and $i$ large. Those annular regions are not what we want as we hope certain maximal condition is held by our annular regions. That is, we hope the annular regions nearly capture all the spaces between bubbles.

\begin{thm}[{cf. \cite[Theorem 3.13]{NVyangmills}}]\label{annular region existence}
    Let $u_{\epsilon}:B_{2R}(p)\rightarrow\R^J$ be a solution to (\ref{*}) with $R^{2-m}\int_{B_{2R}(p)}e_{\epsilon}(u_{\epsilon})\leqslant\Lambda$. For each $0<\delta'<\delta$, there exists $\delta''=\delta''(m,\Lambda,N,R,\delta,\delta')$ such that if $\epsilon\leqslant\epsilon(m,\Lambda,K_N,R,\delta)$ and $u_{\epsilon}$ satisfies
    \begin{enumerate}
        \item $r\frac{\p}{\p r}\bar{\Theta}(x,r)\leqslant\delta^2$, for all $x\in L\cap B_2+p$ and $\delta^2\leqslant r\leqslant4$;
        \item $\bar{\Theta}(p,2;L)\leqslant\delta''$;
        \item $\bar{\Theta}(p,2)\geqslant\varepsilon_0/2$.
    \end{enumerate}
    Then there exists an $(m-2)$-submanifold $\mathcal{T}\subset B_1(p)$ such that $\mathcal{A}=B_1(p)\setminus\overline{B_{\mathfrak{r}_x}(\mathcal{T})}$ is a $\delta$-annular region with
    \begin{enumerate}[(1)]
        \item For each $x\in\mathcal{T}$ and $\mathfrak{r}_x\leqslant r\leqslant2$, we have $\bar{\Theta}(x,r;L_{\mathcal{A}})=r^2\int\rho_r(y-x)(\vert\Pi_{L_{\mathcal{A}}}\nabla u_{\epsilon}\vert^2+F(u_{\epsilon})/\epsilon^2)\leqslant\delta'$;
        \item There is a collection of balls $\lbrace B_{\delta^{-1/2}\mathfrak{r}_{y_j}}(y_j)\rbrace$ with $y_j\in\mathcal{T}$, $\lbrace B_{\delta^{-1/2}\mathfrak{r}_{y_j}/3}(y_j)\rbrace$ is pairwise disjoint such that the following maximal condition holds: if we denote $\lambda=c(m,K_N,\Lambda)\delta$, then
        $$\textup{ for all }y\in B_{C(m,K_N)\sqrt{\delta}\mathfrak{r}_{y_j}}(y_j+L)\cap(y_j+L^{\perp}),\lambda\delta^{-1/2}\mathfrak{r}_{y_j}\frac{\p}{\p r}\bar{\Theta}(y,\lambda\delta^{-1/2}\mathfrak{r}_{y_j})\geqslant c(m,\Lambda,K_N)\delta,
        $$
        and
        $$\Hau^{m-2}(\mathcal{T}\setminus(\cup_jB_{\delta^{-1/2}\mathfrak{r}_{y_j}}(y_j))\leqslant\delta.$$
    \end{enumerate}
\end{thm}
\begin{proof}
    Define
    $$\mathfrak{s}_x=\inf\left\lbrace s\in(0,1]:\forall s\leqslant r\leqslant2,r\frac{\p}{\p r}\bar{\Theta}(x,r)\leqslant\delta
    \right\rbrace.$$
    For each $x\in p+L_{\mathcal{A}}\cap B_2$, we investigate the value $\hat{\mathfrak{s}}_x=\inf_{y\in x+L_{\mathcal{A}}^{\perp}\cap B_2}\mathfrak{s}_y$ and a point $y_x\in x+L_{\mathcal{A}}^{\perp}\cap B_2$ such that $\mathfrak{s}_{y_x}=\hat{\mathfrak{s}}_x$. Note that by assumption we have $\hat{\mathfrak{s}}_x\leqslant\mathfrak{s}_x\leqslant\delta^2$. 
    For each $y$, consider also
    $$\mathfrak{s}_y'=\inf\left\lbrace s\in(0,1]:\forall s\leqslant r\leqslant2,\bar{\Theta}(y,r;L)\leqslant\delta'\right\rbrace,$$
    where $\delta'<\delta$ is a constant to be determined. Set
    $$\mathfrak{q}_y=\max\lbrace\mathfrak{s}_y,\mathfrak{s}_y'\rbrace.$$
    \begin{claim}\label{c1 in annular region existence}
        For $\epsilon$ and $\delta''$ small enough, we must have $r=\vert y_x-x\vert\leqslant\delta$.
    \end{claim}
    \begin{proof}[Proof of Claim \ref{c1 in annular region existence}]
        Since the assumption of Theorem \ref{annular region existence} implies that $u_{\epsilon}$ is $(m-2,\delta)$ symmetric at $x$. If the claim were not true, then $y_x$ is an extra symmetric point which implies that $u_{\epsilon}$ is $(m-1,C\delta)$ symmetric at $x$. Lemma \ref{m-1 flatness} implies $\bar{\Theta}(x,r)\leqslant\gamma$ for some small $\gamma$. Since 
    $$\bar{\Theta}(x,1)-\bar{\Theta}(x,r)=\left\vert\int_r^1\frac{\p}{\p t}\bar{\Theta}(x,t)\dif t\right\vert\leqslant-\delta\log \delta.$$
    If $\delta$ is small enough, then $\bar{\Theta}(x,1)<\varepsilon_0$, contradicts to our assumption.
    \end{proof}
    \begin{claim}\label{c2 in annular region existence}
        If $\epsilon,\delta'$ are small enough, we must have $\bar{\Theta}(y_x,\mathfrak{q}_{y_{x}})\geqslant\varepsilon_0/2$ for all $x\in L\cap B_1+p$
    \end{claim}
    \begin{proof}[Proof of Claim \ref{c2 in annular region existence}]
        Indeed, if this were not true, we must have a sequence of map $u_{\epsilon_i}$ and $y_i\in B_1(p)$ such that 
    $$r_i=\sup\lbrace r\in(0,2]:\bar{\Theta}_{u_{\epsilon_i}}(y_i,r)\leqslant\varepsilon_0/2\rbrace\geqslant\mathfrak{q}_{y_i}.$$Then, let us consider the rescaled maps $v_{\epsilon_i/r_i}(x)=u_{\epsilon_i}(y_i+r_ix)$. As $\epsilon$ and $\delta'\rightarrow0$, we may assume $v_{\epsilon_i/r_i}$ converges strongly on $B_1$ to some $L$ invariant map $u$, which is
    \begin{enumerate}
        \item a harmonic map if $\epsilon_i/r_i\rightarrow0$;
        \item a solution to (\ref{*}) or a vector valued harmonic function on $\R^m$ with finite energy if $\epsilon_i/r_i\rightarrow\epsilon\in(0,\infty]$.
    \end{enumerate}
    In all the cases above, the energy of $u$ is finite and $u$ is non-trivial. In the case (2), this is impossible, since a finite energy $u:\R^2\rightarrow\R^J$ in both cases is trivial, see for example \cite{Linwangharmonicsphere99}. In the case (1), we have
    $$\bar{\Theta}_u(0,1)=\epsilon_0/2,r\frac{\p}{\p r}\bar{\Theta}(0,r)\leqslant\delta\textup{ for all }r\in[1,\delta^{-1}].$$
    By Theorem \ref{2d estimate} again, this implies
    $$\vert\Pi_{L^{\perp}}(y)\vert^2\vert\nabla u\vert^2\leqslant C\delta^{1+\alpha}\vert\Pi_{L^{\perp}}(y)\vert^{\alpha}.$$
    Integrating this yields $\bar{\Theta}_u(x,1)\leqslant C\delta^{1+\alpha}$, contradicts to our assumption.
    \end{proof}
    Now, let us take a maximal subset $\lbrace B_{\delta^{-1/2}\mathfrak{q}_j}(y_j)\rbrace$ of $\lbrace B_{\delta^{-1/2}\mathfrak{q}_{y_x}}(y_x)\rbrace$ such that $\lbrace B_{\delta^{-1/2}\mathfrak{q}_j/3}(y_j)\rbrace$ is pairwise disjoint. 
    Note that we can write $$\bigcup_jB_{\mathfrak{q}_j}(y_j)=\bigcup_jB_{\mathfrak{s}_j}(y_j)\cup\bigcup_{j}B_{\mathfrak{s}_j'}(y_j').$$In the subcollection where $\mathfrak{q}_j=\mathfrak{s}_j'$, we have $\bar{\Theta}(y_j,\mathfrak{s}_j';L)\geqslant\delta'$ for each $j$, hence by disjointness, if we take $\delta''$ small enough, we have
    $$\sum_j\mathfrak{s}_j'^{m-2}\leqslant\frac{1}{\delta'}\sum_j\int_{B_{\mathfrak{s}_j'}(y_j')}(\vert\Pi_L\nabla u_{\epsilon}\vert^2+F(u_{\epsilon})/\epsilon^2)\leqslant\frac{C(m)}{\delta'}\bar{\Theta}(p,2;L)\leqslant\delta^{m/2}.$$
    Next, we prove that $\cup_jB_{\mathfrak{q}_j}(y_j)$ is a good approximation of an $(m-2)$-submanifold.
    \begin{claim}\label{c3 in annular region existence}
        For each $i,j$, we have $\vert\Pi_{L^{\perp}}(y_j-y_i)\vert\leqslant C(m,K_N,\Lambda)\delta^{1/2}\vert\Pi_L(y_j-y_i)\vert$.
    \end{claim}
    \begin{proof}[Proof of Claim \ref{c3 in annular region existence}]
        The proof is again make use of the quantitative cone splitting \ref{quantitative cone splitting}. If the claim is not true, WLOG $\mathfrak{q}_i\leqslant\mathfrak{q}_j$, set $r=\vert y_i-y_j\vert\geqslant\mathfrak{q}_i/3+\mathfrak{q}_j/3\geqslant2\mathfrak{q}_i/3,l=\frac{\Pi_{L^{\perp}}(y_j-y_i)}{\vert\Pi_{L^{\perp}}(y_j-y_i)\vert}$, and we set
        $l'=\alpha (y_j-y_i)+\beta l$, where $l=\frac{\Pi_L(y_j-y_i)}{\vert\Pi_L(y_j-y_i)\vert}$ is a unit vector in $L$ and $\alpha=\vert\Pi_{L^{\perp}}(y_j-y_i)\vert^{-1}\leqslant c(m,K_N)\delta^{-1/2}r,\beta=\frac{\vert\Pi_L(y_j-y_i)\vert}{\vert\Pi_{L^{\perp}}(y_j-y_i)\vert}\leqslant c(m,K_N)\delta^{-1/2}$. Consider an $(m-1)$-subspace $L'=$span$\lbrace L,l'\rbrace$
        \begin{equation}
            \begin{aligned}
                &r^2\int\rho_{2r}(y-y_i)\vert\nabla_{l'}u_{\epsilon}\vert^2\\\leqslant& r^2\int\rho_r(y-x_i)\left(\alpha^2\vert\nabla_{y_i-y_j}u_{\epsilon}\vert^2+\beta^2\vert\nabla_lu_{\epsilon}\vert^2\right)\\\leqslant&c(m,K_N,\Lambda)\delta^{-1}\int\rho_{2r}(y-y_i)(\vert\nabla_{y-y_i}u_{\epsilon}\vert^2+\vert\nabla_{y-y_j}u_{\epsilon}\vert^2)+\delta^{-2/3}r^2\int\rho_r(y-y_i)\vert\nabla_lu_{\epsilon}\vert^2\\\leqslant& c(m,K_N,\Lambda)\delta^{-1}\left(r\frac{\p}{\p r}\bar{\Theta}(y_i,3r)+r\frac{\p}{\p r}\bar{\Theta}(y_j,3r)+\bar{\Theta}(y_i,2r;L)\right)\leqslant c(m,K_N,\Lambda).
            \end{aligned}\nonumber
        \end{equation}
    Hence $u_{\epsilon}$ is $(m-1,c(m,K_n,\Lambda)$ symmetric with respect to $L'$ on $B_{2r}(x_i)$. By Lemma \ref{m-1 flatness}, this implies $\bar{\Theta}(y_i,r)\leqslant\varepsilon_0/3$, contradicts to Claim \ref{c2 in annular region existence}.
    \end{proof}
    This shows that $\lbrace B_{\mathfrak{q}_i}(y_i)\rbrace$ is a good approximation with to an $(m-2)$-submanifold with curvature less than $\sqrt{\delta}$. Now, we can define our annular region. Let us consider the covering $\lbrace B_{\delta^{-1/2}\mathfrak{q}_j}(y_j)\rbrace$ and select a further subcovering such that $\lbrace B_{\delta^{-1/2}\mathfrak{q}_j/3}(y_j)\rbrace$ is pairwise disjoint. Let $x_j=\Pi_L(y_j-p)+p$, then we have a corresponding covering $\lbrace B_{\delta^{-1/2}\mathfrak{q}_j}(x_j)\rbrace$ on $L+p$. By claim \ref{c3 in annular region existence}, $\lbrace B_{\delta^{-1/2}}\mathfrak{q}_j/6(x_j)\rbrace$ is pairwise disjoint. Let $\lbrace\varphi_j\rbrace$ be a partition of unity sub-coordinate to this covering with $\vert\nabla^k\varphi_j\vert\leqslant C(m,k)\delta^{k/2}\mathfrak{q}_j^{-k},k=1,2$. Define the annular region $L\cap B_1+p$ by
    $$\mathfrak{t}_{\mathcal{T}}(x)=\sum_j\varphi_j(x)(y_j-x_j),\mathfrak{r}_{\mathfrak{t}_{\mathcal{T}}(x)}=\sum_j\varphi_j(x)\mathfrak{q}_j,\mathcal{T}=\textup{Graph}\mathfrak{t}_{\mathcal{T}}.$$
    By claim \ref{c1 in annular region existence},\ref{c2 in annular region existence},\ref{c3 in annular region existence} and definition of $\mathfrak{q}_j$, it is not hard to see $\mathcal{A}=B_1(p)\setminus\overline{B_{\mathfrak{r}_x}(\mathcal{T})}$ is an $C(m,K_N,\Lambda)\delta$-annular region provided we take $\delta'\leqslant\delta'(m,\Lambda,R,K_N,\delta)$. Moreover, Theorem \ref{annular region existence} (1) is held for $\mathcal{A}$.\\
    Finally, let us check that Theorem \ref{annular region existence} (2) is valid for $\mathcal{A}$.
    \begin{claim}
        If we choose $\delta'$ and $\epsilon$ small enough, then for all the $j$ such that $\mathfrak{q}_j=\mathfrak{s}_j$, if we set $\lambda=c(m,\Lambda,K_N)\sqrt{\delta}$, then $$\lambda\mathfrak{s}_j\frac{\p }{\p r}\bar{\Theta}(x,\lambda\mathfrak{s}_j)\geqslant c(m,\Lambda,K_N)\delta,\textup{ for all }y\in B_{C(m,K_N)\sqrt{\delta}\mathfrak{s}_j}(y_j+L)\cap(y_j+L^{\perp}).$$
    \end{claim}
    \begin{proof}[Proof of Claim 3.9]
        The proof is again by a contradiction argument. Assume the claim fails for some $\bar{y}\in B_{C(m,K_N)\sqrt{\delta}\mathfrak{s}_j}(y_j+L)\cap(y_j+L^{\perp})$ By sending $\delta'$ and $\epsilon$ to $0$ and arguing like Claim \ref{c2 in annular region existence}, we can consider the rescaled maps $v_i=u(\bar{y}^{(i)}+\mathfrak{s}_j^{(i)}x)$ and assume they converge to either a constant map in the strong sense or to an $L$-invariant harmonic map $u$. In the latter case, by the contradiction assumption 
        $$\lambda\frac{\p}{\p r}\bar{\Theta}(0,\lambda)\leqslant c(m,\Lambda,K_N)\delta,$$
        we have that the convergence is strong on $(B_{100\delta^{-2}}(L)\setminus B_{10^{-2}\lambda}(L))\cap B_2$. By the assumption $\frac{\p}{\p r}\bar{\Theta}_{u_{\epsilon}}(x_j,4)\leqslant\delta^2$, we see that $\frac{\p}{\p r}\bar{\Theta}_{u_{\epsilon}}(y_j,3)\leqslant c(m,\Lambda,K_N)\delta$ as well. We also have $r\frac{\p}{\p r}\bar{\Theta}_{u_{\epsilon}}(\bar{y},r)\leqslant C(m,\Lambda,K_N)\delta$ for all $r\in[\mathfrak{s_j},2]$ Hence, by Theorem \ref{2d estimate}, the limiting map $u$ must satisfies 
        $$\vert\Pi_{L^{\perp}}(y)\vert^2\vert\nabla u\vert^2\leqslant c\delta\left(\frac{\vert\Pi_{L^{\perp}}(y)\vert^{\alpha}}{\lambda^{\alpha}}+\frac{1}{\vert\Pi_{L^{\perp}}(y)\vert{\alpha}}\right),10^{-2}\lambda\leqslant\vert\Pi_{L^{\perp}}(y)\vert\leqslant3.$$
        Trace back to $u_{\epsilon}$, if we integrate this inequality, we see that if $c(m,\Lambda,K_N)$ is taken small enough, we have
        $$r\frac{\p}{\p r}\bar{\Theta}(\bar{y},r)\leqslant\delta/2\textup{ for all }r\in[\mathfrak{s}_j,2].$$
        However, this contradicts the smallest assumption of $\mathfrak{s}_j$ and $y_j$. Hence the proof of Claim is complete.
    \end{proof}
    The covering $B_{\delta^{-1/2}\mathfrak{s}_j}(y_j)$ is the desired covering for Theorem \ref{annular region existence} (2). Hence the proof of the Theorem is complete.
\end{proof}
As we have discussed at the beginning of this section, after capturing the energy in the bubble region using harmonic spheres, we wish to show that the energy in annular regions vanishes. This is the main goal of the rest of this paper. The proof of Theorem \ref{vanishing energy in annular regions} will be the result of Theorem \ref{energy decomposition}, \ref{angular energy estimate} and \ref{radial and tangential energy estimate}
\begin{thm}[{cf. \cite[Theorem 3.16]{NVyangmills}}]\label{vanishing energy in annular regions}
    Let $u_{\epsilon}:B_{2R}(p)\rightarrow\R^L$ be a solution to (\ref{*}) with $R^{2-m}\int_{B_{2R}(p)}e_{\epsilon}(u_{\epsilon})\leqslant\Lambda$. Suppose $\mathcal{A}=B_1(p)\setminus\overline{B_{\mathfrak{r}_x}(\mathcal{T})}$ is a $\delta$-annular region. For each $\sigma>0$, if $R>R(m,\Lambda,K_N,\sigma)$, and $\delta<\delta(m,\Lambda,K_N,\sigma),\epsilon<\epsilon(m,\Lambda,K_N,\sigma)$, then $$\int_{\mathcal{A}_{\sigma}}e_{\epsilon}(u_{\epsilon})\leqslant\sigma,\textup{ where }\mathcal{A}_{\sigma}=B_2(p)\setminus\overline{B_{\sigma\mathfrak{r}_x}(\mathcal{T})}.$$
\end{thm}
\begin{cor}[{cf. \cite[Theorem 3.18]{NVyangmills}}]
    Suppose $u_{\epsilon}$ and $\mathcal{A}$ satisfy the condition of Theorem \ref{vanishing energy in annular regions}. Then
    $$\int_{\mathcal{T}\cap B_1(p)}(\Theta(x,1)-\Theta(x,\mathfrak{r}_x))\dif x\leqslant\sigma.$$
    Here $\Theta(x,r)=r^{2-m}\int_{B_r(x)}e_{\epsilon}(u_{\epsilon})$ is the usual energy density.
\end{cor}
\begin{proof}
    By the Definition \ref{annular regions} (a1), we have
    \begin{equation}
        \begin{aligned}
            \int_{\mathcal{T}\cap B_1(p)}\Theta(x,1)\dif x&=\int_{\mathcal{T}\cap B_1(p)}\dif x\int_{B_1(x)}e_{\epsilon}(u_{\epsilon})(y)\dif y\\&=\int_{B_1(\mathcal{T}\cap B_1(p))}e_{\epsilon}(y)\dif y\int_{B_1(y)\cap\mathcal{T}}\dif x\\&\leqslant (1+C(m)\delta)\omega_{m-2}\int_{B_1(\mathcal{T}\cap B_1(p))}e_{\epsilon}(y)\dif y.
        \end{aligned}\label{Cor 3.11(1)}
    \end{equation}
    On the other hand, take a covering $\lbrace B_{r_j}(x_j)\rbrace$ of $\mathcal{T}\cap B_1(p)$ such that $x_j\in\mathcal{T}\cap B_1(p),r_j=\mathfrak{r}_{x_j}$ and $\lbrace B_{r_j/3}(x_j)\rbrace$ pairwise disjoint. By Definition \ref{annular regions} (a1) and the gradient bounds on $\mathfrak{r}_x$, we have that $(1-\delta)r_j\leqslant\mathfrak{r}_{x}\leqslant(1+\delta)r_j$ if $x\in B_{r_j}(x_j)\cap\mathcal{T}$. Hence
    \begin{equation}
        \begin{aligned}
            &\int_{\mathcal{T}\cap B_1(p)}\Theta(x,\mathfrak{r}_x)\dif x\\\geqslant&(1-C(m)\delta)\sum_j\int_{\mathcal{T}\cap B_{r_j/3}(x_j)}\Theta(x,r_j)\\=&(1-C\delta)\sum_j\int_{B_{r_j}(\mathcal{T}\cap B_{r_j/3}(x_j))}e_{\epsilon}(u_{\epsilon})(y)\dif y\int_{B_{r_j}(y)\cap\mathcal{T}}\dif x\\\geqslant&(1-C\delta)\sum_{j}\int_{B_{r_j}(\mathcal{T}\cap B_{r_j/3}(x_j))\cap\lbrace\dist(y,\mathcal{T})\rbrace\leqslant\sigma r_j\rbrace}e_{\epsilon}(u_{\epsilon})(y)\dif y\int_{B_{r_j}(y)\cap\mathcal{T}}\dif x\\\geqslant&\omega_{m-2}(1-C(\delta+\sigma))\int_{B_{\sigma\mathfrak{r}_x}(\mathcal{T})}e_{\epsilon}(u_{\epsilon})\dif y
        \end{aligned}\label{Cor 3.11(2)}
    \end{equation}
    Combine (\ref{Cor 3.11(1)}),(\ref{Cor 3.11(2)}) we get 
    $$\int_{\mathcal{T}\cap B_1(p)}(\Theta(x,1)-\Theta(x,\mathfrak{r}_x))\dif x\leqslant\omega_{m-2}(1-C(\delta+\sigma))\int_{\mathcal{A}_{\sigma}}e_{\epsilon}(u_{\epsilon})+C(\delta+\sigma)\Lambda\leqslant C\sigma.$$
    This proves our corollary.
\end{proof}
\subsection{Quantitative bubble decomposition and Quantitative energy identity}
In this subsection, we use the existence theorem of bubble regions and annular regions to decompose the domain. Moreover, we prove an effective version of the energy identity.
\begin{thm}[Quantitative bubble decomposition, {cf. \cite[Theorem 3.27]{NVyangmills}}]\label{Quantitative bubble decomposition}
    Let $u_{\epsilon}:B_{2R}(p)\rightarrow\R^L$ be a solution to (\ref{*}) with $R^{2-m}\int_{B_{2R}(p)}e_{\epsilon}(u_{\epsilon})\leqslant\Lambda$. For each $\delta<\delta(m,\Lambda,K_N,R)$, if $\delta''<\delta''(m,\Lambda,K_N,R,\delta)$, $\epsilon<\epsilon(m,\Lambda,K_N,\delta)$ and $u_{\epsilon}$ is $(m-2,\delta'')$ symmetric on $B_2(p)$ with respect to $L^{m-2}$ and $\bar{\Theta}(p,1)\geqslant\varepsilon_0$. Then we can write $$B_1(p)\subset\mathcal{A}_0\cup\bigcup\mathcal{A}_a\cup\bigcup\mathcal{B}_b\cup\bigcup B_{r_c}(x_c),$$with:
    \begin{enumerate}[(1)]
        \item $\mathcal{A}_0\subset B_2(p)$ and $\mathcal{A}_a\subset B_{2r_a}(x_a)$ are $\delta$-annular regions with regards to $L$;
        \item $\mathcal{B}_b\subset B_{2r_b}(x_b)$ are $\delta$-bubble regions with regard to $L$;
        \item $\sum r_a^{m-2}+r_b^{m-2}\leqslant C(m,\Lambda)$ and $\sum r_c^{m-2}\leqslant\delta$;
        \item For each $\mathcal{A}_a$, if we let $\mathcal{N}_a=\lbrace x\in B_{r_a}(x_a):B_{\mathfrak{r}_x^a}(x)$ intersects a $B_{2r_{a'}}(x_{a'})$ or $B_{r_c}(x_c)\rbrace$, then $\Hau^{m-2}(\mathcal{N}_a)\leqslant\delta$.
    \end{enumerate}
\end{thm}
Next, we state the quantitative energy identity. Roughly speaking, this theorem tells us that if the solution to (\ref{*}) is (m-2)-symmetric with respect to a plane $L$, then on generic slices in the direction $L^{\perp}$, the map will be approximated by harmonic spheres in a strong sense. 
\begin{thm}[Quantitative energy identity, {cf. \cite[Theorem 1.12]{NVyangmills}}]\label{Quantitative energy identity}
Let $u_{\epsilon}:B_{2R}(p)\rightarrow\R^L$ be a solution to (\ref{*}) with $R^{2-m}\int_{B_{2R}(p)}e_{\epsilon}(u_{\epsilon})\leqslant\Lambda$. For each $\sigma<\sigma(m,\Lambda,K_N,R)$, if  $\delta''<\delta''(m,\Lambda,K_N,\sigma),R\geqslant R(m,\Lambda,K_N,\sigma),\epsilon<\epsilon(m,\Lambda,K_N,\sigma)$ and $u_{\epsilon}$ is $(m-2,\delta'')$ symmetric on $B_2(p)$ with regards to $L^{m-2}$ and $\bar{\Theta}(p,1)\geqslant\varepsilon_0$.\par
Then, under the decomposition of Theorem \ref{Quantitative bubble decomposition}, we can find $\mathcal{G}_{\sigma}\subset L\cap B_1(p)$ with $\Hau^{m-2}(L\cap B_1(p)\setminus\mathcal{G}_{\sigma})\leqslant\sigma$ and For each $y\in\mathcal{G}_{\sigma}$,
\begin{enumerate}[(1)]
    \item We have
    $$\int_{(L^{\perp}+y)\cap B_1(p)}\vert\Pi_{L^{\perp}}\nabla u_{\epsilon}\vert^2\leqslant C(n)\Lambda\textup{ and }\int_{(L^{\perp}+y)\cap B_1(p)}(\vert\Pi_L\nabla u_{\epsilon}\vert^2+F(u_{\epsilon})/\epsilon^2)\leqslant\sigma;$$
    \item $u_{\epsilon}:(L^{\perp}+y)\cap B_1(p)\rightarrow\R^L$ is a $\sigma$-almost harmonic map with at most $K(m,\Lambda,K_N)$ bubbles.The $\delta$-almost harmonic map is defined in the following.
    \item $$\left\vert\int_{B_1(y)}e_{\epsilon}(u_{\epsilon})-\frac{1}{2}\omega_{m-2}\int_{y+L^{\perp}}\vert\Pi_{L^{\perp}}\nabla u_{\epsilon}\vert
    ^2\right\vert\leqslant\sigma.$$
\end{enumerate}
\end{thm}
\begin{defn}[{cf. \cite[Definition 1.11]{NVyangmills}}]
    Consider a smooth map $v:B_1^2\rightarrow\R^L$. We say it is \emph{$\delta$-almost harmonic} if there exists $\lbrace B_{r_j}(x_j)\rbrace_{j=1}^K\subset B_1^2$ and non-trivial harmonic maps $b_j:\R^2\rightarrow N$ with
    \begin{enumerate}[(1)]
        \item $\vert\nabla b_j\vert^2\leqslant1$ and $\int_{\R^2\setminus B_{\delta^{-1}}}\vert\nabla b\vert^2\leqslant\delta$;
        \item If $r_j\leqslant r_k$, then either $B_{r_j}(x_j)\cap B_{r_k}(x_k)=\emptyset$ or $B_{r_j}(x_j)\subset B_{r_k}(x_k)$;
        \item The rescaled maps $\tilde{b}_j=b_j(\delta r_j^{-1}(x-x_j))$ satisfies $r_j\vert\nabla v-\nabla \tilde{b}_j\vert\leqslant\delta$ on each $B_{r_k}(x_k)\setminus\cup_{r_j\leqslant r_k}B_{r_j}(x_j)$.
    \end{enumerate}
\end{defn}
\begin{proof}[Proof of Theorem \ref{Quantitative bubble decomposition} and Theorem \ref{Quantitative energy identity} assuming \ref{vanishing energy in annular regions}]
Take $\delta'>0$ to be chosen very small. By Lemma \ref{m-2 symmetry alternative} and Theorem \ref{annular region existence}, if $\delta''\leqslant\delta''(m,\Lambda,K_N,\delta)$, we can find a maximal annular region $\mathcal{A}_0=B_1(p)\setminus B_{\mathfrak{r}_x^0}(\mathcal{T}_0)$. Moreover, $\mathcal{T}_0$ can be maximally covered by balls $\lbrace B_{r_{c,1}^i}(x_{c,1}^i)\rbrace$ and $\lbrace B_{r_{b,1}^i}(x_{b,1}^i)\rbrace$ such that $r_{b,1}^i=\delta^{-1/2}\mathfrak{r}_{x_{b,1}^i}^0,\sum r_{c,1}^{m-2}\leqslant\delta,\bar{\Theta}(x_{b,1},2r_{b,1};L)\leqslant\delta',$ and $\lambda r_{b,1}\frac{\p}{\p r}\bar{\Theta}(x,\lambda r_{b,1})\geqslant10^{-2}\delta$ for all $x\in B_{C(m,K_N,\Lambda)\delta)r_{b,1}}(x_{b,1}+L)\cap(x_{b,1}+L^{\perp}).$ Here $\lambda$ is the same constant as in Theorem \ref{bubble region existence} and \ref{annular region existence}. Hence, for each $B_{r_{b,1}}(x_{b,1})$, we can find a bubble region $\mathcal{B}_{b,1}\subset B_{r_{b,1}}(x_{b,1})$.\\
Now, let us prove our theorem inductively. Suppose, in the $k$-th step we have already written 
$$B_1(p)\subset\mathcal{A}_0\cup\bigcup_{i=1}^k\left(\bigcup\mathcal{A}_{a,i}\cup\bigcup\mathcal{B}_{b,i}\cup\bigcup B_{r_{c,i}}(x_{c,i})\right),$$where 
\begin{enumerate}
    \item $\mathcal{A}_{a,i}=B_{r_{a,i}}(x_{a,i})\setminus\overline{B_{\mathfrak{r}_x^{a,i}}(\mathcal{T}_{a,i})}$ are maximal $\delta$-annular regions from Theorem \ref{annular region existence};
    \item $\mathcal{B}_{b,i}=B_{r_{b.i}}(x_{b,i})\setminus\cup B_{r_{a,i+1}}(x_{a,i+1}'+L)$ are $\delta$-bubble regions given by Theorem \ref{bubble region existence}.
    \item $$r_{a,i}^{2-m}\int_{B_{r_{a,i}}(x_{a,i})}e_{\epsilon}(u_{\epsilon})\leqslant\Lambda-\varepsilon_0(i-1)/2,r_{b,i}^{2-m}\int_{B_{r_{b,i}}(x_{b,i})}e_{\epsilon}(u_{\epsilon})\leqslant\Lambda-\varepsilon_0(i-1)/2;$$
    \item $\sum r_{c,i}^{m-2}\leqslant C(m,\Lambda)\delta$ and $\sum r_{a,i}^{m-2}+r_{b,i}^{m-2}\leqslant C(m,\Lambda)$.
\end{enumerate}
By Theorem \ref{annular region existence}, we can cover $\cup\mathcal{T}_{a,i}$ by $\lbrace B_{r_{b,i+1}}(x_{b,i+1})\rbrace\cup\lbrace B_{r_{c,i+1}}(x_{c,i+1})\rbrace$ with $r_{b,i+1}=\delta^{-1/2}\mathfrak{r}^{a,i}_{b,i+1},\sum r_{c,i+1}^{m-2}\leqslant \delta\sum r_{a,i}^{m-2}\leqslant C(m,\Lambda)\delta,\bar{\Theta}(x_{b,i+1};2r_{b,i+1};L)\leqslant\delta'$ and $$\lambda r_{b,i+1}\frac{\p}{\p r}\bar{\Theta}(x,\lambda r_{b,i+1})\geqslant10^{-2}\delta\textup{ for all }x\in B_{C(m,K_N,\Lambda)\delta r_{b,i+1}}(x_{b,i+1}+L)\cap(x_{b,i+1}+L^{\perp}).$$Hence, by Theorem \ref{bubble region existence}, we can find further $\delta$-bubble region $\mathcal{B}_{b,i+1}=B_{r_{b,i+1}}(x_{b,i+1})\setminus\cup B_{r_{a,i+2}}(x_{a,i+2}'+L)$. If $\delta'$ is small enough, we have $$r_{b,i+1}^{2-m}\int_{B_{r_{b,i}}(x_{b,i+1})}e_{\epsilon}(u_{\epsilon})\leqslant r_{a,i}^{2-m}\int_{B_{r_{a,i}}(x_{a,i})}e_{\epsilon}(u_{\epsilon}).$$The volume estimate $\sum r_{b,i+1}^{2-m}\leqslant C\sum r_{a,i}^{2-m}$ follows directly from Definition \ref{annular regions} (a1).\par
For the $\delta$-bubble regions $\mathcal{B}_{b,i}=B_{r_{b.i}}(x_{b,i})\setminus\cup B_{r_{a,i+1}}(x_{a,i+1}'+L)$, we can maximally cover each $B_{r_{a,i+1}}(x_{a,i+1}'+L)$ by balls $B_{r_{a,i+1}}(x_{a,i+1})$. By definition \ref{bubble regions} and Theorem \ref{annular region existence}, we can find maximal $\delta$-annular regions $\mathcal{A}_{a,i+1}\subset B_{r_{a,i+1}}(x_{a,i+1})$. This covering automatically satisfy the volume estimate $\sum r_{a,i+1}^{2-m}\leqslant C\sum r_{b,i}^{2-m}\leqslant C(m,\Lambda)$. Moreover, those subcoverings satisfies $$r_{a,i+1}^{2-m}\int_{B_{r_{a,i+1}}(x_{a,i+1})}e_{\epsilon}(u_{\epsilon})\leqslant r_{b,i}^{2-m}\int_{B_{r_{b,i}}(x_{b,i})}e_{\epsilon}(u_{\epsilon})-\varepsilon_0.$$
It is clear that $B_{r_{c,i+1}}(x_{c,i+1}),\mathcal{B}_{r_{b,i+1}},\mathcal{A}_{a,i+1}$ constructed above still satisfy the inductive hypothesis and we can proceed our induction. By assumption (3), there will be at most $4\Lambda/\varepsilon_0$ steps in our construction. After $4\Lambda/\varepsilon_0$ steps, all $B_{r_{b,i}}$ will become $\varepsilon_0$-regular regions and no more annular region will appear. It is not hard to see that Theorem \ref{Quantitative bubble decomposition} has been proven.\par
Let us construct $\mathcal{G}_{\sigma}$ in Theorem \ref{Quantitative energy identity}. First consider $$\mathcal{G}_{\sigma}^1=\lbrace y\in p+L\cap B_1:y+L^{\perp}\textup{ intersect }r_c\textup{-balls}\rbrace.$$
Then $\Hau^{m-2}(\mathcal{G}^1_{\sigma})\leqslant\delta$. Also, by Theorem \ref{vanishing energy in annular regions}, if $\delta,\epsilon$ are small and $R$ is large, than 
$$r_{a,i}^{2-m}\int_{(\mathcal{A}_{a,i})_{\sigma}}e_{\epsilon}(u_{\epsilon})\leqslant\sigma^2,(\mathcal{A}_{a,i})_{\sigma}=B_{2r_{a,i}}(x_{a,i})\setminus\overline{B_{\sigma\mathfrak{r}_x^{a,i}}(\mathcal{T}_{a,i})},$$
and 
$$\int_{\mathcal{T}_{a_i}\cap B_{r_{a,i}}(x_{a,i})}(\Theta(x,r_{a,i})-\Theta(x,\mathfrak{r}_x^{a,i}))\dif x\leqslant\sigma^2.$$
Hence we define $$\bar{\mathcal{G}}_{\sigma}^{a,i}=\left\lbrace x\in\mathcal{T}_{a_i}\cap B_{r_{a,i}}(x_{a,i}):\int_{x+L^{\perp}\cap \left(B_{2r_{a,i}}\setminus B_{\mathfrak{r}_x^{a,i}}\right)}e_{\epsilon}(u_{\epsilon})\geqslant\sigma\textup{ or }\Theta(x,r_{a,i})-\Theta(x,\mathfrak{r}_x^{a,i})\geqslant\sigma\right\rbrace.$$
Then $\Hau^{m-2}(\bar{\mathcal{G}}_{\sigma}^{a,i})\leqslant\sigma$ due to previous two inequalities. Let $\mathcal{G}^2_{\sigma}$ be the projection of $\cup_{\textup{ all }a,i}\bar{\mathcal{G}}_{\sigma}^{a,i}$ on $p+L$. Then we must have $\Hau^{m-2}(\mathcal{G}_{\sigma}^2)\leqslant C(m,K_N,\Lambda)\sigma$ because there are at most $2\Lambda/\varepsilon_0$ annular regions in our construction.\par
Now, set $\mathcal{G}_{\sigma}=\mathcal{G}_{\sigma}^1\cup\mathcal{G}_{\sigma}^2$, we have $\Hau^{m-2}(\mathcal{G}_{\sigma})\leqslant C\sigma$. Let us verify on $(p+L\cap B_1)\setminus\mathcal{G}_{\sigma}$, Theorem \ref{Quantitative energy identity} is true. Theorem \ref{Quantitative energy identity} (1) (2) are automatically true due to our construction and the definitions of bubble regions and annular regions. Fix $y\in (p+L\cap B_1)\setminus\mathcal{G}_{\sigma}$, let us verify (3) of Theorem \ref{Quantitative energy identity}. Recall that we have written
$$y+L^{\perp}\cap B_1\subset(\mathcal{A}_0\cap (y+L^{\perp}))\cup\bigcup(\mathcal{A}_{a,i'}\cap (y+L^{\perp}))\cup\bigcup(\mathcal{B}_{b,i'}\cap (y+L^{\perp})),$$where $a,i'$ and $b,i'$ denote those index of which $\mathcal{A}_{a,i}$ and $\mathcal{B}_{b,i}$ intersect $y+L^{\perp}$. Again, let us prove our theorem inductively. First, by definition of $\bar{\mathcal{G}}_{\sigma}^{0}$, we have 
$$\left\vert\Theta(y,1)-\Theta(\mathfrak{t}_{\mathcal{T}_0}(y),\mathfrak{r}_y^0)\right\vert\leqslant\sigma,\mathfrak{r}^0_y=\mathfrak{r}^0_{\mathfrak{t}_{\mathcal{T}_0}(y)}.$$
The next region to intersect is $\mathcal{B}_{b,1'}=B_{r_{b,1'}}\setminus\cup B_{r_{a,2'}}(x_{a,2'}')$. Here $r_{b,1'}\geqslant\delta^{-1/2}\mathfrak{r}_{y}/2.$
If $\delta$ is chosen small enough, by Theorem \ref{bubble region properties},
$$\left\vert\Theta(\mathfrak{t}_{\mathcal{T}_0}(y),\mathfrak{r}_{y})-\omega_{m-2}\int_{(\mathcal{B}_{b,1'})_y}e_{\epsilon}(u_{\epsilon})-\sum\Theta(x_{a,2'}',r_{a,2'})\right\vert\leqslant\sigma.$$
We can repeat this process in order to further write each $\Theta(x_{a,2'},r_{a,2'})$ as the energy of bubble regions using the definition $\bar{\mathcal{G}}^{a,2'}_{\sigma}$. Then in next bubble regions the energy will be written as the energy on $y+L^{\perp}$ plus the energies on annular regions. Repeat this process at most $2\Lambda/\varepsilon_0$ times, there will be no more annular region show up in this process. Then Theorem \ref{Quantitative energy identity} (3) is verified.
\end{proof}
The energy identity Theorem \ref{main theorem} is a direct result of Theorem \ref{Quantitative energy identity}.
\begin{proof}[Proof of Theorem \ref{main theorem}]
    Let $u_{\epsilon_i}:\Omega\rightarrow\R^J$ be a sequence of solutions to (\ref{*}) with $E_{\epsilon_i}(u_{\epsilon_i})\leqslant\Lambda$ and $\epsilon_i\rightarrow0$. Passing to a subsequence, we may assume $u_{\epsilon_i}\rightharpoonup u$ a weak harmonic map to $N$ and $$e_{\epsilon_i}(u_{\epsilon_i})\dif x\rightarrow\frac{1}{2}\vert\nabla u\vert^2\dif x +\nu,$$where $\nu=\theta(x)\Hau^{m-2}\llcorner\Sigma$ is $(m-2)$-rectifiable defect measure. \par
    Let $\Sigma^*\subset\Sigma=$supp$\nu$ be the points $x$ such that the $(m-2)$-tangent plane $L_x$ of $\nu$ exists and is unique at $x$. Let $E_{\sigma,\delta}\subset\Sigma^*$ be all $x$ such that there exists at $x_i\rightarrow x$ with
    \begin{enumerate}
        \item There exists $\delta$-bubble region $\mathcal{B}_{i,a}\subset B_{r_{i,a}}(x_{i,a})$ with $x_{i,a}\subset x_i+L_x^{\perp}$;
        \item On $\mathcal{B}_i=\cup_a\mathcal{B}_{i,a}\cap(x_i+L^{\perp}_x)$ we have $\omega_{m-2}\left\vert\int_{\mathcal{B}_i}e_{\epsilon_i}(u_{\epsilon_i})-\theta(x)\right\vert\leqslant\sigma$.
    \end{enumerate}
    Let us show that $E_{\sigma,\delta}\subset\Sigma$ has full measure in $\Sigma$. First note that by $(m-2)$-rectifiability of $\nu$, $\Sigma^*$ has full measure in $\Sigma$. For each $x\in\Sigma^*,\delta'>0$, if $i$ is sufficiently large and $r\leqslant r(x,\delta')$, then $$\vert\Theta_{u_{\epsilon_i}}(x,r)-\theta(x)\vert\leqslant\delta'\textup{ and }u_{\epsilon_i}\textup{ is }(m-2)-\textup{symmetric on } B_r(x).$$ We can find $\mathcal{G}_{i,\sigma}\subset L_x\cap B_r(x)$ given by quantitative energy identity \ref{Quantitative energy identity}. Let $\mathcal{G}_{\sigma}$ be the Hausdorff limit of $\mathcal{G}_{i,\sigma}$ as $i\rightarrow\infty$, then $\Hau^{m-2}(\mathcal{G}_{\sigma})\geqslant\Hau^{m-2}(B_r(x)\cap L_x)-\sigma$. Since $\mathcal{G}_{\sigma}\subset E_{\sigma,\delta}$, it turns out that each $x\in\Sigma^*$ is a density point of $E_{\sigma,\delta}$. Hence $E_{\sigma,\delta}$ itself has full measure in $\Sigma$.\par
    Hence, $E=\cap_{\sigma,\delta}E_{\sigma,\delta}$ also has full measure in $\Sigma$. By definition of $E$, for each $x\in E$, we can find $x_i\rightarrow x$ with $x_{i,a}\in x_i+L_x$, $\delta$-bubble region $\mathcal{B}_{i,a}\subset B_{r_{i,a}}(x_{i,a})$ and $\omega_{m-2}\int_{\mathcal{B}_i}e_{\epsilon_i}(u_{\epsilon_i})\rightarrow\theta(x)$, where $\mathcal{B}_i=\cup_a\mathcal{B}_{i,a}$. Moreover, by Theorem \ref{bubble region properties}, we can find at most $K(m,K_N,\Lambda)$ points $y_{i,a,b}\in\mathcal{B}_{i,a}$ and scales $s_{i,a,b}>0$ such that $\Theta(y_{i,a,b},s_{i,a,b})\geqslant\varepsilon_0$ and for each $\eta>0,S\geqslant S(m,K_N,\Lambda,\eta)$ we have $\int_{\mathcal{B}_i,a\setminus B_{Ss_{i,a,b}}(y_{i,a,b)}}e_{\epsilon_i}(u_{\epsilon_i})\leqslant\eta$. If we define bubble $b_{a,b}=\lim_i(u_{\epsilon_i})_{y_{i,a,b},s_{i,a,b}}$, which is invariant with respect to $L_x$, then
    $$\theta(x)=\frac{1}{2}\sum_{a,b}\int\vert b_{a,b}\vert^2,$$as desired.
\end{proof}
The main goal of the rest of this paper is to prove the small energy Theorem \ref{vanishing energy in annular regions}.

\section{Local best planes}\label{s:best planes}
In Theorem \ref{annular region existence}, we built the annular region which is maximal in the sense that it almost captures all the regions with small energy. However, in order to study the energy behavior in the annular region, that maximal one is not the best one to analyze. This is practically because there is no analytic condition in the maximal region. To handle this issue, in the following two sections we introduce new annular regions which approximate the energy behavior in some best sense.\par
\subsection{Local best planes and their estimates}
The first step is the construction of the local best planes. To be specific, fix an $\delta$-annular region $\mathcal{A}=B_2(p)\setminus B_{\mathfrak{r}_x}(\mathcal{T})$. We can extend $\mathfrak{r}_x$ to $B_2(p)$ by
$$\mathfrak{r}_x=\mathfrak{r}_{\mathfrak{t}_{\mathcal{T}}(\Pi_{p+L}(x))},x\in B_2(p).$$
\begin{defn}[{cf. \cite[Section 6]{naber2024energyidentitystationaryharmonic}}]
    Let $u_{\epsilon}:B_{10R}(p)\rightarrow\R^L$ be a solution to (\ref{*}) and $\mathcal{A}=B_2(p)\setminus\overline{B_{\mathfrak{r}_x}(\mathcal{T})}$  be an $\delta$-annular region. For $x\in B_2(p)$ and $r\geqslant2\max\lbrace\mathfrak{r}_x,\textup{d}(x,\mathcal{T})\rbrace$, define 
    $$\bar{\Theta}_{\mathcal{L}}(x,r)=\min_{L^{m-2}}\bar{\Theta}(x,r;L)=\min_{L^{m-2}}r^2\int\rho_r(y-x)(\vert\Pi_L\nabla u\vert^2+2F(u_{\epsilon})/\epsilon^2),$$
    $$\mathcal{L}_{x,r}=\arg\min_{L^{m-2}}\bar{\Theta}(x,r;L),$$ 
    where the minimum is take over all $(m-2)$-dimensional subspaces in $\R^m$.
\end{defn}
We will be interested in the orthogonal projection $\Pi_{x,r}=\Pi_{\mathcal{L}_{x,r}}:\R^m\rightarrow\R^m$ and its estimates. The results to be proved in this section are the following.
\begin{thm}[{cf. \cite[Theorem 6.1]{naber2024energyidentitystationaryharmonic}}]\label{best plane properties}
    Let $u_{\epsilon}:B_{10R}(p)\rightarrow\R^L$ be a solution to (\ref{*}) and $\mathcal{A}=B_2(p)\setminus\overline{B_{\mathfrak{r}_x}(\mathcal{T})}$ be an $\delta$-annular region. If $\delta\leqslant\delta(m,K_N,\Lambda,R)$, for $x\in B_2(p)$ and $r\geqslant2\max\lbrace\mathfrak{r}_x,\textup{d}(x,\mathcal{T})\rbrace$, the best approximating plane $\mathcal{L}_{x,r}$ satisfies
    \begin{enumerate}[(1)]
        \item $\mathcal{L}_{x,r}$ uniquely and quantitatively minimizes $\bar{\Theta}(x,r;L)$ among $(m-2)$-subspaces by
        $$\bar{\Theta}(x,r;L)\geqslant\bar{\Theta}(x,r;\mathcal{L}_{x,r})+C(m)\bar{\Theta}(x,r)\textup{d}_{Gr}(L,\mathcal{L}_{x,r})^2,$$
        where $\textup{d}_{\textup{Gr}}$ is the Grassmannian distance on Gr$(m,m-2)$.
        \item $\bar{\Theta}_{\mathcal{L}}(x,r)\leqslant C(m)\delta$ and $$r^2\int\rho_r(y-x)\vert\Pi_{x,r}^{\perp}\nabla u_{\epsilon}\vert^2\geqslant2\left(1-C(m)\sqrt{\frac{\delta}{\bar{\Theta}(x,r)}}\right)\bar
        \Theta(x,r);$$
        \item $\mathcal{L}_{x,r}$ is smooth on both $x$ and $r$. Moreover, we have the estimates of derivatives
        \begin{equation}
            \begin{aligned}
                &\left\Vert r\frac{\partial}{\p r}\Pi_{x,r}\right\Vert+\left\Vert r^2\frac{\p}{\p r}\Pi_{x,r}\right\Vert+\left\Vert r\nabla\Pi_{x,r}\right\Vert+\left\Vert r^2\nabla^2\Pi_{x,r}\right\Vert\\\leqslant &C(m)\sqrt{\frac{\bar{\Theta}_{\mathcal{L}}(x,2r)}{\bar{\Theta}(x,r)}}\leqslant C(m,\epsilon_0)\sqrt{\bar{\Theta}_{\mathcal{L}}(x,2r)}.
            \end{aligned}\nonumber
        \end{equation}
    \end{enumerate}
\end{thm}
Through this section, we will repeatedly use the following fact
\begin{lem}\label{energy non doubling}
    Let $u_{\epsilon}:B_{10R}(p)\rightarrow\R^L$ be a solution to (\ref{*}) with $R^{2-m}\int_{B_{10R}(p)}e_{\epsilon}(u_{\epsilon})\leqslant\Lambda$ and $\mathcal{A}=B_2(p)\setminus\overline{B_{\mathfrak{r}_x}(\mathcal{T})}$  be a $\delta$-annular region. If $\delta\leqslant\delta(m,K_N,\Lambda,R)$, for $x\in B_2(p)$ and $r\geqslant2\max\lbrace\mathfrak{r}_x,\textup{d}(x,\mathcal{T})\rbrace$ we have
    $$\bar{\Theta}(x,r)\geqslant c(m)\bar{\Theta}(x,2r)\geqslant c(m)\varepsilon_0.$$
\end{lem}
\begin{proof}
    The fact $\bar{\Theta}(x,r)\geqslant c(m)\varepsilon_0$ follows from Definition \ref{annular regions} (a3) and monotonicity. The fact $\bar{\Theta}(x,2r)\leqslant C(m)\bar{\Theta}(x,r)$ follows from Proposition \ref{properties of heat mollifier} and Definition \ref{annular regions} (a2) that $s\frac{\p}{\p r}\bar{\Theta}(x,s)\leqslant C(m)\delta$ for $4\geqslant s\geqslant r$.
\end{proof}
When we minimize $\bar{\Theta}(x,r;L)$, only the term $r^2\int\rho_r(y-x)\vert\Pi_L\nabla u_{\epsilon}\vert^2$ plays its role. Hence to study $\mathcal{L}_{x,r}$, let us consider the heat mollified energy tensor $$Q(x,r)[v,w]=r^2\int\rho_r(y-x)\langle\nabla_v u_{\epsilon},\nabla_w u_{\epsilon}\rangle.$$ We can view $Q$ as an $m\times m$ matrix with entries depending on $x$ and $r$. Then we have the following estimate of the eigenvalues of $Q$.
\begin{lem}[{cf. \cite[Lemma 6.4]{naber2024energyidentitystationaryharmonic}}]\label{eigenvalue best plane}
     Let $u_{\epsilon}:B_{10R}(p)\rightarrow\R^L$ be a solution to (\ref{*}) with $R^{2-m}\int_{B_{10R}(p)}e_{\epsilon}(u_{\epsilon})\leqslant\Lambda$ and $\mathcal{A}=B_2(p)\setminus\overline{B_{\mathfrak{r}_x}(\mathcal{T})}$  be a $\delta$-annular region. Let $x\in B_2(p)$ and $r\geqslant2\textup{max}\lbrace\mathfrak{r}_x,\dist(x,\mathcal{T})\rbrace$. Let $\lambda_1(x,r)\geqslant\cdots\geqslant\lambda_m(x,r)$ be the eigenvalues of $Q(x,r)$ and $e_1(x,r),\dots,e_m(x,r)$ be the corresponding eigenvectors. Then
     \begin{enumerate}[(1)]
         \item $$1+C(m)\sqrt{\frac{\delta}{\bar{\Theta}(x,r)}}\geqslant\frac{\lambda_1(x,r)}{\bar{\Theta}(x,r)}\geqslant\frac{\lambda_2(x,r)}{\bar{\Theta}(x,r)}\geqslant1-C(m)\sqrt{\frac{\delta}{\bar{\Theta}(x,r)}},$$
         and 
         $$\lambda_j(x,r)\leqslant C(m)\delta\textup{ for }j\geqslant3.$$
         \item The unique best subspace $\mathcal{L}_{x,r}$ is characterized by $\textup{span}\lbrace e_3,\dots,e_m\rbrace$.
     \end{enumerate}
\end{lem}
\begin{proof}
    (2) is immediate. By definition of the annular region, we can find $x_0,\dots,x_{m-2}\in\mathcal{T}\cap B_r(x)$ that are $1/2$-linear independent at scale $r/2$ with $3r\frac{\dif}{\dif r}\bar{\Theta}(x_i,3r)\leqslant\delta,0\leqslant i\leqslant m-2$. Let $L$ be the subspace spanned by $\lbrace x_1-x_0,\dots,x_{m-2}-x_0\rbrace$, then by the quantitative cone splitting \ref{Top dimension quantitative cone splitting}, $$\sum_{j=3}^m\lambda_j(x,r)\leqslant\bar{\Theta}_{\mathcal{L}}(x,r)\leqslant\bar{\Theta}(x,r;L)\leqslant C(m)\bar{\Theta}(x_0,2r;L)\leqslant\sum_{i=0}^{m-2}r\frac{\dif}{\dif r}\bar{\Theta}(x_i,3r)\leqslant C(m)\delta.$$
    To finish the proof we need to estimate $\lambda_1$ and $\lambda_2$. To that end, for any $e\in L^{\perp}$, let us take $\xi=r^2\rho_r(y-x)\langle y-x_0,e\rangle e$ in the stationary equality (\ref{**}). This gives 
    \begin{equation}
        \begin{aligned}
            &\int e_{\epsilon}(u_{\epsilon})(\dot{\rho}_r(y-x)\langle y-x_0,e\rangle\langle y-x,e\rangle+r^2\rho_r(y-x))\\=&\int\dot{\rho}_r(y-x)\langle y-x_0,e\rangle\langle\nabla_{y-x}u_{\epsilon},\nabla_e u_{\epsilon}\rangle+r^2\rho_r(y-x)\vert\nabla_e u\vert^2.
        \end{aligned}\nonumber
    \end{equation}
    Hence
    \begin{equation}
        \begin{aligned}
            &\vert\bar{\Theta}(x,r)-\int r^2\rho_r(y-x)\vert\nabla_eu_{\epsilon}\vert^2\vert\\\leqslant&\left\vert\int \dot{\rho}_r(y-x)\langle y-x_0,e\rangle(e_{\epsilon}(u_{\epsilon})\langle y-x,e\rangle+\langle\nabla_{y-x}u_{\epsilon},\nabla_eu_{\epsilon}\rangle)\right\vert.
        \end{aligned}\label{eigen1}
    \end{equation}
    We can estimate the first term in (\ref{eigen1}) using Cauchy inequality, Lemma \ref{energy non doubling}, the quantitative cone splitting \ref{Top dimension quantitative cone splitting},
    \begin{equation}
        \begin{aligned}
            &\left\vert\int\dot{\rho}_r(y-x)\langle y-x_0,e\rangle\langle y-x,e\rangle e_{\epsilon}(u_{\epsilon})\right\vert^2\\ \leqslant&\left(\int-\dot{\rho}_r(y-x)\vert y-x\vert^2e_{\epsilon}(u_{\epsilon})\right)\left(\int-\dot{\rho}_r(y-x)\vert\Pi_{L^{\perp}}(y-x_0)\vert^2e_{\epsilon}(u_{\epsilon})\right)\\ \leqslant& C(m,R)\left(r^2\int\rho_{1.1r}(y-x)e_{\epsilon}(u_{\epsilon})\right)\\&\left(\int\rho_{2r}(y-x_0)(r^2\vert\Pi_L\nabla u_{\epsilon}\vert^2+r^2F(u_{\epsilon})/\epsilon^2+\vert\Pi_{L^{\perp}(y-x_0)}\vert^2\vert\nabla_{L^{\perp}u_{\epsilon}}\vert^2)\right)\\ \leqslant& C(m)\delta\bar{\Theta}(x,r).
        \end{aligned}\nonumber
    \end{equation}
    We have a similar estimate for the second term in (\ref{eigen1}). Hence $$\left\vert\bar{\Theta}(x,r)-\int r^2\rho(y-x)\vert\nabla_e u_{\epsilon}\vert^2\right\vert\leqslant C(m)\sqrt{\delta\bar{\Theta}(x,r)}.$$Now, pick $l\in L$ such that the first eigenvector $e_1=\alpha e+\beta l,\alpha^2+\beta^2=1$, then 
    \begin{equation}
        \begin{aligned}
            \lambda_1(x,r)&=r^2\int\rho_r(y-x)\vert\alpha\nabla_e u_{\epsilon}+\beta\nabla_l u_{\epsilon}\vert^2\\&\leqslant r^2\int\rho_r(y-x)(\alpha^2\vert\nabla_e u_{\epsilon}\vert^2+\beta^2\vert\nabla_l u_{\epsilon}\vert^2+2\langle\nabla_eu_{\epsilon},\nabla_lu_{\epsilon}\rangle\\&\leqslant r^2\int\rho_r(y-x)((\alpha^2+\beta^2)\vert\nabla_eu_{\epsilon}\vert^2+(\alpha^2+\beta^2)\vert\nabla_l u_{\epsilon}\vert^2\\&\leqslant\bar{\Theta}(x,r)+C(m)\sqrt{\delta\bar{\Theta}(x,r)}+C(m)\delta\leqslant\bar{\Theta}(x,r)+C(m)\sqrt{\delta\bar{\Theta}(x,r)}.
        \end{aligned}\nonumber
    \end{equation}
    provided $\delta$ is small enough. Note that $\bar{\Theta}(x,r)\geqslant c(m)\epsilon_0$ by our choice of $r$. The estimate on $\lambda_2(x,r)$ follows from
    \begin{equation}
        \begin{aligned}
            \lambda_1+\lambda_2=\textup{tr}Q(x,r)-\sum_{i=3}^m\lambda_3\geqslant2\bar{\Theta}(x,r)-2\bar{\Theta}_{\mathcal{L}}(x,r)\geqslant2\bar{\Theta}(x,r)-C(m)\delta.
        \end{aligned}\nonumber
    \end{equation}
\end{proof}
With Lemma \ref{eigenvalue best plane} in hand, we are ready to show the uniqueness of the best subspace $\mathcal{L}_{x,r}$. Recall that the Grassmannian Gr$(m,k)$ is the set of all $k$-dimensional subspaces in $\R^m$. It can also be regarded as the set of $(m-k)$-dimensional subspaces in $\R^m$ via orthogonal projection as well as the set of orthogonal projections on $\R^m$: 
$$\textup{Gr}(m,k)=\lbrace \Pi:\R^m\rightarrow\R^m:\Pi=\Pi^{\textup{T}},\Pi^2=1,\textup{rank}\Pi=k\rbrace.$$
Moreover, we have the equivalence of norms $$C(m)^{-1}\textup{d}_{\textup{Hau}}(\overline{V\cap B_1},\overline{W\cap B_1})\leqslant\textup{d}_{\textup{Gr}}(V,W)\leqslant C(m)\textup{d}_{\textup{Hau}}(\overline{V\cap B_1},\overline{W\cap B_1}).$$ $$C(m)^{-1}\Vert\Pi_V-\Pi_W\Vert\leqslant\textup{d}_{\textup{Gr}}(V,W)\leqslant C(m)\Vert\Pi_V-\Pi_W\Vert,$$for all $V,W\in\textup{Gr}(m,k)$, where we used the operator norm for linear maps.
\begin{lem}[{cf. \cite[Lemma 6.5]{naber2024energyidentitystationaryharmonic}}]\label{uniqueness of best plane}
     Let $u_{\epsilon}:B_{10R}(p)\rightarrow\R^L$ be a solution to (\ref{*}) with $R^{2-m}\int_{B_{10R}(p)}e_{\epsilon}(u_{\epsilon})\leqslant\Lambda$ and $\mathcal{A}=B_2(p)\setminus\overline{B_{\mathfrak{r}_x}(\mathcal{T})}$  be a $\delta$-annular region. Let $x\in B_2(p)$ and $r\geqslant2\textup{max}\lbrace\mathfrak{r}_x,\dist(x,\mathcal{T})\rbrace$. Then for any $L$ in $\textup{Gr}(m,m-2)$, we have $$\bar{\Theta}(x,r;L)\geqslant\bar{\Theta}(x,r;\mathcal{L}_{x,r})+C(m)\bar{\Theta}(x,r)\textup{d}_{Gr}(L,\mathcal{L}_{x,r})^2.$$
\end{lem}
\begin{proof}
    Due to the eigenvalue separation \ref{eigenvalue best plane}, we only need to consider the case where $L$ is close to $\mathcal{L}_{x,r}$. In this case, we can assume $L$ is the graph of linear map $G:\mathcal{L}_{x,r}\rightarrow\mathcal{L}_{x,r}^{\perp}$ with $G(e_i)=\alpha_ie_1+\beta_ie_2,i=3,\dots,m$ and $\lbrace (1-\alpha_i^2-\beta_i^2)^{1/2}e_i+\alpha_ie_1+\beta_ie_2\rbrace_{i=3}^m$ forms an orthonormal basis of $L$. Then $$C(m)^{-1}\textup{d}_{Gr}(L,\mathcal{L}_{x,r})^2\leqslant\sum_{i=3}^m(\alpha_i^2+\beta_i^2)\leqslant C(m)\textup{d}_{Gr}(L,\mathcal{L}_{x,r})^2$$Moreover, due to the definition of $Q$, we have 
    \begin{equation}
        \begin{aligned}
            &\bar{\Theta}(x,r;L)\\=&\frac{1}{2}\sum_{i=3}^mQ(x,r)[(1-\alpha_i^2-\beta_i^2)^{1/2}e_i+\alpha_ie_1+\beta_ie_2,(1-\alpha_i^2-\beta_i^2)^{1/2}e_i+\alpha_ie_1+\beta_ie_2]\\+&\int\rho_r(y-x)F(u_{\epsilon})/\epsilon^2\\=&\frac{1}{2}\sum_{i=3}^m((1-\alpha_i^2-\beta_i^2)\lambda_i(x,r)+\alpha_i^2\lambda_1(x,r)+\beta_i^2\lambda_2(x,r))+\int\rho_r(y-x)F(u_{\epsilon})/\epsilon^2\\=&\bar{\Theta}(x,r;\mathcal{L}_{x,r})+\frac{1}{2}\sum_{i=3}^m(\alpha_i^2(\lambda_1(x,r)-\lambda_i(x,r))+\beta_i^2(\lambda_2(x,r)-\lambda_i(x,r)))\\\geqslant&\bar{\Theta}(x,r;\mathcal{L}_{x,r})+\frac{1}{2}\sum_{i=3}^m(\alpha_i^2+\beta_i^2)(\bar{\Theta}(x,r)-C(m)\sqrt{\delta\bar{\Theta}(x,r)}-C(m)\delta).
        \end{aligned}\nonumber
    \end{equation}
    Since $\bar{\Theta}(x,r)\geqslant c(m)\epsilon_0,$ take $\delta$ small enough yields the result.
\end{proof}
We have the following corollary due to the Lemma \ref{uniqueness of best plane}.
\begin{cor}[{cf. \cite[Corollary 6.6]{naber2024energyidentitystationaryharmonic}}]
    Let $u_{\epsilon}:B_{10R}(p)\rightarrow\R^L$ be a solution to (\ref{*}) with $R^{2-m}\int_{B_{10R}(p)}e_{\epsilon}(u_{\epsilon})\leqslant\Lambda$ and $\mathcal{A}=B_2(p)\setminus\overline{B_{\mathfrak{r}_x}(\mathcal{T})}$  be a $\delta$-annular region. Let $x\in B_2(p)$ and $r\geqslant2\textup{max}\lbrace\mathfrak{r}_x,\dist(x,\mathcal{T})\rbrace$. Then  $$\bar{\Theta}(x,2r;\mathcal{L}_{x,r})\leqslant C(m)\bar{\Theta}_{\mathcal{L}}(x,2r).$$
\end{cor}
\begin{proof}
    By definition of $\mathcal{L}_{x,r}$,
    $$\bar{\Theta}_{\mathcal{L}}(x,r)\leqslant\bar{\Theta}(x,r;\mathcal{L_{x,2r}})\leqslant C(m)\bar{\Theta}_{\mathcal{L}}(x,2r).$$
    Hence take $L=\mathcal{L}_{x,2r}$ in Lemma \ref{uniqueness of best plane}, we get
    $$\textup{d}_{Gr}(\mathcal{L}_{x,r},\mathcal{L}_{x,2r})^2\leqslant\frac{\bar{\Theta}_{\mathcal{L}}(x,2r)}{\bar{\Theta}(x,r)}.$$
    The corollary follows from
    $$\bar{\Theta}(x,2r;\mathcal{L}_{x,r})\leqslant2\bar{\Theta}(x,2r;\mathcal{L}_{x,2r})+2C(m)\textup{d}_{\textup{Gr}}(\mathcal{L}_{x,r},\mathcal{L}_{x,2r})^2\bar{\Theta}(x,2r)\leqslant C(m)\bar{\Theta}_{\mathcal{L}}(x,2r).$$
    We used Lemma \ref{energy non doubling} again.
\end{proof}
We are also able to estimate the derivatives of $\Pi_{x,r}$.
\begin{lem}[{cf. \cite[Lemma 6.7]{naber2024energyidentitystationaryharmonic}}]\label{distance of best planes}
    Let $u_{\epsilon}:B_{10R}(p)\rightarrow\R^L$ be a solution to (\ref{*}) with $R^{2-m}\int_{B_{10R}(p)}e_{\epsilon}(u_{\epsilon})\leqslant\Lambda$ and $\mathcal{A}=B_2(p)\setminus\overline{B_{\mathfrak{r}_x}(\mathcal{T})}$  be a $\delta$-annular region. Let $x\in B_2(p)$ and $r\geqslant2\textup{max}\lbrace\mathfrak{r}_x,\dist(x,\mathcal{T})\rbrace$, then $\Pi_{x,r}=\Pi_{\mathcal{L}_{x,r}}$ is smooth with respect to $x$ and $r$ and 
    \begin{equation}
        \begin{aligned}
            &\left\Vert r\frac{\partial}{\p r}\Pi_{x,r}\right\Vert+\left\Vert r^2\frac{\p}{\p r}\Pi_{x,r}\right\Vert+\left\Vert r\nabla\Pi_{x,r}\right\Vert+\left\Vert r^2\nabla^2\Pi_{x,r}\right\Vert\\\leqslant &C(m)\sqrt{\frac{\bar{\Theta}_{\mathcal{L}}(x,2r)}{\bar{\Theta}(x,r)}}\leqslant C(m,\epsilon_0)\sqrt{\bar{\Theta}_{\mathcal{L}}(x,2r)}.
        \end{aligned}\nonumber
    \end{equation}
\end{lem}
\begin{proof}
    The smoothness of $\Pi_{x,r}$ is a corollary of eigenvalue separation \ref{eigenvalue best plane}, uniqueness \ref{uniqueness of best plane} and the smoothness of $Q(x,r)$. We prove the estimate for $r\frac{\p}{\p r}\Pi_{x,r}$, the other estimates follow in a similar way. For any $v,w\in\R^m$, $Q(x,r)[\Pi_{x,r}(v),\Pi_{x,r}^{\perp}(w)]=0\label{Lem4.6 1}$. Taking the $r$ derivative to this, we have 
    \begin{equation}
        \begin{aligned}
            0=&(2-m)Q(x,r)[\Pi_{x,r}(v),\Pi_{x,r}^{\perp}(w)]-r^2\int\dot{\rho}_r(y-x)\left\vert\frac{y-x}{r}\right\vert^2\langle\nabla_{\Pi_{x,r}(v)}u_{\epsilon},\nabla_{\Pi_{x,r}^{\perp}(w)}u_{\epsilon}\rangle\\&+Q(x,r)\left[r\frac{\p}{\p r}\Pi_{x,r}(v),\Pi_{x,r}^{\perp}(w)\right]+Q(x,r)\left[\Pi_{x,r}(v),r\frac{\p}{\p r}\Pi_{x,r}^{\perp}(w)\right].
        \end{aligned}\label{Lem 4.6 2}
    \end{equation}
    The first term in (\ref{Lem 4.6 2}) vanishes, while the second term can be estimated using Proposition \ref{properties of heat mollifier} and Lemma \ref{distance of best planes} as follows
    \begin{equation}
        \begin{aligned}
            \left\vert-r^2\int\dot{\rho}_r(y-x)\left\vert\frac{y-x}{r}\right\vert^2\langle\nabla_{\Pi_{x,r}}(v),\nabla_{\Pi_{x,r}^{\perp}(w)}u\rangle\right\vert&\leqslant C(m)\sqrt{\bar{\Theta}(x,2r;\mathcal{L}_{x,r})\bar{\Theta}(x,2r)}\\&\leqslant C(m)\sqrt{\bar{\Theta}_{\mathcal{L}}(x,2r)\bar{\Theta}(x,2r)}.
        \end{aligned}\label{Lemma 4.6 4}
    \end{equation}
    Note that $$\frac{\p}{\p r}\Pi_{x,r}(w)+\frac{\p}{\p r}\Pi_{x,r}^{\perp}(w)=0,\frac{\p}{\p r}\Pi_{x,r}:\mathcal{L}_{x,r}\rightarrow\mathcal{L}_{x,r}^{\perp},\frac{\p}{\p r}\Pi_{x,r}:\mathcal{L}^{\perp}\rightarrow\mathcal{L}_{x,r}.$$This is because $$-\left\langle\frac{\p}{\p r}\Pi_{x,r}\Pi_{x,r}v,w\right\rangle=\left\langle
    \frac{\p}{\p r}\Pi_{x,r}v,\Pi_{x,r}^{\perp}w\right\rangle.$$
    Choose unit vectors $v\in\mathcal{L}_{x,r},w\in\mathcal{L}_{x,r}^{\perp}$ with $$\left\vert\frac{\p}{\p r}\Pi_{x,r}(v)\right\vert=\left\Vert\frac{\p}{\p r}\Pi_{x,r}\bigg|_{\mathcal{L}_{x,r}}\right\Vert,w=\frac{\p}{\p r}\Pi_{x,r}(v)\big/\left\Vert\frac{\p}{\p r}\Pi_{x,r}(v)\right\Vert.$$
    Then, we can use the last two terms in (\ref{Lem 4.6 2}) to control the radial derivative of $\Pi_{x,r}$
    \begin{equation}
        \begin{aligned}
            &Q(x,r)\left[r\frac{\p}{\p r}\Pi_{x,r}(v),\Pi_{x,r}^{\perp}(w)\right]-Q(x,r)\left[\Pi_{x,r}(v),r\frac{\p}{\p r}\Pi_{x,r}(w)\right]\\\geqslant&\lambda_2(x,r)\left\Vert\frac{\p}{\p r}\Pi_{x,r}\big|_{\mathcal{L}_{x,r}}\right\Vert-\lambda_3(x,r)\left\Vert\frac{\p}{\p r}\Pi_{x,r}\big|_{\mathcal{L}_{x,r}^{\perp}}\right\Vert.
        \end{aligned}\label{Lemma 4.6 5}
    \end{equation}
    Similarly if we choose $w\in\mathcal{L}_{x,r}^{\perp},v\in\mathcal{L}_{x,r}$ with $$\left\vert\frac{\p}{\p r}\Pi_{x,r}(w)\right\vert=\left\Vert\frac{\p}{\p r}\Pi_{x,r}\big|_{\mathcal{L}_{x,r}^{\perp}}\right\Vert,v=\frac{\p}{\p r}\Pi_{x,r}(w)\big/\left\Vert\frac{\p}{\p r}\Pi_{x,r}(w)\right\Vert,$$
    then
    \begin{equation}
        \begin{aligned}
            &Q(x,r)\left[r\frac{\p}{\p r}\Pi_{x,r}(v),\Pi_{x,r}^{\perp}(w)\right]-Q(x,r)\left[\Pi_{x,r}(v),r\frac{\p}{\p r}\Pi_{x,r}(w)\right]\\\geqslant&\lambda_2(x,r)\left\Vert\frac{\p}{\p r}\Pi_{x,r}\big|_{\mathcal{L}_{x,r}^{\perp}}\right\Vert-\lambda_3(x,r)\left\Vert\frac{\p}{\p r}\Pi_{x,r}\big|_{\mathcal{L}_{x,r}}\right\Vert.
        \end{aligned}\label{Lemma 4.6 6}
    \end{equation}
    Combining (\ref{Lem 4.6 2}), (\ref{Lemma 4.6 4}),(\ref{Lemma 4.6 5}),(\ref{Lemma 4.6 6}), we see from the eigenvalue separation \ref{eigenvalue best plane} that if we choose $\delta$ small enough, then 
    \begin{equation}
        \begin{aligned}
            \left\Vert\frac{\p}{\p r}\Pi_{x,r}\right\Vert&=\max\left\lbrace\left\Vert\frac{\p}{\p r}\Pi_{x,r}\big|_{\mathcal{L}_{x,r}}\right\Vert,\left\Vert\frac{\p}{\p r}\Pi_{x,r}\big|_{\mathcal{L}_{x,r}^{\perp}}\right\Vert\right\rbrace\\&\leqslant\left\Vert\frac{\p}{\p r}\Pi_{x,r}\big|_{\mathcal{L}_{x,r}}\right\Vert+\left\Vert\frac{\p}{\p r}\Pi_{x,r}\big|_{\mathcal{L}_{x,r}^{\perp}}\right\Vert\leqslant C(m)\frac{\sqrt{\bar{\Theta}_{\mathcal{L}(x,2r)}\bar{\Theta}(x,2r)}}{\bar{\Theta}(x,r)}.
        \end{aligned}\nonumber
    \end{equation}
    Using the Lemma \ref{energy non doubling} again we have the desired estimate.
\end{proof}

\subsection{Improved quantitative cone splitting theorem in annular regions}
In this subsection we wish to establish the quantitative cone splitting result related to $\mathcal{L}_{x,r}$. Recall that Theorem \ref{Top dimension quantitative cone splitting} tells us that if $x_0,x_1,\dots,x_{m-2}$ effectively span an $(m-2)$-subspace $L$ on scale $r$, then
$$\bar{\Theta}(x,r;L)+\bar{\Theta}(x,r;L^{\perp})\leqslant C(m)\sum_{j=0}^{m-2}r\frac{\p}{\p r}\bar{\Theta}(x,r).$$ 
We want to make this inequality more effective in annular regions. Namely, we wish the $(m-2)$-subspace on left hand side is our best plane $\mathcal{L}_{x,r}$ and right hand side is a quantity related to the annular region. We first have the following improvement.
\begin{prop}[{cf. \cite[Theorem 2.23]{naber2024energyidentitystationaryharmonic}}]\label{quantitative cone splitting in annular regions}
    Let $u_{\epsilon}:B_{10R}(p)\rightarrow\R^L$ be a solution to (\ref{*}) with $R^{2-m}\int_{B_{10R}(p)}e_{\epsilon}(u_{\epsilon})\leqslant\Lambda$ and $\mathcal{A}=B_2(p)\setminus\overline{B_{\mathfrak{r}_x}(\mathcal{T})}$  be a $\delta$-annular region. Then for any $x\in\mathcal{T},B_{r}(x)\subset B_2(p)$ there exists a $\bar{x}$ and an $(m-2)$-plane $L=L_{x,r}$ with $$\vert\bar{x}-x\vert\leqslant C(m)\sqrt{\delta}r,\textup{d}_{\textup{Gr}}(L,L_{\mathcal{A}})\leqslant C(m)\sqrt{\delta}$$ such that $$\bar{\Theta}(\bar{x},r;L)+\bar{\Theta}(\bar{x},r;L^{\perp})\leqslant C(m)\fint_{\mathcal{T}\cap B_{r/2}(x)}r\frac{\p}{\p r}\bar{\Theta}(z,3r/2)\dif\mathcal{H}^{m-2}(z).$$
\end{prop}
\begin{proof}
    We may assume $r=1$. By definition (a1) of an annular region, we can take $x_0=x,x_1\dots x_{m-2}\in B_{1/20}(x)$ that are 1/2-linearly independent (see Definition \ref{effective independent}. Set 
    $$\bar{x}_i=\fint_{B_{1/1000}(x_i)\cap\mathcal{T}}z\dif\mathcal{H}^{m-2}(z).$$
    Then by the smallness of the curvature of $\mathcal{T}$, we have $x,\bar{x}_1,\dots,\bar{x}_{m-2}$ are 1/3- linear independent if $\delta$ is small enough. Using Jensen's inequality we have
    \begin{equation}
        \begin{aligned}
            &\frac{\p}{\p r}\bar{\Theta}(\bar{x}_i,1.4)\\=&\int\left(-\dot{\rho}_{1.4}(y-\bar{x}_i)\left\vert\left\langle\nabla u_{\epsilon},y-\fint_{B_{1/1000}(x_i)\cap\mathcal{T}}z\dif\mathcal{H}^{m-2}(z)\right\rangle\right\vert^2+\rho_{1.4}(y-\bar{x}_i)F(u_{\epsilon})/\epsilon^2\right)\dif y\\\leqslant&\int\left(-\dot{\rho}_{1.4}(y-\bar{x_i})\fint_{B_{1/1000}(x_i)\cap\mathcal{T}}\vert\langle\nabla u_{\epsilon},y-z\rangle\vert^2\dif\mathcal{H}^{m-2}(z)+\rho_{1.4}(y-\bar{x}_i)F(u_{\epsilon})/\epsilon^2\right)\dif y\\=&\fint_{B_{1/1000}(x_i)\cap\mathcal{T}}\int\left(-\dot{\rho}_{1.4}(y-\bar{x_i})\vert\langle\nabla u_{\epsilon},y-z\rangle\vert^2+\rho_{1.4}(y-\bar{x}_i)F(u_{\epsilon})/\epsilon^2\right)\dif y\dif\mathcal{H}^{m-2}(z)\\\leqslant &C(m)\fint_{B_{1/1000}(x_i)\cap\mathcal{T}}\int\left(-\dot{\rho}_{3/2}(y-z)\vert\langle\nabla u_{\epsilon},y-z\rangle\vert^2+\rho_{3/2}(y-z)F(u_{\epsilon})/\epsilon^2\right)\dif y\dif\mathcal{H}^{m-2}(z)\\=&C(m)\fint_{B_{1/1000}(x_i)\cap\mathcal{T}}\frac{\p}{\p r}\bar{\Theta}(z,3/2)\dif\mathcal{H}^{m-2}(z).
        \end{aligned}\nonumber
    \end{equation}
    As $\mathcal{T}$ is bi-Lipshitz to $L_{\mathcal{A}}$ with Lipshitz constants bounded by $\delta$, for $\delta$ small we have 
    $$c(m)\leqslant\mathcal{H}^{m-2}(\mathcal{T}\cap B_{1/1000}(x_i)\leqslant\mathcal{H}^{m-2}(\mathcal{T}\cap B_{1/2}(x)\leqslant C(m).$$
    Hence 
    $$\fint_{B_{1/1000}(x_i)\cap\mathcal{T}}\frac{\p}{\p r}\bar{\Theta}(z,3/2)\dif\mathcal{H}^{m-2}(z)\leqslant C(m)\fint_{B_{1/2}(x)\cap\mathcal{T}}\frac{\p}{\p r}\bar{\Theta}(z,3/2)\dif\mathcal{H}^{m-2}(z).$$
    Combining the top stratum cone splitting \ref{Top dimension quantitative cone splitting}, we have the desired estimate.
\end{proof}
Next, let us see how to improve the $(m-2)$-subspace on the left side to $\mathcal{L}_{x,r}$.
\begin{thm}[{cf. \cite[Theorem 6.9]{naber2024energyidentitystationaryharmonic}}]\label{effective quantitative cone splitting1}
    Let $u_{\epsilon}:B_{10R}(p)\rightarrow\R^L$ be a solution to (\ref{*}) with $R^{2-m}\int_{B_{10R}(p)}e_{\epsilon}(u_{\epsilon})\leqslant\Lambda$ and $\mathcal{A}=B_2(p)\setminus\overline{B_{\mathfrak{r}_x}(\mathcal{T})}$  be a $\delta$-annular region. For each $x\in\mathcal{T}$ with $r\geqslant2\mathfrak{r}_x$, we have $$\bar{\Theta}(x,r;\mathcal{L}_{x,r})+\bar{\Theta}(x,r;\mathcal{L}_{x,r}^{\perp})\leqslant C(m)\left(r\frac{\p}{\p r}\bar{\Theta}(x,2r)+\fint_{\mathcal{T}\cap B_r(x)}r\frac{\p}{\p r}\bar{\Theta}(z,2r)\dif\mathcal{H}^{m-2}(z)\right).$$
\end{thm}
\begin{proof}
    By Proposition \ref{quantitative cone splitting in annular regions}, there exists an $\bar{x},L$, with 
    $$\bar{\Theta}(\bar{x},1.1r;L)+\bar{\Theta}(\bar{x},1.1r;L^{\perp})\leqslant C\fint_{\mathcal{T}\cap B_{r/2}(x)}r\frac{\p}{\p r}\bar{\Theta}(z,3r/2)\dif\mathcal{H}^{m-2}(z),$$ 
    $$r\frac{\p}{\p r}\bar{\Theta}(\bar{x},1.1r)\leqslant C\fint_{\mathcal{T}\cap B_{r/2}(x)}r\frac{\p}{\p r}\bar{\Theta}(z,3r/2)\dif\mathcal{H}^{m-2}(z).$$
    A simple comparison gives the control 
    $$\bar{\Theta}(x,r;\mathcal{L}_{x,r})\leqslant C(m)\bar{\Theta}(\bar{x},1.1r;L).$$
    For the $\bar{\Theta}(x,r;\mathcal{L}_{x,r}^{\perp})$, we have by Lemma \ref{uniqueness of best plane}
    $$\textup{d}_{\textup{Gr}}(L,\mathcal{L}_{x,r})^2\leqslant C(m)\frac{\bar{\Theta}(x,r;L)}{\bar{\Theta
    }(x,r)}\leqslant \frac{C(m)}{\bar{\Theta}(x,r)}\fint_{\mathcal{T}\cap B_{r/2}(x)}r\frac{\p}{\p r}\bar{\Theta}(z,3r/2)\dif\mathcal{H}^{m-2}(z).$$ 
    On the other hand,
    \begin{equation}
        \begin{aligned}
            \vert\Pi_{x,r}^{\perp}(x-\bar{x})\vert^2\bar{\Theta}(x,r)&\leqslant\frac{1}{2}r^2\int\rho_r(y-x)\vert\nabla_{x-\bar{x}}u_{\epsilon}\vert^2\\&\leqslant r^2\int\rho_r(y-x)(\vert\nabla_{y-x}u_{\epsilon}\vert^2+\vert\nabla_{y-\bar{x}}u_{\epsilon}\vert^2)\\&\leqslant C(m)r^2\left(r\frac{\p}{\p r}\bar{\Theta}(x,3r/2)+r\frac{\p}{\p r}\bar{\Theta}(\bar{x},3r/2)\right).
        \end{aligned}\nonumber
    \end{equation}
    Hence
    \begin{equation}
        \begin{aligned}
            &\bar{\Theta}(x,r;\mathcal{L}_{x,r}^{\perp})\\=&\int\rho_r(y-x)\vert\Pi_{x,r}^{\perp}(y-x)\vert^2\vert\Pi_{x,r}^{\perp}\nabla u_{\epsilon}\vert^2\\\leqslant&\int\rho_r(y-x)(2\vert\Pi_{x,r}^{\perp}(y-\bar{x})\vert^2(2\vert\Pi_{L}^{\perp}\nabla u_{\epsilon}\vert^2+2\vert\nabla u_{\epsilon}\vert^2\vert\Pi_{x,r}^{\perp}-\Pi_{L}^{\perp}\vert^2)+2\vert\Pi_{x,r}^{\perp}(x-\bar{x})\vert^2\vert\nabla u_{\epsilon}\vert^2)\\\leqslant& C(m)\bar{\Theta}(\bar{x},1.1r;L^{\perp})+C(m)\bar{\Theta}(x,1.1r)\textup{d}_{\textup{Gr}}(L,\mathcal{L}_{x,r})^2+2\frac{\vert \Pi_{x,r}^{\perp}(x-\bar{x})\vert^2}{r^2}\bar{\Theta}(x,r)\\\leqslant& C(m)\left(r\frac{\p}{\p r}\bar{\Theta}(x,2r)+\fint_{\mathcal{T}\cap B_{r/2}(x)}r\frac{\p}{\p r}\bar{\Theta}(z,3r/2)\dif\mathcal{H}^{m-2}(z)\right).
        \end{aligned}\nonumber
    \end{equation}
\end{proof}

\section{Best approximating submanifolds}\label{s:best submanifold}
In this section we would glue together the best planes $\mathcal{L}_{x,r}$ to build best approximating submanifolds. Roughly speaking, near the submanifold $\mathcal{T}$ in the annular region the density of $u_{\epsilon}$ roughly looks like a constant in a fixed $(m-2)$-subspace. As a result, the heat mollified density function looks roughly like a 2-dimensional Gaussian. We will construct the approximating submanifold $\mathcal{T}_r$ such that for each $x\in\mathcal{T}_r$ with $r>2\mathfrak{r}_x$, $\bar{\Theta}$ assume its maximal on $\mathcal{T}_r$. Hence $\mathcal{T}_r$ is the center of this Gaussian. Precisely, we have the following.
\begin{thm}[{cf. \cite[Theorem 7.1]{naber2024energyidentitystationaryharmonic}}]\label{apporximating submanifold}
    Let $u_{\epsilon}:B_{10R}(p)\rightarrow\R^L$ be a solution to (\ref{*}) with $R^{2-m}\int_{B_{10R}(p)}e_{\epsilon}(u_{\epsilon})\leqslant\Lambda$ and $\mathcal{A}=B_2(p)\setminus\overline{B_{\mathfrak{r}_x}(\mathcal{T})}$  be a $\delta$-annular region. For each $r>0$ there exists a submanifold $\mathcal{T}_r\subset B_{3/2}(p)$ with the followings hold:
    \begin{enumerate}[(1)]
        \item We can globally write $\mathcal{T}_r=p+\textup{Graph}_{L_\mathcal{A}}\mathfrak{t}_r$, where $\mathfrak{t}_r:L_\mathcal{A}\rightarrow L_{\mathcal{A}}^{\perp}$ satisfies $$\vert \mathfrak{t}_r\vert+\vert\nabla\mathfrak{t}_r\vert+r\vert\nabla^2\mathfrak{t}_r\vert\leqslant C(m)\sqrt{\delta};$$
        \item For each $x\in\mathcal{T}_r$ with $r\geqslant 2\mathfrak{r}_x$ we can locally write $\mathcal{T}_r\cap B_r(x)=x+\textup{Graph}_{\mathcal{L}_{x,r}}\mathfrak{t}_{x,r}$, where $\mathfrak{t}_{x,r}:\mathcal{L}_{x,r}\cap B_r\rightarrow\mathcal{L}_{x,r}^{\perp}$ satisfies $$\vert\nabla\mathfrak{t}_{x,r}\vert+r\vert\nabla^2\mathfrak{t}_{x,r}\vert+r^2\vert\nabla^3\mathfrak{t}_{x,r}\vert\leqslant C(m,\varepsilon_0)\sqrt{\bar{\Theta}_{\mathcal{L}}(x,2r)};$$
        \item For each $x\in\mathcal{T}_r$ with $r\geqslant2\mathfrak{r}_x$, we have $$\left\vert\frac{\p}{\p r}\mathfrak{t}_r(x)\right\vert+\left\vert r\frac{\p}{\p r}\nabla\mathfrak{t}_r(x)\right\vert\leqslant C(m,\epsilon_0)\left(\sqrt{r\frac{\p}{\p r}\bar{\Theta}(x,2r)}+\sqrt{\bar{\Theta}_{\mathcal{L}}(x,2r)}\right)\leqslant C(m,\varepsilon_0)\sqrt{\delta};$$
        \item If $x\in\mathcal{T}_r$ and $r\leqslant\mathfrak{r}_x$, then $\mathcal{T}_r\cap B_r(x)=\mathcal{T}\cap B_r(x)$;
        \item If $x\in B_{3/2}(p)$ and $r\geqslant2\mathfrak{r}_x$, then the following Eular-Lagrange equation holds $$\Pi_{x,r}^{\perp}\nabla\bar{\Theta}(x,r)=0\iff\forall v\in\mathcal{L}_{x,r}^{\perp},\int-\dot{\rho}_r(y-x)\langle\nabla_{y-x}u_{\epsilon},\nabla_vu_{\epsilon}\rangle=0;$$
        \item For each $x\in\mathcal{T}_r$ with $r\geqslant2\mathfrak{r}_x$, we have 
        $$\Vert\Pi_{\mathcal{T}_r}(x)-\Pi_{x,r}\Vert\leqslant C(m,\varepsilon_0)\left(\sqrt{r\frac{\p}{\p r}\bar{\Theta}(x,2r)}+e^{-R/2}\right)\sqrt{\bar{\Theta}_{\mathcal{L}}(x,2r)}.$$
        Here $\Pi_{\mathcal{T}_r}(x)$ is the orthogonal projection to the tangent space $\textup{T}_x\mathcal{T}_r$.
        \item For all $r,r'$ with $r/10\leqslant r'\leqslant 10 r$ we have $$\int_{\mathcal{T}_r\cap B_{1/4}(p)}r\frac{\p}{\p r}\bar{\Theta}(x,r)\dif\mathcal{H}^{m-2}(x)\leqslant C(m)\int_{\mathcal{T}\cap B_{1/3}(p)}r'\frac{\p}{\p r}\bar{\Theta}(x,1.1r')\dif\mathcal{H}^{m-2}(x).$$
    \end{enumerate}
\end{thm}
\subsection{The local construction}
Let us first state the local construction of $\mathcal{T}_r$.
\begin{lem}[{cf. \cite[Lemma 7.6]{naber2024energyidentitystationaryharmonic}}]\label{constructing approx submanifold}
    Let $u_{\epsilon}:B_{10R}(p)\rightarrow\R^L$ be a solution to (\ref{*}) with $R^{2-m}\int_{B_{10R}(p)}e_{\epsilon}(u_{\epsilon})\leqslant\Lambda$ and $\mathcal{A}=B_2(p)\setminus\overline{B_{\mathfrak{r}_x}(\mathcal{T})}$  be a $\delta$-annular region. Let $x\in\mathcal{T}$ and $r\geqslant2\mathfrak{r}_x$ and $L$ be an $(m-2)$-subspace with $\textup{d}_{\textup{Gr}}(L,L_{\mathcal{A}})\leqslant c_m$, where $c_m$ is a small constant depending only on $m$. Then there exists a unique $\mathfrak{t}_{L,r}:L\cap B_r(\Pi_L(x-p))\rightarrow L^{\perp}$ with the following properties.
    \begin{enumerate}[(1)]
        \item For each $z\in\tilde{\mathcal{T}}_r=p+\textup{Graph}_L\mathfrak{t}_{L,r}$, we have $\Pi_{z,r}^{\perp}\nabla\bar{\Theta}(z,r)=0$;
        \item $\vert\nabla \mathfrak{t}_{L,r}\vert+r\vert\nabla^2\mathfrak{t}_{L,r}\vert+r^2\vert\nabla^3\mathfrak{t}_{L,r}\vert\leqslant C(m,\varepsilon_0)\sqrt{\bar{\Theta}(x,2r;L)}$;
        \item $\vert\frac{\p}{\p r}\mathfrak{t}_{L,r}\vert+r\vert\frac{\p}{\p r}\nabla\mathfrak{t}_{L,r}\vert\leqslant C(m,\varepsilon_0)\left(\sqrt{\bar{\Theta}_{\mathcal{L}}(x,2r)}+\sqrt{r\frac{\p}{\p r}\bar{\Theta}(x,2r)}\right)$;
        \item $\vert z-\Pi_{\tilde{\mathcal{T}}_r}(z)\vert\leqslant C(m,\varepsilon_0)\sqrt{\delta}r$ for all $z\in\mathcal{T}$, where $\hat{z}=\Pi_{\tilde{\mathcal{T}}_r}(z)$ is the only point on $\tilde{\mathcal{T}}_r$ such that $\Pi_{L}(z-p)=\Pi_{L}(\hat{z}-p)$;
        \item For all $z\in\tilde{\mathcal{T}}_r$, let $\hat{z}$ be the only point on $\mathcal{T}$ such that $\Pi_{L_{\mathcal{A}}}(z-p)=\Pi_{L_{\mathcal{A}}}(\hat{z}-p)$, then for all $r'\in[r/10,10r]$, we have $$r'\frac{\p}{\p r}\bar{\Theta}(z,r')\leqslant C(m) r'\frac{\p}{\p r}\bar{\Theta}(\hat{z},1.1r').$$
    \end{enumerate}
\end{lem}
The proof is based on an implicit function theorem argument. Set $$\Omega(z)=r\Pi_{L^{\perp}}\Pi^{\perp}_{z,r}\nabla\bar{\Theta}(z,r).$$
The strength of considering the projection $\Pi_{L^{\perp}}$ is that it has a fixed range $L^{\perp}$ which is easier to control compared to varying planes $\mathcal{L}_{z,r}$. Also, since $L$ is close to $\mathcal{L}_{z,r}$ for $z\in B_r(x)$, we know that $\Pi_{L^{\perp}}:\mathcal{L}_{z,r}^{\perp}\rightarrow L^{\perp}$ is bi-Lipchitz. As a result, $\Omega(z)=0$ if and only if $\Pi_{z,r}^{\perp}\nabla\bar{\Theta}(z,r)=0$. We will prove the following properties of $\Omega(z)$:
\begin{enumerate}
    \item For $z\in B_{2r}(x)\cap\mathcal{T},\vert\Omega(z)\vert\leqslant C(m,\Lambda)\sqrt{\delta}$;
    \item For all $z\in B_{2r}(x)\cap B_{c(m)r}(\mathcal{T})$, the $2\times2$ matrix $r\Pi_{L^{\perp}}\nabla\Omega(z)$ is bounded from above by $-\frac{1}{5}\bar{\Theta}(z,r)\textup{Id}_{L^{\perp}}$, namely 
    \begin{equation}
        \begin{aligned}
            r\langle\nabla_v\Omega(z),v\rangle\leqslant-\frac{1}{5}\bar{\Theta}(z,r)\vert v\vert^2,\textup{ for all }v\in L^{\perp},
        \end{aligned}\label{positivity of derivative}
    \end{equation}
\end{enumerate}
Note that above two claims guarantees that for each $z\in B_{2r}(x)\cap\mathcal{T}$, there exists a unique $y\in (z+L^{\perp})\cap B_{c(m)r}(z)$ such that $\Omega(y)=0$. To see the existence of such $y$ in $z+L^{\perp}$, one can simply consider the curve $\gamma(t)\in z+L^{\perp}$ with 
$$\gamma'(t)=\Omega(\gamma(t)),\gamma(0)=z,$$
and calculate $\frac{\dif}{\dif t}\vert\Omega(\gamma(t))\vert^2$. Hence, using implicit function theorem, we see that there exists a smooth map $\mathfrak{t}_{L,r}:L\cap B_{3r/2}(\Pi_L(x-p))\rightarrow L^{\perp}$ such that 
$$\Omega(z+\mathfrak{t}_{L,r}(z))=0,\textup{ for all } z\in B_{3r/2}\cap L.$$
\begin{proof}
    The first claim is straightforward. Indeed, using Proposition \ref{spacial derivative est}, we see that $$\vert\Omega(z)\vert\leqslant Cr\vert\Pi_{z,r}^{\perp}\nabla\bar{\Theta}(z,r)\vert\leqslant C(m)\sqrt{r\frac{\p}{\p r}\bar{\Theta}(z,r)}\sqrt{\bar{\Theta}(z,2r)}\leqslant C(m)\sqrt{\bar{\Theta}(z,2r)}\sqrt{\delta}.$$
    For the second claim, let $v$ be a unit vector in $L^{\perp}$, then 
    \begin{equation}
        \begin{aligned}            r\langle\nabla_v\Omega(z),v\rangle&=r\nabla_v\langle\Omega(z),v\rangle=r^2\nabla_v\langle\Pi_{z,r}^{\perp}\nabla\bar{\Theta}(z,r),v\rangle=r^2\nabla_v\langle\nabla\bar{\Theta}(z,r),\Pi^{\perp}_{z,r}v\rangle\\&=r^2\nabla^2\bar{\Theta}(z,r)[v,\Pi^{\perp}_{z,r}(v)]+r^2\langle\nabla\bar{\Theta}(z,r),(\nabla_v\Pi^{\perp}_{z,r})v\rangle.
        \end{aligned}\label{construction of approx submanifold 1}
    \end{equation}
    The second piece in (\ref{construction of approx submanifold 1}) is small. Indeed, by spacial gradient estimate \ref{spacial gradient}, we see that $r\vert\nabla\bar{\Theta}(z,r)\vert\leqslant C(m)\bar{\Theta}(z,2r)$. Also, from Theorem \ref{best plane properties} (3), we have $\vert(\nabla_v\Pi^{\perp}_{z,r})v\vert\leqslant C(m)\sqrt{\delta/\bar{\Theta}(z,2r)}$. Hence, we have $$r^2\langle\nabla\bar{\Theta}(z,r),(\nabla_v\Pi^{\perp}_{z,r})v\rangle\leqslant C(m)\sqrt{\delta}\sqrt{\bar{\Theta}(z,2r)}.$$
    For the first piece in (\ref{construction of approx submanifold 1}), recall (\ref{spcial hessian}): 
    \begin{equation}
        \begin{aligned}
            &r^2\nabla^2_L\bar{\Theta}(z,r)[v,\Pi^{\perp}_{z,r}(v)]\\=&r^2\int\ddot{\rho}_r(y-z)\left\langle\nabla_{\frac{z-y}{r}}u_{\epsilon},\nabla_vu_{\epsilon}\right\rangle\left\langle\frac{z-y}{r},\Pi^{\perp}_{z,r}(v)\right\rangle+\dot{\rho}_r(y-z)\left\langle\nabla_vu_{\epsilon},\nabla_{\Pi^{\perp}_{z,r}(v)}u_{\epsilon}\right\rangle.\label{hessian again}
        \end{aligned}
    \end{equation}
    To estimate the first part in this Hessian, let $\tilde{z}\in\mathcal{T}$ be the only point such that $\Pi_{L_{\mathcal{A}}}(z-p)=\Pi_{L_{\mathcal{A}}}(\tilde{z}-p)$ with $\vert z-\tilde{z}\vert\leqslant2\textup{dist}(z,\mathcal{T})$. Then,
    \begin{equation}
        \begin{aligned}
        &r^2\left\vert\int\ddot{\rho}_r(y-z)\left\langle\nabla_{\frac{z-y}{r}}u_{\epsilon},\nabla_vu_{\epsilon}\right\rangle\left\langle\frac{z-y}{r},\Pi^{\perp}_{z,r}(v)\right\rangle\right\vert\\\leqslant &C(m)\sqrt{\int\rho_{1.5r}(y-z)\vert\nabla_{y-z}u_{\epsilon}\vert^2}\sqrt{\bar{\Theta}(z,2r)},
    \end{aligned}\nonumber
    \end{equation}
    \begin{equation}
        \begin{aligned}
            \int\rho_{1.5r}(y-z)\vert\nabla_{y-z}u_{\epsilon}\vert^2\leqslant&2\int\rho_{1.5r}(y-z)\vert\nabla_{y-\tilde{z}}u_{\epsilon}\vert^2+8\int\rho_{1.5r}(y-z)\vert\nabla u_{\epsilon}\vert^2\textup{dist}(z,\mathcal{T})^2\\\leqslant &C(m)\int\dot{\rho}_{2r}(y-\tilde{z})\vert\nabla_{y-\tilde{z}}u_{\epsilon}\vert^2+C(m)\bar{\Theta}(z,2r)\frac{\textup{dist}(x,\mathcal{T})^2}{r^2}\\\leqslant &C(m)\left(\delta+\bar{\Theta}(z,2r)\frac{\textup{dist}(z,\mathcal{T})^2}{r^2}\right)\\\leqslant &C(m)\left(\delta+\bar{\Theta}(z,r)\frac{\textup{dist}(z,\mathcal{T})^2}{r^2}\right),
        \end{aligned}\nonumber
    \end{equation}
    where in the last step we used Lemma \ref{energy non doubling}. For the second part in (\ref{hessian again}), we first have
    \begin{equation}
        \begin{aligned}
            &r^2\int\dot{\rho}_r(y-z)\left\langle\nabla_vu_{\epsilon},\nabla_{\Pi^{\perp}_{z,r}(v)}u_{\epsilon}\right\rangle\\=&r^2\int (\dot{\rho}_r(y-z)+\rho_r(y-z))\left\langle\nabla_vu_{\epsilon},\nabla_{\Pi^{\perp}_{z,r}(v)}u_{\epsilon}\right\rangle-\rho_r(y-z)\left\langle\nabla_vu_{\epsilon},\nabla_{\Pi^{\perp}_{z,r}(v)}u_{\epsilon}\right\rangle.
        \end{aligned}\nonumber
    \end{equation}
    By Proposition \ref{properties of heat mollifier},
    $$r^2\int (\dot{\rho}_r(y-z)+\rho_r(y-z))\left\langle\nabla_vu_{\epsilon},\nabla_{\Pi^{\perp}_{z,r}(v)}u_{\epsilon}\right\rangle\leqslant C(m)e^{-R/2}\bar{\Theta}(z,2r)\leqslant C(m)e^{-R/2}\bar{\Theta}(z,r).$$
    On the other hand, since $\textup{d}_{\textup{Gr}}(L,L_{\mathcal{A}})\leqslant c_m$ is small and $\textup{d}_{\textup{Gr}}(L_{\mathcal{A}},\mathcal{L}_{z,r})\leqslant C(m,\epsilon_0)\sqrt{\delta}$, we see that $\textup{d}_{\textup{Gr}}(L,\mathcal{L}_{z,r})=\textup{d}_{\textup{Gr}}(L^{\perp},\mathcal{L}_{z,r}^{\perp})$ is small. Hence by $v\in L^{\perp}$ is unit, we have
    $$\vert\Pi_{z,r}^{\perp}v-v\vert=\vert\Pi_{z,r}v\vert\leqslant C\textup{d}_{\textup{Gr}}(L,\mathcal{L}_{z,r}),\vert\Pi_{z,r}^{\perp}v\vert^2\geqslant1/2.$$ 
    As a result,
    \begin{equation}
        \begin{aligned}
            &r^2\int\rho_r(y-z)\left\langle\nabla_vu_{\epsilon},\nabla_{\Pi^{\perp}_{z,r}(v)}u_{\epsilon}\right\rangle\\=&r^2\int\rho_r(y-z)\left\langle\nabla_{\Pi^{\perp}_{z,r}(v)}u_{\epsilon},\nabla_{\Pi^{\perp}_{z,r}(v)}u_{\epsilon}\right\rangle+\rho_r(y-z)\left\langle\nabla_{\Pi_{z,r}(v)}u_{\epsilon},\nabla_{\Pi^{\perp}_{z,r}(v)}u_{\epsilon}\right\rangle\\\geqslant&\frac{1}{4}\bar{\Theta}(z,r)-\sqrt{\bar{\Theta}(z,r)}\sqrt{\bar{\Theta}_{\mathcal{L}}(z,r)}\textup{d}_{\textup{Gr}}(L,\mathcal{L}_{z,r}).
        \end{aligned}\nonumber
    \end{equation}
    Putting all those estimates together, we see that 
    \begin{equation}
        \begin{aligned}
        &r\langle\nabla_v\Omega(z),v\rangle=r\nabla_v\langle\Omega(z),v\rangle\\\leqslant&\left(\left(-\frac{1}{4}+C(m)e^{-R/2}+C(m)\frac{\textup{dist}(z,\mathcal{T})}{r}\right)\bar{\Theta}(z,r)+C(m)\sqrt{\delta}\sqrt{\bar{\Theta}(z,r)}\right).
        \end{aligned}\nonumber
    \end{equation}
    Hence if $R\geqslant R(m)$ is sufficiently large and $\delta\leqslant\delta(m),\textup{dist}(z,\mathcal{T})/r\leqslant c(m)$ are sufficiently small, then $$r\langle\nabla_v\Omega(z),v\rangle\leqslant-\frac{1}{5}\bar{\Theta}(z,r),$$
    which is the desired estimate.
\end{proof}
\begin{proof}[Proof of spatial derivative estimates]
    By implicit function theorem, the gradient estimate follows from 
    \begin{equation}
        \vert r\nabla_L\Omega(z)\vert\leqslant C(m)\sqrt{\bar{\Theta}(z,2r)}\sqrt{\bar{\Theta}(z,2r;L)}.\label{best approx spacial}
    \end{equation}
    This follows from a similar strategy of implicit function claim but is easier. Let $l$ be any unit vector in $L$, then take $v\in L^{\perp}$ be the unit vector that is parallel to $\nabla_l\Omega(z)$, we have
    \begin{equation}
        \begin{aligned}
            &\vert r\nabla_l\Omega(z)\vert=r\langle\nabla_l\Omega(z),v\rangle\\=&r^2\nabla^2\bar{\Theta}(z,r)[l,\Pi^{\perp}_{z,r}(v)]+r^2\langle\nabla\bar{\Theta}(z,r),(\nabla_v\Pi^{\perp}_{z,r})l\rangle.
        \end{aligned}\nonumber
    \end{equation}
    By spacial gradient estimates \ref{spacial derivative est} and the definition of best approximating plane, 
    \begin{equation}
        r^2\langle\nabla\bar{\Theta}(z,r),(\nabla_v\Pi^{\perp}_{z,r})l\rangle\leqslant C(m)\frac{\sqrt{\bar{\Theta}_{\mathcal{L}}(z,2r)}}{\sqrt{\bar{\Theta}(z,r)}}\bar{\Theta}(z,2r)\leqslant C(m)\sqrt{\bar{\Theta}(z,2r)}\sqrt{\bar{\Theta}(z,2r;L)}.\label{best approx spacial1}
    \end{equation}
    Recall (\ref{spcial hessian}) again,
    \begin{equation}
        \begin{aligned}
            &\vert r^2\nabla^2_L\bar{\Theta}(z,r)[l,\Pi^{\perp}_{z,r}(v)]\vert\\=&\left\vert r^2\int\ddot{\rho}_r(y-z)\left\langle\nabla_{\frac{z-y}{r}}u_{\epsilon},\nabla_lu_{\epsilon}\right\rangle\left\langle\frac{z-y}{r},\Pi^{\perp}_{z,r}(v)\right\rangle+\dot{\rho}_r(y-z)\left\langle\nabla_lu_{\epsilon},\nabla_{\Pi^{\perp}_{z,r}(v)}u_{\epsilon}\right\rangle\right\vert\\ \leqslant& C(m)\sqrt{\bar{\Theta}(z,2r)}\sqrt{\bar{\Theta}(z,2r;L)}.
        \end{aligned}\label{best approx spacial2}
    \end{equation}
    (\ref{best approx spacial1}) and (\ref{best approx spacial2}) implies (\ref{best approx spacial}). Hence implicit function theorem gives 
    $$\vert\nabla \mathfrak{t}_{L,r}\vert\leqslant C(m)\frac{\sqrt{\bar{\Theta}(z,2r;L)}}{\sqrt{\bar{\Theta}(z,2r)}}.$$
    The estimates for $\nabla^2\mathfrak{t}_{L,r}$ and $\nabla^3\mathfrak{t}_{L,r}$ follows from a similar manner.
\end{proof}
\begin{proof}[Proof of radial derivative estimate]
    By definition, for $z\in L\cap B_{3r/2}(\Pi_L(x))$, at $w=z+\mathfrak{t}_{L,r}(z)$ we have 
    \begin{equation}
        \begin{aligned}
            0=&r\frac{\p}{\p r}\Omega(z+\mathfrak{t}_{L,r}(z))\\=&r\Pi_{L^{\perp}}\left(r\frac{\p}{\p r}\Pi_{w,r}^{\perp}(\nabla\bar{\Theta}(w,r))+\Pi_{w,r}^{\perp}\left(r\nabla\frac{\p}{\p r}\bar{\Theta}(w,r)\right)\right)+r\nabla_{\frac{\p}{\p r}\mathfrak{t}_{L,r}}\Omega(w).
        \end{aligned}\label{radial derivative equation}
    \end{equation}
    We use the last term in (\ref{radial derivative equation}) and  (\ref{positivity of derivative}) to control
    \begin{equation}
        \begin{aligned}
            \left\vert \frac{\p}{\p r}\mathfrak{t}_{L,r}(z)\right\vert&\leqslant\frac{5}{\bar{\Theta}(w,r)}\left\vert r\nabla_{\frac{\p}{\p r}\mathfrak{t}_{L,r}(z)}\Omega(w)\right\vert\\&\leqslant\frac{5}{\bar{\Theta}(w,r)}\left(\left\Vert r\frac{\p}{\p r}\Pi_{w,r}^{\perp}\right\Vert\vert r\nabla\bar{\Theta}(w,r)\vert+\left\vert r^2\nabla\frac{\p}{\p r}\bar{\Theta}(w,r)\right\vert\right).
        \end{aligned}\nonumber
    \end{equation}
    By Theorem \ref{best plane properties} and Proposition \ref{spacial derivative est}, the first term in (\ref{radial derivative equation}) is bounded by
    $$\left\Vert r\frac{\p}{\p r}\Pi^{\perp}_{w,r}\right\Vert=\left\Vert r\frac{\p}{\p r}\Pi_{w,r}\right\Vert\leqslant C(m)\frac{\sqrt{\bar{\Theta}_{\mathcal{L}}(w,2r)}}{\bar{\Theta}(w,2r)},\vert r\nabla\bar{\Theta}(w,r)\vert\leqslant C(m)\bar{\Theta}(w,2r).$$
    Moreover, for any unit vector $v$, we calculate 
    $$r^2\nabla_v\frac{\p}{\p r}\bar{\Theta}(w,r)=\int\ddot{\rho}_r(y-w)\left\langle \frac{y-w}{r},v\right\rangle\vert\nabla_{y-w}u_{\epsilon}\vert^2+2\dot{\rho}(y-w)\langle\nabla_{\frac{y-w}{r}}u_{\epsilon},\nabla_v u_{\epsilon}\rangle.$$
    Thus we obtain the estimate of second term in (\ref{radial derivative equation}) by Lemma \ref{properties of heat mollifier}, 
    $$\left\vert r^2\nabla\frac{\p}{\p r}\bar{\Theta}(w,r)\right\vert\leqslant C(m)\sqrt{\bar{\Theta}(z,2r)}\sqrt{r\frac{\p}{\p r}\bar{\Theta}(w,2r)}.$$
    Hence,
    $$\left\vert\frac{\p}{\p r}\mathfrak{t}_{L,r}(z)\right\vert\leqslant C(m)\frac{\sqrt{\bar{\Theta}_{\mathcal{L}}(w,2r)}+\sqrt{r\frac{\p}{\p r}\bar{\Theta}(w,2r)}}{\sqrt{\bar{\Theta}(w,r)}}.$$
    We can similarly get the estimate for $r\nabla\frac{\p}{\p r}\mathfrak{t}_{L,r}$.
\end{proof}
\begin{proof}[Proof of distance estimates between $\mathcal{T}$ and $\tilde{\mathcal{T}}_r$]
    For any $z\in\mathcal{T}$, let $\hat{z}=\Pi_{\tilde{\mathcal{T}}_r}(z)$, the estimate of $$\left\vert\frac{z-\hat{z}}{r}\right\vert^2\leqslant C(m)\frac{r\frac{\p}{\p r}\bar{\Theta}(z,2r)}{\bar{\Theta}(z,r)},$$
    follows directly from the implicit function claim. Hence, for all $r'\in[r,10r]$, we have 
    \begin{equation}
        \begin{aligned}
            r'\frac{\p}{\p r}\bar{\Theta}(\hat{z},r')&\leqslant-2\int\dot{\rho}_{r'}(y-\hat{z})\vert\nabla_{y-z}u\vert^2-2\int\dot{\rho}_{r'}(y-\hat{z})\vert\nabla_{z-\hat{z}}u\vert^2\\&\leqslant C(m)\left(r'\frac{\p}{\p r}\bar{\Theta}(z,2r')+\bar{\Theta}(z,2r')\left\vert\frac{z-\hat{z}}{r}\right\vert^2\right)\\&\leqslant C(m)r'\frac{\p}{\p r}\bar{\Theta}(z,2r').
        \end{aligned}\nonumber
    \end{equation}
\end{proof}
\subsection{Gluing local pieces of best approximating submanifold and estimates}

Now, we can construct the best approximating submanifold globally. Let $\lbrace x_i\rbrace\subset\mathcal{T}\cap B_{3/2}(p)$ be such that $\lbrace B_{r/4}(x_i)\rbrace$ is a maximal disjoint set. Then $B_{r/2}(\mathcal{T}\cap B_{3/2}(p))\subset\cup B_r(x_i)$. Let $x_i'=\Pi_{L_{\mathcal{A}}}(x_i-p)$, then $B_{r/5}(x_i')$ are disjoint with $L_{\mathcal{A}}\cap B_{3/2}\subset\cup B_{3r/4}(x_i')$. Let $\lbrace\phi_i\rbrace\subset C_c^{\infty}(L_{\mathcal{A}}\cap B_{3/2})$ be a partition of unity such that 
$$\textup{supp}\phi_i\subset B_r(x_i'),r^k\vert\nabla^{(k)}\phi_i\vert\leqslant C(m,k),\sum\phi_i=1\textup{ on }L_{\mathcal{A}}\cap B_{3/2}(\Pi_{L_{\mathcal{A}}}(p)).$$
For each $i$, define $\mathfrak{t}_i:B_r(x_i')\rightarrow L_{\mathcal{A}}^{\perp}$ as
\begin{equation}
    \mathfrak{t_i}=\begin{cases}
        \mathfrak{t}_{L_{\mathcal{A}},r},&\textup{ if }r\geqslant5\mathfrak{r}_{x_i}/3;\\
        \mathfrak{t}_{\mathcal{T}},&\textup{ if }r<4\mathfrak{r}_{x_i}/3.
    \end{cases}\nonumber
\end{equation}
Finally, set $\mathfrak{t}_r(x)=\sum\phi_i(x)\mathfrak{t}_i(x)$ and $\mathcal{T}_r=p+$Graph$\mathfrak{t}_r$. Then Theorem \ref{apporximating submanifold}(1)(3)(4)(5) is clearly satisfied by $\mathcal{T}_r$. (2) follows by applying the Lemma \ref{constructing approx submanifold} to $L=\mathcal{L}_{x,r}$ and use the estimates from previous section. (6) is valid because the distance estimate proved above always holds on $\mathcal{T}_r$. Finally, let us prove Theorem \ref{apporximating submanifold} (6).
\begin{proof}[Proof of Theorem \ref{apporximating submanifold} (6)]
    For any $x\in\mathcal{T}_r$ with $r\geqslant 2\mathfrak{r}_x$, take a unit vector $v\in\textup{T}_x\mathcal{T}_r$, we need to show that 
    $$\vert\Pi_{x,r}^{\perp}(v)\vert\leqslant C(m)\left(\frac{\sqrt{r\frac{\p}{\p r}\bar{\Theta}(x,2r)}\sqrt{\bar{\Theta}_{\mathcal{L}}(x,2r)}}{\bar{\Theta}(x,r)}+e^{-R/2}\frac{\sqrt{\bar{\Theta}_{\mathcal{L}}(x,2r)}}{\sqrt{\bar{\Theta}(x,r)}}\right).$$
    To this end, by definition we have 
    \begin{equation}
        0=\nabla_v\Pi_{x,r}^{\perp}\nabla\bar{\Theta}(x,r)=(\nabla_v\Pi_{x,r}^{\perp})\nabla\bar{\Theta}(x,r)+\Pi_{x,r}^{\perp}\nabla_v\nabla\bar{\Theta}(x,r).\label{5.1(6)1}
    \end{equation}
    From Theorem \ref{spacial derivative est} and Theorem \ref{best plane properties} (3) we can bound the first term in (\ref{5.1(6)1}) 
    \begin{equation}
        \begin{aligned}
            r^2\vert(\nabla_v\Pi_{x,r}^{\perp})\nabla\bar{\Theta}(x,r)\vert\leqslant &C(m)\frac{\sqrt{\bar{\Theta}_{\mathcal{L}}(x,2r)}}{\sqrt{\bar{\Theta}(x,r)}}\sqrt{\bar{\Theta}(x,2r)}\sqrt{\frac{\p}{\p r}\bar{\Theta}(x,2r)}\\\leqslant &C(m)\sqrt{\bar{\Theta}_{\mathcal{L}}(x,2r)}\sqrt{\frac{\p}{\p r}\bar{\Theta}(x,2r)}.
        \end{aligned}
    \end{equation}
    On the other hand, recall once more (\ref{spcial hessian})
    \begin{equation}
        \begin{aligned}
            &r^2\nabla^2_L\bar{\Theta}(x,r)[v,\Pi^{\perp}_{x,r}(v)]\\=&r^2\int\ddot{\rho}_r(y-x)\left\langle\nabla_{\frac{x-y}{r}}u_{\epsilon},\nabla_vu_{\epsilon}\right\rangle\left\langle\frac{x-y}{r},\Pi^{\perp}_{x,r}(v)\right\rangle+\dot{\rho}_r(y-x)\left\langle\nabla_vu_{\epsilon},\nabla_{\Pi^{\perp}_{x,r}(v)}u_{\epsilon}\right\rangle.
        \end{aligned}\nonumber
    \end{equation}
    We employ Theorem \ref{best plane properties} (2) to control $\vert\Pi_{x,r}(v)\vert^2$ using (\ref{5.1(6)1}):
    \begin{equation}
        \begin{aligned}
            &\frac{1}{2}\bar{\Theta}(x,r)\vert\Pi_{x,r}^{\perp}(v)\vert^2\\\leqslant &r^2\int\rho_r(y-x)\vert\nabla_{\Pi_{x,r}^{\perp}(v)}u_{\epsilon}\vert^2\leqslant C(m)r^2\int-\dot{\rho}_r(y-x)\vert\nabla_{\Pi_{x,r}^{\perp}(v)}u_{\epsilon}\vert^2\\=&r^2\int-\dot{\rho}_r(y-x)\langle\nabla_{\Pi^{\perp}_{x,r}(v)}u_{\epsilon},\nabla_vu_{\epsilon}\rangle+\dot{\rho}_r(y-x)\langle\nabla_{\Pi_{x,r}(v)}u_{\epsilon},\nabla_{\Pi_{x,r}^{\perp}(v)}u_{\epsilon}\rangle\\=&r^2\int\ddot{\rho}_r(y-x)\left\langle\nabla_{\frac{x-y}{r}}u_{\epsilon},\nabla_vu_{\epsilon}\right\rangle\left\langle\frac{x-y}{r},\Pi^{\perp}_{x,r}(v)\right\rangle\\+&r^2\int\dot{\rho}_r(y-x)\langle\nabla_{\Pi_{x,r}(v)}u_{\epsilon},\nabla_{\Pi_{x,r}^{\perp}(v)}u_{\epsilon}\rangle+\langle(\nabla_v\Pi_{x,r}^{\perp})\nabla\bar{\Theta}(x,r),\Pi^{\perp}_{x,r}(v)\rangle.
        \end{aligned}\label{5.1(6)2}
    \end{equation}
    The first term in (\ref{5.1(6)2}) can be estimated using
    \begin{equation}
        \begin{aligned}
            &r^2\int\ddot{\rho}_r(y-x)\left\langle\nabla_{\frac{x-y}{r}}u_{\epsilon},\nabla_vu_{\epsilon}\right\rangle\left\langle\frac{x-y}{r},\Pi^{\perp}_{x,r}(v)\right\rangle\\\leqslant&C(m)\left\vert\Pi_{x,r}^{\perp}(v)\right\vert r\int-\dot{\rho}_{1.5r}(y-x)\vert\nabla_{y-x}u_{\epsilon}\vert\left(\gamma\vert\nabla_{\Pi_{x,r}^{\perp}(v)}u_{\epsilon}\vert+\frac{1}{\gamma}\vert\nabla_{\Pi_{x,r}(v)}u_{\epsilon}\vert\right)\\\leqslant&C(m)\left\vert\Pi_{x,r}^{\perp}(v)\right\vert\left(\gamma\left\vert\Pi_{x,r}^{\perp}(v)\right\vert\bar{\Theta}(x,r)+\frac{1}{\gamma}\sqrt{\bar{\Theta}_{\mathcal{L}}(x,2r)}\sqrt{r\frac{\p}{\p r}\bar{\Theta}(x,2r)}\right),
        \end{aligned}\nonumber
    \end{equation}
    for some $\gamma>0$ to be determined.
    On the other hand, using 
    $$r^2\int\rho_r(y-x)\langle\nabla_{\Pi_{x,r}^{\perp}(v)}u_{\epsilon},\nabla_{\Pi_{x,r}(v)}u_{\epsilon}\rangle=0,$$
    we have 
    \begin{equation}
        \begin{aligned}
            &r^2\int\dot{\rho}_r(y-x)\langle\nabla_{\Pi_{x,r}(v)}u_{\epsilon},\nabla_{\Pi_{x,r}^{\perp}(v)}u_{\epsilon}\rangle\\=&r^2\int(\rho_r(y-x)+\dot{\rho}_r(y-x))\langle\nabla_{\Pi_{x,r}(v)}u_{\epsilon},\nabla_{\Pi_{x,r}^{\perp}(v)}u_{\epsilon}\rangle\\\leqslant& C(m)r^2e^{-R/2}\int\rho_{2r}(y-x)\vert\nabla_{\Pi_{x,r}(v)}u_{\epsilon}\vert\vert\nabla_{\Pi_{x,r}^{\perp}(v)}u_{\epsilon}\vert\\\leqslant &C(m)e^{-R/2}\sqrt{\bar{\Theta}(x,2r)}\sqrt{\bar{\Theta}_{\mathcal{L}}(x,2r)}\left\vert\Pi_{x,r}^{\perp}(v)\right\vert.
        \end{aligned}\nonumber
    \end{equation}Putting those estimates together, we see that 
    \begin{equation}
        \begin{aligned}
            &\frac{1}{2}\bar{\Theta}(x,r)\vert\Pi_{x,r}^{\perp}(v)\vert^2\leqslant C(m)\gamma\bar{\Theta}(x,r)\vert\Pi_{x,r}^{\perp}(v)\vert^2+\\&C(m)\left(\left(\frac{1}{\gamma}+1\right)\sqrt{r\frac{\p}{\p r}\bar{\Theta}(x,2r)}\sqrt{\bar{\Theta}_{\mathcal{L}}(x,2r)}+e^{-R/2}\sqrt{\bar{\Theta}_{\mathcal{L}}(x,2r)}\sqrt{\bar{\Theta}(x,r)}\right)\left\vert\Pi_{x,r}^{\perp}(v)\right\vert.
        \end{aligned}\nonumber
    \end{equation}
    Hence, if we choose $\gamma C(m)\leqslant\frac{1}{4}$, we have the desired estimate.
\end{proof}
We also have the improved cone splitting on our best approximating submanifold. The proof goes exactly as Theorem \ref{effective quantitative cone splitting1}
\begin{thm}{cf. \cite[Theorem 6.9]{naber2024energyidentitystationaryharmonic}}\label{effective qunatitative cone splitting2}
    Let $u_{\epsilon}:B_{10R}(p)\rightarrow\R^L$ be a solution to (\ref{*}) with $R^{2-m}\int_{B_{10R}(p)}e_{\epsilon}(u_{\epsilon})\leqslant\Lambda$ and $\mathcal{A}=B_2(p)\setminus\overline{B_{\mathfrak{r}_x}(\mathcal{T})}$  be a $\delta$-annular region. For $\mathcal{T}_r$ constructed in Theorem \ref{apporximating submanifold} and each $x\in\mathcal{T}_r$ with $r\geqslant2\mathfrak{r}_x$, we have $$\bar{\Theta}(x,r;\mathcal{L}_{x,r})+\bar{\Theta}(x,r;\mathcal{L}_{x,r}^{\perp})\leqslant C(m)\left(\frac{\p}{\p r}\bar{\Theta}(x,2r)+\fint_{\mathcal{T}_r\cap B_r(x)}r\frac{\p}{\p r}\bar{\Theta}(z,2r)\dif\mathcal{H}^{m-2}(z)\right).$$
\end{thm}

\section{Energy decomposition in annular regions}\label{s:energy decomposition}
In this subsection let us see how to decompose the total energy in an annular region into three parts
$$\bar{\Theta}_{\mathcal{L}}(x,r),\bar{\Theta}(x,r;\alpha_{x,r})\textup{ and }\bar{\Theta}(x,r;n_{x,r}).$$
As we have seen in the eigenvalue separation, namely Theorem \ref{best plane properties} (2), that the energy in the perpendicular direction $\bar{\Theta}(x,r;\mathcal{L}_{x,r}^{\perp})$ is almost the total energy $\bar{\Theta}(x,r)$. Hence, it is bounded below by $\varepsilon_0$. If we simply integrate this term over $\mathcal{T}$ or $\mathcal{T}_r$, it is impossible to get any smallness result required by Theorem \ref{vanishing energy in annular regions}. Hence, we need to cut off the energy in the perpendicular direction away from $x+\mathcal{L}_{x,r}$ if $x\in \mathcal{T}$ or $\mathcal{T}_r$. 
\begin{defn}[$L$-heat mollifier,{\cite[Definition 2.30]{naber2024energyidentitystationaryharmonic}}]
    Let $\psi:\R_{\geqslant0}\rightarrow\R_{\geqslant0}$ be a smooth cut-off function with $\psi(t)=0$ for $t\leqslant e^{-2R}$ and $\psi(t)=1$ for $t\geqslant2e^{-2R}$. Define $$\hat{\rho}_r(y;L)=\rho_r(y)\psi\left(\frac{\vert\Pi_{L^{\perp}}(y)\vert^2}{2r^2}\right).$$
\end{defn}
An easy computation shows
\begin{lem}[{\cite[Lemma 2.32]{naber2024energyidentitystationaryharmonic}}]\label{rough estimate for L-cutoff}
    $$\vert r\nabla\hat{\rho}_r(y;L)\vert\leqslant C(m)\rho_{2r}(y).$$
\end{lem}
Now, let us define the heat mollified energy away from $\mathcal{L}_{x,r}$. Recall that in Section \ref{s:best planes} the best approximating plane $\mathcal{L}_{x,r}$ is defined as the $(m-2)$-subspace in $\R^m$ which realizes the minimum $\min_{L^{m-2}}\bar{\Theta}(x,r;L)$. $\mathcal{L}_{x,r}$ is well-defined and unique. 
\begin{defn}[cf. {\cite[Section 8.1]{naber2024energyidentitystationaryharmonic}}]
    Let $u_{\epsilon}:B_{2R}(p)\rightarrow\R^L$ be a solution to (\ref{*}) with $R^{2-m}\int_{B_{2R}(p)}e_{\epsilon}(u_{\epsilon})\leqslant\Lambda$. Define 
    $$\hat{\Theta}_{\alpha}(x,r)=\int\hat{\rho}_r(y-x;\mathcal{L}_{x,r})\vert\nabla_{\alpha^{\perp}_{x,r}}u_{\epsilon}\vert^2\vert\Pi_{x,r}^{\perp}(y-x)\vert^2,$$ 
    $$\hat{\Theta}_n(x,r)=\int\hat{\rho}_r(y-x;\mathcal{L}_{x,r})\vert\nabla_{\Pi_{x,r}^{\perp}(y-x)}u_{\epsilon}\vert^2=\int\hat{\rho}_r(y-x;\mathcal{L}_{x,r})\vert\nabla_{n^{\perp}_{x,r}}u_{\epsilon}\vert^2\vert\Pi_{x,r}^{\perp}(y-x)\vert^2.$$
    Here $\alpha_{x,r}^{\perp}$ and $n_{x,r}^{\perp}$ are respectively the angular and radial coordinate in affine 2-plane $x+\mathcal{L}_{x,r}^{\perp}$.
\end{defn}
It's easy to verify the following estimates of $\hat{\rho}_r(y-x,\mathcal{L}_{x,r})$, taken in to account the the derivative estimates Theorem \ref{best plane properties} (3).
\begin{prop}[{\cite[Lemma 8.4]{naber2024energyidentitystationaryharmonic}}]\label{refine estimates for L-cutoff}
    Let $u_{\epsilon}:B_{10R}(p)\rightarrow\R^L$ be a solution to (\ref{*}) with $R^{2-m}\int_{B_{10R}(p)}e_{\epsilon}(u_{\epsilon})\leqslant\Lambda$ and $\mathcal{A}=B_2(p)\setminus\overline{B_{\mathfrak{r}_x}(\mathcal{T})}$  be a $\delta$-annular region. If $\delta$ is sufficiently small, for each $x\in B_2(p)$ with $r>2\max\lbrace\mathfrak{r}_x,\textup{d}(x,\mathcal{T})\rbrace$, we have 
    $$\left\vert r\frac{\p}{\p r}\hat{\rho}_r(y-x;\mathcal{L}_{x,r})\right\vert+\left\vert\left(r\frac{\p}{\p r}\right)^2\hat{\rho}_r(y-x;\mathcal{L}_{x,r})\right\vert\leqslant C(m)\rho_{1.1r}(y-x);$$ 
    $$r\left\vert \nabla^x\hat{\rho}_r(y-x;\mathcal{L}_{x,r})\right\vert+ r^2\left\vert\nabla^x\nabla^x\hat{\rho}_r(y-x;\mathcal{L}_{x,r})\right\vert\leqslant C(m)\rho_{1.1r}(y-x);$$ 
    $$r\left\vert \nabla^y\hat{\rho}_r(y-x;\mathcal{L}_{x,r})\right\vert+ r^2\left\vert\nabla^y\nabla^y\hat{\rho}_r(y-x;\mathcal{L}_{x,r})\right\vert\leqslant C(m)\rho_{1.1r}(y-x).$$
\end{prop}
Next we want to control the whole energy in the annular region using the energies along best approximating submanifold $\mathcal{T}_r$. To this end, let us introduce the cut off function along $\mathcal{T}_r$.
\begin{defn}[cf. {\cite[Definition 8.5]{naber2024energyidentitystationaryharmonic}}]
    Let $u_{\epsilon}:B_{10R}(p)\rightarrow\R^L$ be a solution to (\ref{*}) with $R^{2-m}\int_{B_{10R}(p)}e_{\epsilon}(u_{\epsilon})\leqslant\Lambda$ and $\mathcal{A}=B_2(p)\setminus\overline{B_{\mathfrak{r}_x}(\mathcal{T})}$  be a $\delta$-annular region. Let $\phi:\R_{\geqslant0}\rightarrow\R_{\geqslant0}$ be non-increasing cut-off function with $\phi(t)=1$ if $t\leqslant1$ and $\phi(t)=0$ if $t\geqslant2$. Then define $$\psi_{\mathcal{T}}(x,r)=\phi(\mathfrak{r}_x/r)\phi(r)\phi\left(\vert\Pi_{L_{\mathcal{A}}}(x-p)\vert^2\right).$$
\end{defn}
Here we collect the basic properties of $\psi_{\mathcal{T}}$
\begin{prop}[{\cite[Lemma 8.8]{naber2024energyidentitystationaryharmonic}}]\label{estimates of T-cutoff}
    Let $u_{\epsilon}:B_{10R}(p)\rightarrow\R^L$ be a solution to (\ref{*}) with $R^{2-m}\int_{B_{10R}(p)}e_{\epsilon}(u_{\epsilon})\leqslant\Lambda$ and $\mathcal{A}=B_2(p)\setminus\overline{B_{\mathfrak{r}_x}(\mathcal{T})}$  be a $\delta$-annular region. If $\delta\leqslant\delta(m,\Lambda,K_N,R)$ is sufficiently small, then
    \begin{enumerate}[(1)]
        \item $\psi_{\mathcal{T}}(x,r)=0$ if $r\leqslant\mathfrak{r}_x/2$ or $\vert\Pi_{\mathcal{L}_{\mathcal{A}}}(x-p)\vert\geqslant2$;
        \item $\psi_{\mathcal{T}}(x,r)=1$ if $\mathfrak{r}_x\leqslant r\leqslant1$ and $\vert\Pi_{\mathcal{L}_{\mathcal{A}}}(x-p)\vert\leqslant1$;
        \item $\left\vert r\frac{\p}{\p r}\psi_{\mathcal{T}}\right\vert+r\left\vert\nabla\psi_{\mathcal{T}}\right\vert+r^2\left\vert\frac{\p^2}{\p r^2}\psi_{\mathcal{T}}\right\vert+r^2\left\vert \nabla^2\psi_{\mathcal{T}}\right\vert+r^2\left\vert\nabla\frac{\p}{\p r}\psi_{\mathcal{T}}\right\vert\leqslant C(m)$;
        \item $\int\frac{\dif r}{r}\int_{\mathcal{T}_r}\left\vert r\frac{\p}{\p r}\psi_{\mathcal{T}}\right\vert+r\left\vert\nabla\psi_{\mathcal{T}}\right\vert+r^2\left\vert\frac{\p^2}{\p r^2}\psi_{\mathcal{T}}\right\vert+r^2\left\vert \nabla^2\psi_{\mathcal{T}}\right\vert+r^2\left\vert\nabla\frac{\p}{\p r}\psi_{\mathcal{T}}\right\vert\dif \mathcal{H}^{m-2}\leqslant C(m)$.
    \end{enumerate}
\end{prop}
With all those tools in hand, we are now in a position to define the energies along $\mathcal{T}_r$.
\begin{defn}[cf. {\cite[Section 8.4]{naber2024energyidentitystationaryharmonic}}]
    Let $u_{\epsilon}:B_{10R}(p)\rightarrow\R^L$ be a solution to (\ref{*}) with $R^{2-m}\int_{B_{10R}(p)}e_{\epsilon}(u_{\epsilon})\leqslant\Lambda$ and $\mathcal{A}=B_2(p)\setminus\overline{B_{\mathfrak{r}_x}(\mathcal{T})}$  be a $\delta$-annular region. Define 
    $$\mathcal{E}_{\mathcal{L}}(\mathcal{T}_r)=\int_{\mathcal{T}_r}\psi_{\mathcal{T}}(x,r)\bar{\Theta}_{\mathcal{L}}(x,r)=r^2\int_{\mathcal{T}_r}\psi_{\mathcal{T}}(x,r)\int\rho_r(y-x)(\vert\Pi_{x,r}\nabla u_{\epsilon}\vert^2+F(u_{\epsilon})/\epsilon^2);$$
    $$\mathcal{E}_{\alpha}(\mathcal{T}_r)=\int_{\mathcal{T}_r}\psi_{\mathcal{T}}(x,r)\bar{\Theta}_{\alpha}(x,r)=\int_{\mathcal{T}_r}\psi_{\mathcal{T}}(x,r)\int\rho_r(y-x)\vert\Pi_{x,r}^{\perp}(y-x)\vert^2\vert\nabla_{\alpha_{x,r}^{\perp}}u_{\epsilon}\vert^2;$$
    $$\hat{\mathcal{E}}_{\alpha}(\mathcal{T}_r)=\int_{\mathcal{T}_r}\psi_{\mathcal{T}}(x,r)\hat{\Theta}_{\alpha}(x,r)=\int_{\mathcal{T}_r}\psi_{\mathcal{T}}(x,r)\int\hat{\rho}_r(y-x,\mathcal{L}_{x,r})\vert\Pi_{x,r}^{\perp}(y-x)\vert^2\vert\nabla_{\alpha_{x,r}^{\perp}}u_{\epsilon}\vert^2;$$
    $$\mathcal{E}_n(\mathcal{T}_r)=\int_{\mathcal{T}_r}\psi_{\mathcal{T}}(x,r)\bar{\Theta}_n(x,r)=\int_{\mathcal{T}_r}\psi_{\mathcal{T}}(x,r)\int\rho_r(y-x)\vert\Pi_{x,r}^{\perp}(y-x)\vert^2\vert\nabla_{n_{x,r}^{\perp}}u_{\epsilon}\vert^2,$$
    $$\hat{\mathcal{E}}_n(\mathcal{T}_r)=\int_{\mathcal{T}_r}\psi_{\mathcal{T}}(x,r)\hat{\Theta}_n(x,r)=\int_{\mathcal{T}_r}\psi_{\mathcal{T}}(x,r)\int\hat{\rho}_r(y-x,\mathcal{L}_{x,r})\vert\Pi_{x,r}^{\perp}(y-x)\vert^2\vert\nabla_{n_{x,r}^{\perp}}u_{\epsilon}\vert^2.$$
    Roughly speaking, $\mathcal{E}_{\mathcal{L}}(\mathcal{T}_r),\mathcal{E}_{\alpha}(\mathcal{T}_r),\mathcal{E}_n(\mathcal{T}_r)$ represents the total tangential, angular, radial energy in $B_r(\mathcal{T}_r)$ respectively while the cut-off version $\hat{\mathcal{E}}_{\alpha}(\mathcal{T}_r),\hat{\mathcal{E}}_n(\mathcal{T}_r)$ are the angular and radial energy in $B_r(\mathcal{T}_r)\setminus B_{\mathfrak{r}_x}(\mathcal{T}_r)$.
\end{defn}
The basic estimates for the energies along $\mathcal{T}_r$ are
\begin{prop}[cf. {\cite[Lemma 8.10]{naber2024energyidentitystationaryharmonic}}]\label{basic estimate of energy}
    Let $u_{\epsilon}:B_{10R}(p)\rightarrow\R^L$ be a solution to (\ref{*}) with $R^{2-m}\int_{B_{10R}(p)}e_{\epsilon}(u_{\epsilon})\leqslant\Lambda$ and $\mathcal{A}=B_2(p)\setminus\overline{B_{\mathfrak{r}_x}(\mathcal{T})}$  be a $\delta$-annular region. If $\delta\leqslant\delta(m,\Lambda,K_N,R)$ is sufficiently small, then
    $$\mathcal{E}_{\mathcal{L}}(\mathcal{T}_r)+\hat{\mathcal{E}}_{\alpha}(\mathcal{T}_r)+\hat{\mathcal{E}}_n(\mathcal{T}_r)\leqslant\mathcal{E}_{\mathcal{L}}(\mathcal{T}_r)+\mathcal{E}_{\alpha}(\mathcal{T}_r)+\mathcal{E}_n(\mathcal{T}_r)\leqslant C(m)\int_{\mathcal{T}}r\frac{\p}{\p r}\bar{\Theta}(x,3r).$$
\end{prop}
\begin{proof}
    The first inequality is clear from definition. For the second inequality, by the improved cone splitting on $\mathcal{T}_r$, we have 
    \begin{equation}
        \begin{aligned}
            \mathcal{E}_{\mathcal{L}}(\mathcal{T}_r)+\mathcal{E}_{\alpha}(\mathcal{T}_r)+\mathcal{E}_n(\mathcal{T}_r)&=\int_{\mathcal{T}_r}\psi_{\mathcal{T}}(x,r)(\bar{\Theta}_{\mathcal{L}}(x,r)+\bar{\Theta}(x,r;\mathcal{L}_{x,r}^{\perp}))\\&\leqslant C(m)\int_{\mathcal{T}_r}\psi(x,r)\left(r\frac{\p}{\p r}\bar{\Theta}(x,2r)+\fint_{\mathcal{T}_r\cap B_r(x)}r\frac{\p}{\p r}\bar{\Theta}(z,2r)\right)\\&\leqslant\int_{\mathcal{T}_r\cap B_{1.5}(p)}r\frac{\p}{\p r}\bar{\Theta}(x,2.5r).
        \end{aligned}\nonumber
    \end{equation}
    where in the last inequality we used the fact that $\mathcal{T}_r$ is a Lipschitz graph over $\mathcal{L}_{\mathcal{A}}$ with small Lipschitz factor to exchange integral. Now we can turn the last integral to an integral on $\mathcal{T}$ by enlarging 2.5r a little bit using the property of $\mathcal{T}_r$.
\end{proof}
Finally, let us integrate the energies along $\mathcal{T}_r$ to get the full energy within $\mathcal{A}$.
\begin{defn}[cf. {\cite[Section 8.5]{naber2024energyidentitystationaryharmonic}}]
    Let $u_{\epsilon}:B_{10R}(p)\rightarrow\R^L$ be a solution to (\ref{*}) with $R^{2-m}\int_{B_{10R}(p)}e_{\epsilon}(u_{\epsilon})\leqslant\Lambda$ and $\mathcal{A}=B_2(p)\setminus\overline{B_{\mathfrak{r}_x}(\mathcal{T})}$  be a $\delta$-annular region. Define
    $$\mathcal{E}_{\mathcal{L}}(\mathcal{A})=\int\mathcal{E}_{\mathcal{L}}(\mathcal{T}_r)\frac{\dif r}{r},\mathcal{E}_{\alpha}(\mathcal{A})=\int\hat{\mathcal{E}}_{\alpha}(\mathcal{T}_r)\frac{\dif r}{r},\mathcal{E}_n(\mathcal{A})=\int\hat{\mathcal{E}}_n(\mathcal{T}_r)\frac{\dif r}{r}.$$
\end{defn}
Then we can check above three parts of energy indeed give rise to full energy on $\mathcal{A}$.
\begin{thm}[cf. {\cite[Lemma 8.12]{naber2024energyidentitystationaryharmonic}}]\label{energy decomposition}
    Let $u_{\epsilon}:B_{10R}(p)\rightarrow\R^L$ be a solution to (\ref{*}) with $R^{2-m}\int_{B_{10R}(p)}e_{\epsilon}(u_{\epsilon})\leqslant\Lambda$ and $\mathcal{A}=B_2(p)\setminus\overline{B_{\mathfrak{r}_x}(\mathcal{T})}$  be a $\delta$-annular region. If $\delta\leqslant\delta(m,\Lambda,K_N,R)$ and $R\geqslant R(\sigma)$, then we have 
    $$\int_{\mathcal{A}_{\sigma}\cap B_1(p)}e_{\epsilon}(u_{\epsilon})\leqslant C(m)(\mathcal{E}_{\mathcal{L}}(\mathcal{A})+\mathcal{E}_{\alpha}(\mathcal{A})+\mathcal{E}_n(\mathcal{A})),\mathcal{A}_{\sigma}= B_2(p)\setminus\overline{B_{\sigma\mathfrak{r}_x}(\mathcal{T})}.$$
\end{thm}
\begin{proof}
    On the one hand, if $\delta$ is small and $R$ is large, we have
    \begin{equation}
        \begin{aligned}
            &\int\frac{\dif r}{r}\int_{\mathcal{T}}\psi_{\mathcal{T}}(x,r)\dif x\frac{1}{r^{m-2}}\int_{B_{10r}(x)\setminus B_{e^{-2R}r}(x+\textup{T}_x\mathcal{T})}e_{\epsilon}(u_{\epsilon})\dif y\\\geqslant &c(m)\int_0^1\frac{\dif r}{r}\int_{\lbrace2e^{-2R}\leqslant\vert y-\Pi_{\mathcal{T}}(y)\vert\leqslant5r,r\geqslant\frac{1}{10}\mathfrak{r}_y\rbrace}e_{\epsilon}(u_{\epsilon})\dif y\\\geqslant& c(m)\int_{\mathcal{A}_{\sigma}\cap B_1(p)}e_{\epsilon}(u_{\epsilon})\dif y\int_{3e^{-2R}\vert y-\Pi_{\mathcal{T}}(y)\vert}^{2\vert y-\Pi_{\mathcal{T}}(y)\vert}\frac{\dif r}{r}\\\geqslant&c(m)\int_{\mathcal{A}\cap B_1(p)}e_{\epsilon}(u_{\epsilon}),
        \end{aligned}\nonumber
    \end{equation}
    where in the first inequality we used the fact that on $x\in B_1(p)$ and $\frac{1}{10}\mathfrak{r}_x\leqslant r\leqslant1$ we have $\psi_{\mathcal{T}}(x,r)\geqslant c(m)$. On the other hand, by definition $$\frac{1}{r^{m-2}}\int_{B_{10r}(x)\setminus B_{e^{-2R}r}(x+\mathcal{L}_{x,r})}e_{\epsilon}(u_{\epsilon})\dif y\leqslant C(m)(\bar{\Theta}_{\mathcal{L}}(x,r)+\hat{\Theta}_{\alpha}(x,r)+\hat{\Theta}_n(x,r)).$$
    Hence, using the fact that 
    $$\Vert\Pi_{\textup{T}_x\mathcal{T}}-\Pi_{x,r}\Vert\leqslant C(m,\varepsilon_0)\sqrt{\delta},\textup{dist}(\mathcal{T},\mathcal{T}_r)\leqslant C(m,\varepsilon_0)\sqrt{\delta}r,$$ 
    if $\delta$ is small enough  we have $$\int_{\mathcal{T}}\psi_{\mathcal{T}}(x,r)\dif x\frac{1}{r^{m-2}}\int_{B_{10r}(x)\setminus B_{e^{-2R}r}(x+\textup{T}_x\mathcal{T})}e_{\epsilon}(u_{\epsilon})\dif y\leqslant C(m)(\mathcal{E}_{\mathcal{L}}(\mathcal{T}_r)+\hat{\mathcal{E}}_{\alpha}(\mathcal{T}_r)+\hat{\mathcal{E}}_n(\mathcal{T}_r)).$$
    Integrating over $\dif r/r$ gives the desired result.
\end{proof}
\section{Smallness of energy in annular regions: angular energy}\label{s:angular energy}
In this section we prove the angular part of energy in the annular region is small, namely,
\begin{thm}[cf. {\cite[Theorem 9.1]{naber2024energyidentitystationaryharmonic}}]\label{angular energy estimate}
    Let $u_{\epsilon}:B_{10R}(p)\rightarrow\R^L$ be a solution to (\ref{*}) with $R^{2-m}\int_{B_{10R}(p)}e_{\epsilon}(u_{\epsilon})\leqslant\Lambda$ and $\mathcal{A}=B_2(p)\setminus\overline{B_{\mathfrak{r}_x}(\mathcal{T})}$  be a $\delta$-annular region. If $\delta\leqslant\delta(m,R,K_N,\Lambda)$ and $\epsilon\leqslant\epsilon(m,R,K_N,\Lambda)$, then $$\mathcal{E}_{\alpha}(\mathcal{A})\leqslant C(m,R,\Lambda,K_N)\sqrt{\delta}.$$
\end{thm}
Fix $x\in\mathcal{T}_r$, for $y\in\R^m$, we can use the best approximate affine space $x+\mathcal{L}_{x,r}$ to define the decomposition  $$y-x=y_{x,r}+y_{x,r}^{\perp}=\Pi_{x,r}(y-x)+\Pi_{x,r}^{\perp}(y-x).$$
Moreover, we can consider the polar coordinates in $x+\mathcal{L}_{x,r}^{\perp}$, which reads $y_{x,r}^{\perp}=(s_{x,r}^{\perp},\alpha_{x,r}^{\perp})$. Using this, we define the scale invariant angular energies $$e_{\alpha,x,r}(y)=s^{\perp2}_{x,r}\vert\nabla_{\alpha_{x,r}^{\perp}}u_{\epsilon}(y)\vert^2,$$ 
$$\mathcal{E}_{\alpha,x,r}(y)=s^{\perp2}_{x,r}\int_{\alpha_{x,r}^{\perp}\in S^1}\vert\nabla_{\alpha_{x,r}^{\perp}}u_{\epsilon}(y)\vert^2\dif \alpha_{x,r}^{\perp}.$$
We will show that this energy satisfies a uniform sub-harmonic property with respect to the conformal Laplacian. To be specific, introduce $$\bar{\Delta}_{x,r}=\left(s_{x,r}^{\perp}\frac{\p}{\p s_{x,r}^{\perp}}\right)^2+s^{\perp2}_{x,r}\Delta_{\mathcal{L}_{x,r}}.$$

Note that the condition $\epsilon/r$ is sufficiently small in the following lemma can be ruled out if we assume $\epsilon$ is small.
\begin{lem}\label{angular technical lem1}
    Let $u_{\epsilon}:B_{2}(p)\rightarrow\R^L$ be a solution to (\ref{*}) with $2^{2-m}\int_{B_{2}(p)}e_{\epsilon}(u_{\epsilon})\leqslant\Lambda$. There exists $r_0(m,K_N,\Lambda,\gamma)\in(0,1)$ such that if $x\in B_1(p),0<r<r_0$ satisfies $\bar{\Theta}(x,r)\geqslant\varepsilon_0$, then we have $$\frac{\epsilon}{r}\leqslant \gamma.$$
\end{lem}
\begin{proof}
    Suppose, on the contrary, there exists a sequence $u_{\epsilon_j}:B_{2R}(p)\rightarrow\R^L$ solutions to (\ref{*}) and $x_j\in B_1(p),r_j\rightarrow0$ with $\bar{\Theta}_{u_{\epsilon_j}}(x_j,r_j)\geqslant\varepsilon_0$ but $\epsilon_j\geqslant\gamma_0 r_j$. Let $r_j'\leqslant r_j$ be the greatest number such that $\bar{\Theta}_{u_j}(x_j,r_j')\leqslant\varepsilon_0/2$. Now, let us define the rescaling maps $$u_j(y)=u_{\epsilon_j}(x_j+r_j'y).$$
    After passing to a subsequence, we may assume $u_j\rightharpoonup u$. Note that our assumption above ensures $B_1$ is an $\varepsilon_0$-regularity region for all $u_j$. Hence the convergence on $B_1(p)$ is strong and $\bar{\Theta}_{u_j}(0,1)\geqslant\varepsilon_0/2$ implies $u$ is nontrivial. Now, either $\epsilon_j/r_j\rightarrow\epsilon>0$ or $\epsilon_j/r_j\rightarrow\infty$. Hence $u$ is either nontrivial solution for $\Delta u+f(u)/\epsilon^2=0$ or a nontrivial solution for $\Delta u=0$. Those two cases are both impossible, see for example \cite{Linwangharmonicsphere99}. A contradiction.
\end{proof}
\begin{lem}[cf. {\cite[Lemma 9.2]{naber2024energyidentitystationaryharmonic}}]\label{super convexity angular}
    Let $u_{\epsilon}:B_{10R}(p)\rightarrow\R^L$ be a solution to (\ref{*}) with $R^{2-m}\int_{B_{10R}(p)}e_{\epsilon}(u_{\epsilon})\leqslant\Lambda$ and $\mathcal{A}=B_2(p)\setminus\overline{B_{\mathfrak{r}_x}(\mathcal{T})}$  be a $\delta$-annular region. Let $x\in\mathcal{T}_r$ with $r\geqslant\mathfrak{r}_x$. If $\delta\leqslant\delta(m,R,K_N,\Lambda)$ and $\frac{\epsilon}{r}\leqslant c(m,R,K_N,\Lambda)$, then for each $y\in B_{Rr}(x)$ with $\vert\Pi_{x,r}^{\perp}(y-x)\vert\geqslant e^{-R}r$, we have that $$\bar{\Delta}_{x,r}\mathcal{E}_{\alpha,x,r}(y)\geqslant\left(\frac{3}{2}-C(m,R,K_N)\sqrt{\delta}\right)\mathcal{E}_{\alpha,x,r}(y).$$
\end{lem}
\begin{proof}
    By the quantitative cone splitting on $\mathcal{T}_r$ and the assumption $\vert\Pi_{x,r}^{\perp}(y-x)\vert\geqslant e^{-R}r$, we have 
    $$\left(\frac{1}{2s_{x,r}^{\perp}}\right)^{m-2}\int_{B_{\frac{1}{2}s_{x,r}^{\perp}}(y)}e_{\epsilon}(u_{\epsilon})\leqslant C(m,R)(\bar{\Theta}(x,Rr;L)+\bar{\Theta}(x,Rr;L^{\perp}))\leqslant C(m,R)\delta.$$
    Hence, if $\delta$ is sufficiently small, then $\varepsilon_0$-regularity is valid on $B_{\frac{1}{2}s_{x,r}^{\perp}}(y)$. Direct calculation shows 
    \begin{equation}
        \begin{aligned}
            &\bar{\Delta}_{x,r}\mathcal{E}_{\alpha,x,r}(y)\\=&4s_{x,r}^{\perp2}\int\vert\nabla_{\alpha^{\perp}_{x,r}}u_{\epsilon}\vert^2+10s_{x,r}^{\perp3}\int\left\langle\frac{\p}{\p s_{x,r}^{\perp}}\nabla_{\alpha_{x,r}^{\perp}}u_{\epsilon},\nabla_{\alpha_{x,r}^{\perp}u_{\epsilon}}\right\rangle\\+&2s_{x,r}^{\perp4}\int\left\vert\frac{\p}{\p s_{x,r}^{\perp}}\nabla_{\alpha_{x,r}^{\perp}}u_{\epsilon}\right\vert^2+\vert\nabla^{\mathcal{L}_{x,r}}\nabla_{\alpha^{\perp}_{x,r}}u_{\epsilon}\vert^2+\left\langle\nabla_{\alpha_{x,r}^{\perp}}u_{\epsilon},\left(\frac{\p^2}{\p s_{x,r}^{\perp2}}+\Delta_{x,r}\right)\nabla_{\alpha_{x,r}^{\perp}}u_{\epsilon}\right\rangle\\=&4s_{x,r}^{\perp2}\int(\vert\nabla_{\alpha^{\perp}_{x,r}}u_{\epsilon}\vert^2+\vert\nabla_{\alpha_{x,r}^{\perp}}\nabla_{\alpha_{x,r}^{\perp}}u_{\epsilon}\vert^2)+10s_{x,r}^{\perp3}\int\left\langle\frac{\p}{\p s_{x,r}^{\perp}}\nabla_{\alpha_{x,r}^{\perp}}u_{\epsilon},\nabla_{\alpha_{x,r}^{\perp}u_{\epsilon}}\right\rangle\\+&2s_{x,r}^{\perp4}\int\left\vert\frac{\p}{\p s_{x,r}^{\perp}}\nabla_{\alpha_{x,r}^{\perp}}u_{\epsilon}\right\vert^2+\vert\nabla^{\mathcal{L}_{x,r}}\nabla_{\alpha^{\perp}_{x,r}}u_{\epsilon}\vert^2+\langle\nabla_{\alpha_{x,r}^{\perp}}u_{\epsilon},\Delta\nabla_{\alpha_{x,r}^{\perp}}u_{\epsilon}\rangle\\\geqslant&(6-2\sqrt{5})s_{x,r}^{\perp2}\int\vert\nabla_{\alpha^{\perp}_{x,r}}u_{\epsilon}\vert^2+2s^{\perp4}_{x,r}\int\left\langle\nabla_{\alpha_{x,r}^{\perp}}u_{\epsilon},\nabla_{\alpha_{x,r}^{\perp}}\frac{f(u_{\epsilon})}{\epsilon^2}\right\rangle\\\geqslant&\left(\frac{3}{2}-C(m,K_N)\sqrt{\delta}\right)s_{x,r}^{\perp2}\int\vert\nabla_{\alpha^{\perp}_{x,r}}u_{\epsilon}\vert^2,
        \end{aligned}\nonumber
    \end{equation}
    where in the first inequality we used the Poincare inequality $\int_{S^1}\vert\nabla_{\alpha}\nabla_{\alpha}u\vert^2\dif\alpha\geqslant\int_{S^1}\vert\nabla_{\alpha}u\vert^2\dif\alpha$ and in the second inequality we used the Corollary \ref{second derivative est} and Lemma \ref{angular technical lem1}.
\end{proof}
We also need the following technical estimates for $\mathcal{T}_r$.
\begin{lem}[cf. {\cite[Lemma 9.3]{naber2024energyidentitystationaryharmonic}}]\label{angular technical lem2}
    Let $u_{\epsilon}:B_{10R}(p)\rightarrow\R^L$ be a solution to (\ref{*}) with $R^{2-m}\int_{B_{10R}(p)}e_{\epsilon}(u_{\epsilon})\leqslant\Lambda$ and $\mathcal{A}=B_2(p)\setminus\overline{B_{\mathfrak{r}_x}(\mathcal{T})}$  be a $\delta$-annular region. We have 
    \begin{align}
        &\int_{\mathcal{T}_r}\left(\psi_{\mathcal{T}}(x,r)+r\vert\nabla\psi_{\mathcal{T}}(x,r)\vert+r\left\vert\dot{\psi}_{\mathcal{T}}(x,r)\right\vert+r^2\left\vert\Dot{\psi}_{\mathcal{T}}(x,r)\right\vert\right)(\bar{\Theta}(x,2r;\mathcal{L}_{x,r})+\bar{\Theta}(x,2r;\mathcal{L}_{x,r}^{\perp})\nonumber\\\leqslant&C(m)\int_{\mathcal{T}\cap B_2(p)}r\frac{\p}{\p r}\bar{\Theta}(x,3r).\nonumber
    \end{align}
\end{lem}
\begin{proof}
    By properties of $\psi_{\mathcal{T}}$, Proposition \ref{estimates of T-cutoff}, we have 
    \begin{equation}
        \begin{aligned}
            &\int_{\mathcal{T}_r}\left(\psi_{\mathcal{T}}(x,r)+r\vert\nabla\psi_{\mathcal{T}}(x,r)\vert+r\left\vert\dot{\psi}_{\mathcal{T}}(x,r)\right\vert+r^2\left\vert\Dot{\psi}_{\mathcal{T}}(x,r)\right\vert\right)(\bar{\Theta}(x,2r;\mathcal{L}_{x,r})+\bar{\Theta}(x,2r;\mathcal{L}_{x,r}^{\perp})\\\leqslant&C(m)\int_{\mathcal{T}_r\cap\lbrace r\geqslant\frac{1}{2}\mathfrak{r}_x\rbrace\cap B_1(p)}\bar{\Theta}(x,2r;\mathcal{L}_{x,r})+\bar{\Theta}(x,2r;\mathcal{L}_{x,r}^{\perp})\\\leqslant& C(m)\int_{\mathcal{T}_r\cap\lbrace r\geqslant\frac{1}{2}\mathfrak{r}_x\rbrace\cap B_1(p)}r\frac{\p}{\p r}\bar{\Theta}(x,2.5r)+\fint_{B_r(x)\cap\mathcal{T}_r}\frac{\p}{\p r}\bar{\Theta}(z,2.5r)\\\leqslant&C(m)\int_{\mathcal{T}_r\cap B_{\frac{3}{2}}(p)}r\frac{\p}{\p r}\bar{\Theta}(x,2.5r)\leqslant C(m)\int_{\mathcal{T}\cap B_2(p)}r\frac{\p}{\p r}\bar{\Theta}(x,3r).
        \end{aligned}\nonumber
    \end{equation}
\end{proof}
\begin{lem}[cf. {\cite[Lemma 9.5]{naber2024energyidentitystationaryharmonic}}]\label{angular technical lemma 3}
    Let $u_{\epsilon}:B_{10R}(p)\rightarrow\R^L$ be a solution to (\ref{*}) with $R^{2-m}\int_{B_{10R}(p)}e_{\epsilon}(u_{\epsilon})\leqslant\Lambda$ and $\mathcal{A}=B_2(p)\setminus\overline{B_{\mathfrak{r}_x}(\mathcal{T})}$  be a $\delta$-annular region. Let $x\in\mathcal{T}_r$ and $r\geqslant\frac{1}{2}\mathfrak{r}_x$. Then
    \begin{enumerate}[(1)]
        \item $$r^{-1}\int\left\vert\nabla^x\left(s_{x,r}^{\perp2}\hat{\rho}_r(y-x;\mathcal{L}_{x,r})\right)\right\vert e_{\alpha,x,r}(y)\dif y\leqslant C(m,R)\bar{\Theta}(x,2r;\mathcal{L}_{x,r}^{\perp});$$        
        \item 
        \begin{equation}
            \begin{aligned}
                &\int\left\vert\nabla^x\left(s_{x,r}^{\perp2}\hat{\rho}_r(y-x;\mathcal{L}_{x,r})\right)\right\vert\left\vert\nabla^xe_{\alpha,x,r}(y)\right\vert\dif y\\\leqslant &C(m,R)\sqrt{\delta}(\bar{\Theta}(x,2r;\mathcal{L}_{x,r})+\bar{\Theta}(x,2r;\mathcal{L}_{x,r}^{\perp}));
            \end{aligned}\nonumber
        \end{equation}
        \item 
        \begin{equation}
            \begin{aligned}
                &\int\left(\hat{\rho}_r(y-x;\mathcal{L}_{x,r})+r\left\vert\nabla^y\hat{\rho}_r(y-x;\mathcal{L}_{x,r})\right\vert\right)\left(\left\vert r\frac{\p}{\p r}e_{\alpha,x,r}(y)\right\vert+r\left\vert\nabla^xe_{\alpha,x,r}(y)\right\vert\right)\dif y\\\leqslant& C(m,R)\sqrt{\delta}(\bar{\Theta}(x,2r;\mathcal{L}_{x,r})+\bar{\Theta}(x,2r;\mathcal{L}_{x,r}^{\perp})).
            \end{aligned}\nonumber
        \end{equation}
    \end{enumerate}
\end{lem}
\begin{proof}
    The proof is a straightforward consideration. Since the $x$ and $r$ derivative only hit $\Pi_{x,r}(y-x)$ and $\alpha_{x,r}^{\perp}$ but not $\vert\nabla u_{\epsilon}(y)\vert$ in $e_{\alpha,x,r}(y)$, on the support of $\hat{\rho}_r(y-x;\mathcal{L}_{x,r})$ we can use Theorem \ref{best plane properties} to bound:
    \begin{equation}
        \begin{aligned}
            r\left\vert\frac{\p}{\p r}e_{\alpha,x,r}(y)\right\vert+r\vert\nabla^xe_{\alpha,x,r}(y)\vert&\leqslant C(m,R)\sqrt{\delta}r^2\vert\nabla u_{\epsilon}\vert^2\\&\leqslant C(m,R)\sqrt{\delta}(r^2\vert\Pi_{x,r}\nabla u_{\epsilon}\vert^2+\vert\Pi_{x,r}^{\perp}(y-x)\vert^2\vert\Pi_{x,r}^{\perp}\nabla u_{\epsilon}\vert^2).
        \end{aligned}\nonumber
    \end{equation}
    Similarly $r^{-1}\vert\nabla^x s^{\perp2}_{x,r}\vert\leqslant C(m,R)$. Hence, the estimates on $\hat{\rho}_r(y-x;\mathcal{L}_{x,r})$, Proposition \ref{refine estimates for L-cutoff} yields the desired results.
\end{proof}
Now we can state and prove our estimate on $\hat{\mathcal{E}}_{\alpha}(\mathcal{T}_r)$.
\begin{prop}[cf. {\cite[Proposition 9.6]{naber2024energyidentitystationaryharmonic}}]\label{angular derivative superconvexity}
    Let $u_{\epsilon}:B_{10R}(p)\rightarrow\R^L$ be a solution to (\ref{*}) with $R^{2-m}\int_{B_{10R}(p)}e_{\epsilon}(u_{\epsilon})\leqslant\Lambda$ and $\mathcal{A}=B_2(p)\setminus\overline{B_{\mathfrak{r}_x}(\mathcal{T})}$  be a $\delta$-annular region. Let $x\in\mathcal{T}_r$ with $r\geqslant\mathfrak{r}_x$. If $\delta\leqslant\delta(m,R,K_N,\Lambda)$ and $\epsilon\leqslant\epsilon(m,R,K_N,\Lambda)$, then we have 
    $$\hat{\mathcal{E}}_{\alpha}(\mathcal{T}_r)\leqslant\mathfrak{E}_1(r)+r\frac{\dif}{\dif r}\mathfrak{E}_2(r),$$
    where $\mathfrak{E}_1,\mathfrak{E}_2$ are $C_c^{\infty}$ functions on $(0,\infty)$ and $\mathfrak{E}_1$ satisfies 
    $$\int\vert\mathfrak{E}_1(r)\vert\frac{\dif r}{r}\leqslant C(m,R,K_N,\Lambda)\sqrt{\delta}.$$
\end{prop}
\begin{rmk}
    As calculated in \cite[Proposition 9.6]{naber2024energyidentitystationaryharmonic}, the estimate is in fact of the form 
    $$\hat{\mathcal{E}}_{\alpha}(\mathcal{T}_r)\leqslant\left(r\frac{\dif}{\dif r}\right)^2\hat{\mathcal{E}}_{\alpha}(\mathcal{T}_r)+r\frac{\dif}{\dif r}\mathfrak{E}_2(r)+\mathfrak{E}_1(r).$$
    For simplicity, we do not calculate the first derivative of $\hat{\mathcal{E}}_{\alpha}(\mathcal{T}_r)$. However, the calculation is essentially same as the following discussion. Indeed, calculated as in Lemma \ref{computation angular}, we can get 
    $$r\frac{\dif}{\dif r}\hat{\mathcal{E}}_{\alpha}(\mathcal{T}_r)=\int_{\mathcal{T}_r}\psi_{\mathcal{T}}\int\hat{\rho}_r\nabla^y_{\Pi_{x,r}^{\perp}(y-x)}e_{\alpha,x,r}+\mathfrak{E}_1(r).$$
\end{rmk}
In the following proof, $\mathfrak{E}_1$ may change from line to line but always denotes a term satisfying the conclusion of Proposition \ref{angular derivative superconvexity}. Let us calculate
\begin{lem}[{cf. {\cite[Lemma 9.12]{naber2024energyidentitystationaryharmonic}}}]\label{computation angular}
    \begin{equation}
        \begin{aligned}
            &r\frac{\dif}{\dif r}\int_{\mathcal{T}_r}\psi_{\mathcal{T}}\int\hat{\rho}_r\nabla^y_{\Pi_{x,r}^{\perp}(y-x)}e_{\alpha,x,r}\\=&\int_{\mathcal{T}_r}\psi_{\mathcal{T}}(x,r)\dif x\int\hat{\rho}_r(y-x;\mathcal{L}_{x,r})(\nabla_{\Pi_{x,r}^{\perp}(y-x)}^y)^2e_{\alpha,x,r}(y)\dif y+\mathfrak{E}_1(r).
        \end{aligned}\nonumber
    \end{equation}
\end{lem}
\begin{proof}
    We calculate
    \begin{equation}
        \begin{aligned}
            &r\frac{\dif}{\dif r}\int_{\mathcal{T}_r}\psi_{\mathcal{T}}\int\hat{\rho}_r\nabla^y_{\Pi_{x,r}^{\perp}(y-x)}e_{\alpha,x,r}\\=&\int_{\mathcal{T}_r}r\frac{\p}{\p r}\psi_{\mathcal{T}}\int\hat{\rho}_r\nabla^y_{\Pi_{x,r}^{\perp}(y-x)}e_{\alpha,x,r}+\int_{\mathcal{T}_r}\psi_{\mathcal{T}}\left\langle r\frac{\p}{\p r}\mathcal{T}_r,H_{\mathcal{T}_r}\right\rangle\int\hat{\rho}_r\nabla^y_{\Pi_{x,r}^{\perp}(y-x)}e_{\alpha,x,r}\\+&\int_{\mathcal{T}_r}\psi_{\mathcal{T}}\int r\frac{\p}{\p r}\hat{\rho}_r\nabla^y_{\Pi_{x,r}^{\perp}(y-x)}e_{\alpha,x,r}+\int_{\mathcal{T}_r}\psi_{\mathcal{T}}\int\hat{\rho}_rr\frac{\p}{\p r}\nabla^y_{\Pi_{x,r}^{\perp}(y-x)}e_{\alpha,x,r}\\+&\int_{\mathcal{T}_r}\psi_{\mathcal{T}}\int\rho_r\psi'\frac{1}{r}\left\langle\Pi_{x,r}^{\perp}(y-x),\frac{\p}{\p r}\Pi_{x,r}^{\perp}(y-x)\right\rangle\nabla^y_{\Pi_{x,r}^{\perp}(y-x)}e_{\alpha,x,r},
        \end{aligned}\nonumber
    \end{equation}
where $H_{\mathcal{T}_r}$ denotes the mean curvature of $\mathcal{T}_r$. We divide the rest of proof into two parts
\begin{claim}\label{angular claim 1}
    $$\int_{\mathcal{T}_r}\psi_{\mathcal{T}}\int r\frac{\p}{\p r}\hat{\rho}_r\nabla^y_{\Pi_{x,r}^{\perp}(y-x)}e_{\alpha,x,r}=\int_{\mathcal{T}_r}\psi_{\mathcal{T}}\dif x\int\hat{\rho}_r(\nabla_{\Pi_{x,r}^{\perp}(y-x)}^y)^2e_{\alpha,x,r}(y)\dif y+\mathfrak{E}_1(r).$$
\end{claim}
\begin{proof}[Proof of Claim \ref{angular claim 1}]
    From the definition of $\hat{\rho}$ we have 
    $$r\frac{\p}{\p r}\hat{\rho}_r+m\hat{\rho}_r+\langle\nabla^y\hat{\rho}_r,y-x\rangle=0.$$
    We integrate by parts to get 
    $$-\int\langle\nabla^y\hat{\rho}_r,\Pi_{x,r}^{\perp}(y-x)\rangle\nabla^y_{\Pi_{x,r}^{\perp}(y-x)}e_{\alpha,x,r}=\int\hat{\rho}_r(2\nabla^y_{\Pi_{x,r}^{\perp}(y-x)}e_{\alpha,x,r}+(\nabla^y_{\Pi_{x,r}^{\perp}(y-x)})^2e_{\alpha,x,r}).$$
    To estimate the rest parts, we first exchange the $x$ and $y$ derivative. From the definition $$\hat{\rho}_r(y-x;\mathcal{L}_{x,r})=\frac{1}{r^m}\rho\left(\frac{\vert y-x\vert^2}{2r^2}\right)\psi\left(\frac{\vert\Pi_{x,r}^{\perp}(y-x)\vert^2}{2r^2}\right),$$ 
    taking derivative gives
    \begin{equation}
        \begin{aligned}
            &\langle(\nabla^x+\nabla^y)\hat{\rho}_r(y-x;\mathcal{L}_{x,r}),\Pi_{x,r}(y-x)\rangle\\=&\rho_r(y-x)\psi'\left(\frac{\vert\Pi_{x,r}^{\perp}(y-x)\vert^2}{2r^2}\right)\frac{1}{2r^2}\langle\Pi_{x,r}(y-x),\nabla^x\vert\Pi_{x,r}^{\perp}(y-x)\vert^2\rangle\\=&\rho_r(y-x)\psi'\left(\frac{\vert\Pi_{x,r}^{\perp}(y-x)\vert^2}{2r^2}\right)\frac{1}{r^2}\langle\Pi_{x,r}^{\perp}(y-x),(\nabla_{\Pi_{x,r}(y-x)}\Pi_{x,r}^{\perp})(y-x)\rangle.
        \end{aligned}\nonumber
    \end{equation}
    By the estimate for $\nabla\Pi_{x,r}$, Theorem \ref{best plane properties}, we see that 
    \begin{equation}
        \begin{aligned}
            &\left\vert\nabla^y_{\Pi_{x,r}^{\perp}}\left(\rho_r(y-x)\psi'\left(\frac{\vert\Pi_{x,r}^{\perp}(y-x)\vert^2}{2r^2}\right)\frac{1}{r^2}\langle\Pi_{x,r}^{\perp}(y-x),(\nabla_{\Pi_{x,r}(y-x)}\Pi_{x,r}^{\perp})(y-x)\right)\right\vert\\\leqslant&C(m,R,\varepsilon_0)\sqrt{\delta}\rho_{2r}(y-x).
        \end{aligned}\nonumber
    \end{equation}
    Hence, integrate by parts with respect to $y$, we have 
    \begin{equation}
        \begin{aligned}
            &\left\vert\int_{\mathcal{T}_r}\psi_{\mathcal{T}}\int\langle(\nabla^x+\nabla^y)\hat{\rho}_r(y-x;\mathcal{L}_{x,r}),\Pi_{x,r}(y-x)\rangle\nabla_{\Pi_{x,r}^{\perp}(y-x)}e_{\alpha,x,r}(y)\right\vert\\\leqslant&C(m,R,\varepsilon_0)\sqrt{\delta}\int_{\mathcal{T}_r}\psi_{\mathcal{T}}(x,r)\bar{\Theta}(x,2r;\mathcal{L}_{x,r}^{\perp})\\\leqslant& C(m,R,\varepsilon_0)\sqrt{\delta}\int_{\mathcal{T}\cap B_2(p)}r\frac{\p}{\p r}\bar{\Theta}(x,3r)=\mathfrak{E}_1(r),
        \end{aligned}\nonumber
    \end{equation}
    where we used the Proposition \ref{basic estimate of energy} in the last inequality. Next, we can exchange $\Pi_{x,r}(y-x)$ with $\Pi_{\mathcal{T}_r}(y-x)$ and a small error. Integrate by parts, we have
    \begin{equation}
        \begin{aligned}
            &-\int_{\mathcal{T}_r}\psi_{\mathcal{T}}\int\langle\nabla^x\hat{\rho}_r,\Pi_{x,r}(y-x)-\Pi_{\mathcal{T}_r}(y-x)\rangle\nabla^y_{\Pi^{\perp}_{x,r}(y-x)}e_{\alpha,x,r}(y)\\=&\int_{\mathcal{T}_r}\psi_{\mathcal{T}}\int(\langle\nabla^y_{\Pi_{x,r}^{\perp}(y-x)}\nabla^x\hat{\rho}_r,(\Pi_{x,r}-\Pi_{\mathcal{T}_r})(y-x)\rangle\\+&\langle\nabla^x\hat{\rho}_r,(\Pi_{x,r}-\Pi_{\mathcal{T}_r})(\Pi_{x,r}^{\perp}(y-x))\rangle)e_{\alpha,x,r}(y)\\+&2\int_{\mathcal{T}_r}\psi_{\mathcal{T}}\int\langle\nabla^x\hat{\rho}_r,\Pi_{x,r}(y-x)-\Pi_{\mathcal{T}_r}(y-x)\rangle e_{\alpha,x,r}.
        \end{aligned}\nonumber
    \end{equation}
    Using the fact that $\Pi_{\mathcal{T}_r}$ is close to $\Pi_{x,r}$, Theorem \ref{apporximating submanifold} (6), and Proposition \ref{refine estimates for L-cutoff},
    \begin{equation}
        \begin{aligned}
            \Vert\Pi_{x,r}-\Pi_{\mathcal{T}_r}\Vert&\leqslant C(m,\varepsilon_0,R)\sqrt{\delta},r\vert\nabla^x\hat{\rho}_r(y-x;\mathcal{L}_{x,r})\vert+r^2\vert\nabla^y\nabla^x\hat{\rho}_r(y-x;\mathcal{L}_{x,r})\vert\\&\leqslant C(m)\rho_{2r}(y-x),
        \end{aligned}\nonumber
    \end{equation}
    we see that the above are $\mathfrak{E}_1$ terms.
    Finally, we show that $$\int_{\mathcal{T}_r}\psi_{\mathcal{T}}\int((m-2)\hat{\rho}_r-\langle\nabla^x\hat{\rho}_r,\Pi_{\mathcal{T}_r}(y-x)\rangle)\nabla^y_{\Pi_{x,r}^{\perp}(y-x)}e_{\alpha,x,r}(y)$$ is an $\mathfrak{E}_1$ term. Indeed, integrate by parts with respect to $x$,
    \begin{equation}
        \begin{aligned}
            &\int_{\mathcal{T}_r}\psi_{\mathcal{T}}((m-2)\hat{\rho}_r-\langle\nabla^x\hat{\rho}_r,\Pi_{\mathcal{T}_r}(y-x)\rangle\nabla^y_{\Pi_{x,r}^{\perp}(y-x)}e_{\alpha,x,r}(y)\dif x\\=&\int_{\mathcal{T}_r}\hat{\rho}_r((\langle\nabla^x\psi_{\mathcal{T}},\Pi_{\mathcal{T}_r}(y-x)\rangle+\psi_{\mathcal{T}}\langle y-x,H_{\mathcal{T}_r}\rangle)\nabla^y_{\Pi^{\perp}_{x,r}(y-x)}e_{\alpha,x,r}(y)\\&+\psi_{\mathcal{T}}\langle\nabla^x\nabla^y_{\Pi_{x,r}^{\perp}(y-x)}e_{\alpha,x,r}(y),\Pi_{\mathcal{T}_r}(y-x)\rangle)\dif x
        \end{aligned}\label{Claim 7.8.1}
    \end{equation}
    where $H_{\mathcal{T}_r}$ is the mean curvature of $\mathcal{T}_r$. The first term in (\ref{Claim 7.8.1}) can be controlled by 
    \begin{equation}
        \begin{aligned}
            &\int_{\mathcal{T}_r}\int\hat{\rho}_r\langle\nabla^x\psi_{\mathcal{T}},\Pi_{\mathcal{T}_r}(y-x)\rangle\nabla^y_{\Pi^{\perp}_{x,r}(y-x)}e_{\alpha,x,r}(y)\\=&\int_{\mathcal{T}_r}\int(\nabla^y_{\Pi_{x,r}^{\perp}(y-x)}\hat{\rho}_r\langle\nabla^x\psi_{\mathcal{T}},\Pi_{\mathcal{T}_r}(y-x)\rangle+\hat{\rho}_r\langle\nabla^x\psi_{\mathcal{T}},\Pi_{\mathcal{T}_r}(\Pi_{x,r}^{\perp}(y-x))\rangle)e_{\alpha,x,r}(y)\\+&2\int_{\mathcal{T}_r}\int\hat{\rho}_r\langle\nabla^x\psi_{\mathcal{T}},\Pi_{\mathcal{T}_r}(y-x)\rangle e_{\alpha,x,r}(y).
        \end{aligned}\nonumber
    \end{equation}
    Using
    $$\left\vert\nabla^y_{\Pi_{x,r}^{\perp}(y-x)}\hat{\rho}_r\right\vert\leqslant C(m)\rho_{2r}(y-x),\vert\langle\nabla^x\psi_{\mathcal{T}},\Pi_{\mathcal{T}_r}(y-x)\rangle\vert\leqslant C(m,R)r\vert\nabla\psi\vert,$$
    we have that 
    \begin{equation}
        \begin{aligned}
            \left\vert\int_{\mathcal{T}_r}\int\hat{\rho}_r\langle\nabla^x\psi_{\mathcal{T}},\Pi_{\mathcal{T}_r}(y-x)\rangle\nabla^y_{\Pi^{\perp}_{x,r}(y-x)}e_{\alpha,x,r}(y)\right\vert&\leqslant C(m,R)\int_{\mathcal{T}_r}r\vert\nabla\psi_{\mathcal{T}}\vert\bar{\Theta}(x,2r;\mathcal{L}_{x,r}^{\perp})\\&\leqslant C(m,R)\sqrt{\delta}\int_{\mathcal{T}_r}r\vert\nabla\psi_{\mathcal{T}}\vert
        \end{aligned}\nonumber
    \end{equation}
    is an $\mathfrak{E}_1$ term by the estimate of $r\vert\nabla\psi_{\mathcal{T}}\vert$, Proposition \ref{estimates of T-cutoff}. The second term in (\ref{Claim 7.8.1}) can be controlled by 
    $$\left\vert\int_{\mathcal{T}_r}\psi_{\mathcal{T}}\langle H_{\mathcal{T}_r},y-x\rangle\int\hat{\rho}_r\nabla^y_{\Pi_{x,r}^{\perp}(y-x)}e_{\alpha,x,r}(y)\right\vert\leqslant C(m,R,\varepsilon_0)\sqrt{\delta}\int_{\mathcal{T}_r}\psi_{\mathcal{T}}\bar{\Theta}(x,2r;\mathcal{L}_{x,r}^{\perp})=\mathfrak{E}_1(r).$$
    Finally, the third term in (\ref{Claim 7.8.1}) can be written as 
    \begin{equation}
        \begin{aligned}
            &\int_{\mathcal{T}_r}\psi_{\mathcal{T}}\int\hat{\rho}_r\left\langle\nabla^x\nabla^y_{\Pi_{x,r}^{\perp}(y-x)}e_{\alpha,x,r}(y),\Pi_{\mathcal{T}_r}(y-x)\right\rangle\\=&\int_{\mathcal{T}_r}\psi_{\mathcal{T}}\int\hat{\rho}_r\bigg(\nabla_{\Pi_{x,r}^{\perp}(y-x)}^y\nabla^x_{\Pi_{\mathcal{T}_r}(y-x)}e_{\alpha,x,r}+\left\langle\nabla^ye_{\alpha,x,r},\nabla_{\Pi_{\mathcal{T}_r}(y-x)}^x\Pi_{x,r}^{\perp}(y-x)\right\rangle\\&-\nabla^x_{\Pi_{\mathcal{T}_r}\Pi_{x,r}^{\perp}(y-x)}e_{\alpha,x,r}\bigg).
        \end{aligned}\nonumber
    \end{equation}
    We use the Lemma \ref{angular technical lemma 3} to bound 
    \begin{equation}
        \begin{aligned}
        &\left\vert\int_{\mathcal{T}_r}\psi_{\mathcal{T}}\int\hat{\rho}_r\left(\nabla_{\Pi_{x,r}^{\perp}(y-x)}^y\nabla^x_{\Pi_{\mathcal{T}_r}(y-x)}e_{\alpha,x,r}-\nabla^x_{\Pi_{\mathcal{T}_r}\Pi_{x,r}^{\perp}(y-x)}e_{\alpha,x,r}\right)\right\vert\\=&\left\vert\int_{\mathcal{T}_r}\psi_{\mathcal{T}}\int\left(\nabla^y_{\Pi_{x,r}^{\perp}(y-x)}\hat{\rho}_r\nabla^x_{\Pi_{\mathcal{T}_r}(y-x)}e_{\alpha,x,r}+\hat{\rho}_r\nabla^x_{\Pi_{\mathcal{T}_r}\Pi_{x,r}^{\perp}(y-x)}e_{\alpha,x,r}\right)\right\vert\\\leqslant&C(m,R)\int_{\mathcal{T}_r}\psi_{\mathcal{T}}\int\left(\hat{\rho}_r+r\vert\nabla^y\hat{\rho}_r\vert\right)r\vert\nabla^xe_{\alpha,x,r}\vert\\\leqslant &C(m,R)\sqrt{\delta}\int_{\mathcal{T}_r}\psi_{\mathcal{T}}(\bar{\Theta}(x,2r;\mathcal{L}_{x,r})+\bar{\Theta}(x,2r;\mathcal{L}_{x,r}^{\perp}))=\mathfrak{E}_1(r).
        \end{aligned}\nonumber
    \end{equation}
    Also, using 
    $$\left\vert\nabla^x_{\Pi_{\mathcal{T}_r}(y-x)}\Pi_{x,r}^{\perp}(y-x)\right\vert=\left\vert\Pi_{x,r}^{\perp}\Pi_{\mathcal{T}_r}(y-x)+(\nabla^x_{\Pi_{\mathcal{T}_r}(y-x)}\Pi_{x,r}^{\perp})(y-x)\right\vert\leqslant C(m,\varepsilon_0,R)r\sqrt{\delta},$$ $$\left\vert\textup{div}^y\nabla^x_{\Pi_{\mathcal{T}_r}(y-x)}\Pi_{x,r}^{\perp}(y-x)\right\vert\leqslant C(m)\left\vert\nabla^x\Pi_{x,r}(y-x)\right\vert\leqslant C(m,\varepsilon_0,R)\sqrt{\delta}.$$
    We have 
    \begin{equation}
        \begin{aligned}
            &\left\vert\int_{\mathcal{T}_r}\psi_{\mathcal{T}}\int\hat{\rho}_r\langle\nabla^ye_{\alpha,x,r},\nabla_{\Pi_{\mathcal{T}_r}(y-x)}^x\Pi_{x,r}^{\perp}(y-x)\rangle\right\vert\\\leqslant &C(m,\varepsilon_0,R)\sqrt{\delta}\int_{\mathcal{T}_r}\psi_{\mathcal{T}}\int(\hat{\rho}_r+r\vert\nabla^y\hat{\rho}_r\vert)e_{\alpha,x,r}=\mathfrak{E}_1(r).
        \end{aligned}\nonumber
    \end{equation}
    These estimates finish the proof of Claim \ref{angular claim 1}.
\end{proof}
\begin{claim}\label{angular claim 2}
    \begin{equation}
        \begin{aligned}
            &\int_{\mathcal{T}_r}r\frac{\p}{\p r}\psi_{\mathcal{T}}\int\hat{\rho}_r\nabla^y_{\Pi_{x,r}^{\perp}(y-x)}e_{\alpha,x,r}+\int_{\mathcal{T}_r}\psi_{\mathcal{T}}\left\langle r\frac{\p}{\p r}\mathcal{T}_r,H_{\mathcal{T}_r}\right\rangle\int\hat{\rho}_r\nabla^y_{\Pi_{x,r}^{\perp}(y-x)}e_{\alpha,x,r}\\+&\int_{\mathcal{T}_r}\psi_{\mathcal{T}}\int\hat{\rho}_r\frac{\p}{\p r}\nabla^y_{\Pi_{x,r}^{\perp}(y-x)}e_{\alpha,x,r}\\+&\int_{\mathcal{T}_r}\psi_{\mathcal{T}}\int\rho_r\psi'\frac{1}{r}\left\langle\Pi_{x,r}^{\perp}(y-x),\frac{\p}{\p r}\Pi_{x,r}^{\perp}(y-x)\right\rangle\nabla^y_{\Pi_{x,r}^{\perp}(y-x)}e_{\alpha,x,r}=\mathfrak{E}_1(r).
        \end{aligned}\nonumber
    \end{equation}
\end{claim}
\begin{proof}[Proof of Claim \ref{angular claim 2}]
    We analyze term by term. Using arguments similar to those in the proof of Claim \ref{angular claim 1}, 
    \begin{equation}
        \begin{aligned}
            \left\vert\int_{\mathcal{T}_r}r\frac{\p}{\p r}\psi_{\mathcal{T}}\int\hat{\rho}_r\nabla^y_{\Pi_{x,r}^{\perp}(y-x)}e_{\alpha,x,r}\right\vert&\leqslant C(m,R)\int r\left\vert\frac{\p }{\p r}\psi_{\mathcal{T}}(x,r)\right\vert\bar{\Theta}(x,2r;\mathcal{L}_{x,r}^{\perp})\\&\leqslant C(m,\varepsilon_0,R)\sqrt{\delta}\int_{\mathcal{T}_r}r\left\vert\frac{\p}{\p r}\psi_{\mathcal{T}}(x,r)\right\vert=\mathfrak{E}_1(r).
        \end{aligned}\nonumber
    \end{equation}
    \begin{equation}
        \begin{aligned}
            &\left\vert\int_{\mathcal{T}_r}\psi_{\mathcal{T}}\int\hat{\rho}_rr\frac{\p}{\p r}\nabla^y_{\Pi_{x,r}^{\perp}(y-x)}e_{\alpha,x,r}\right\vert\\=&\left\vert\int_{\mathcal{T}_r}\psi_{\mathcal{T}}\int\left(\hat{\rho}_r\nabla^y_{\Pi_{x,r}^{\perp}(y-x)}\frac{\p}{\p r}e_{\alpha,x,r}+\hat{\rho}_r\nabla^y_{\p_r\Pi_{x,r}^{\perp}(y-x)}e_{\alpha,x,r}\right)\right\vert\\\leqslant& C(m,\varepsilon_0,R)\int_{\mathcal{T}_r}\psi_{\mathcal{T}}\int\left(r\vert\nabla^y\hat{\rho}_r\vert r\left\vert\frac{\p}{\p r}e_{\alpha,x,r}\right\vert+\sqrt{\delta}r\vert\nabla^y\hat{\rho}_r\vert e_{\alpha,x,r}\right)\\\leqslant &C(m,R,\varepsilon_0)\sqrt{\delta}\int_{\mathcal{T}_r}\psi_{\mathcal{T}}(\bar{\Theta}(x,2r;\mathcal{L}_{x,r})+\bar{\Theta}(x,2r;\mathcal{L}_{x,r}^{\perp}))=\mathfrak{E}_1(r).
        \end{aligned}\nonumber
    \end{equation}
    By the properties of best approximation submanifold, Theorem \ref{apporximating submanifold}, $$r\vert\nabla^2\mathfrak{t}_r\vert+\left\vert\frac{\p}{\p r}\mathfrak{t}_r\right\vert\leqslant C(m,\varepsilon_0)\sqrt{\delta}.$$
    we have (note that the velocity vector $\p_r\mathcal{T}_r$ is simply $\p_r\mathfrak{t}_r$ evaluating at the corresponding point.) $$\left\vert\int_{\mathcal{T}_r}\psi_{\mathcal{T}}\left\langle r\frac{\p }{\p r}\mathcal{T}_r,H_{\mathcal{T}_r}\right\rangle\int\hat{\rho}_r\nabla^y_{\Pi_{x,r}^{\perp}(y-x)}e_{\alpha,x,r}\right\vert\leqslant C(m,\varepsilon_0,R)\delta\int_{\mathcal{T}_r}\psi_{\mathcal{T}}\bar{\Theta}(x,2r;\mathcal{L}_{x,r}^{\perp})=\mathfrak{E}_1(r).$$
    Finally
    \begin{equation}
        \begin{aligned}
            &\left\vert\int_{\mathcal{T}_r}\psi_{\mathcal{T}}\int\rho_r\psi'\frac{1}{r}\left\langle\Pi_{x,r}^{\perp}(y-x),\frac{\p}{\p r}\Pi_{x,r}^{\perp}(y-x)\right\rangle\nabla^y_{\Pi_{x,r}^{\perp}(y-x)}e_{\alpha,x,r}\right\vert\\=&\left\vert\int_{\mathcal{T}_r}\psi_{\mathcal{T}}\int(4\rho_r\psi'+\nabla^y_{\Pi_{x,r}^{\perp}(y-x)}(\rho_r\psi'))\left\langle\frac{\Pi_{x,r}^{\perp}(y-x)}{r},\frac{\p}{\p r}\Pi_{x,r}^{\perp}(y-x)\right\rangle e_{\alpha,x,r}\right\vert\\\leqslant& C(m,\varepsilon_0,R)\sqrt{\delta}\int_{\mathcal{T}_r}\psi_{\mathcal{T}}\bar{\Theta}(x,2r;\mathcal{L}_{x,r})=\mathfrak{E}_1(r).
        \end{aligned}\nonumber
    \end{equation}
    We have finished proof of Claim \ref{angular claim 2}, and hence Lemma \ref{computation angular} is completed.
\end{proof}
\end{proof}
In order to apply Lemma \ref{super convexity angular}, we need to calculate two parts of derivatives: the $\left(s_{x,r}^{\perp}\frac{\p}{\p s_{x,r}^{\perp}}\right)^2$ and $\Delta_{\mathcal{L}_{x,r}}$. From Lemma \ref{computation angular}, we have seen the radial derivative because $\nabla_{\Pi_{x,r}^{\perp}(y-x)}=s_{x,r}^{\perp}\frac{\p}{\p s_{x,r}^{\perp}}$. Next, we wish to recover the $\Delta_{\mathcal{L}_{x,r}}$ in Lemma \ref{super convexity angular}. Indeed, this term is small:
\begin{lem}[{cf. \cite[Lemma 9.13]{naber2024energyidentitystationaryharmonic}}]
    $$\int_{\mathcal{T}_r}\psi_{\mathcal{T}}(x,r)\int\hat{\rho}_r(y-x;\mathcal{L}_{x,r})\left\vert\Pi_{x,r}^{\perp}(y-x)\right\vert^2\Delta_{\mathcal{L}_{x,r}}e_{\alpha,x,r}(y)=\mathfrak{E}_1(r).$$
\end{lem}
\begin{proof}
    The idea is similar to before. First, use Theorem \ref{best plane properties}, $r\vert\nabla\Pi_{x,r}\vert+r^2\vert\nabla^2\Pi_{x,r}\vert\leqslant C(m,\varepsilon_0)\sqrt{\delta}$ and a direct but tedious calculation, we can exchange $x$ and $y$ derivative $$\left\vert(\Delta_{\mathcal{L}_{x,r}}^x-\Delta_{\mathcal{L}_{x,r}}^y)\left(\hat{\rho}_r(y-x;\mathcal{L}_{x,r})\left\vert\Pi_{x,r}^{\perp}(y-x)\right\vert^2\right)\right\vert\leqslant C(m,\varepsilon_0,R)\sqrt{\delta}\rho_{2r}(y-x).$$
    Next, use the rough bounds $\Vert\Pi_{x,r}-\Pi_{\mathcal{T}_r}\Vert\leqslant C(m,\varepsilon_0)\sqrt{\delta}$ and $\vert\nabla\mathfrak{t}_r\vert+r\vert\nabla^2\mathfrak{t}_r\vert\leqslant C(m,\varepsilon_0)\sqrt{\delta}$, we can calculate 
    $$\left\vert(\Delta_{\mathcal{L}_{x,r}}^x-\Delta_{\mathcal{T}_r}^x)\left(\hat{\rho}_r(y-x;\mathcal{L}_{x,r})\left\vert\Pi_{x,r}^{\perp}(y-x)\right\vert^2\right)\right\vert\leqslant C(m,\varepsilon_0,R)\sqrt{\delta}\rho_{2r}(y-x).$$
    Hence 
    $$\int_{\mathcal{T}_r}\psi_{\mathcal{T}}\int(\Delta^x_{\mathcal{T}_r}(\hat{\rho}_r\vert\Pi_{x,r}(y-x)\vert^2)e_{\alpha,x,r}+\hat{\rho}_r\vert\Pi_{x,r}^{\perp}(y-x)\vert^2\Delta_{\mathcal{L}_{x,r}}^ye_{\alpha,x,r})=\mathfrak{E}_1(r).$$
    Now we can integrate by parts with respect to $x$ to get 
    \begin{equation}
        \begin{aligned}
            &\int_{\mathcal{T}_r}\psi_{\mathcal{T}}\Delta_{\mathcal{T}_r}(\vert\Pi_{x,r}^{\perp}(y-x)\vert^2\hat{\rho}_r)e_{\alpha,x,r}\\=&-\int_{\mathcal{T}_r}\langle\nabla^x_{\mathcal{T}_r}\psi_{\mathcal{T}},\nabla^x_{\mathcal{T}}(\vert\Pi_{x,r}^{\perp}(y-x)\vert^2\hat{\rho}_r)e_{\alpha,x,r}+\psi_{\mathcal{T}}\langle\nabla^x_{\mathcal{T}}(\vert\Pi_{x,r}^{\perp}(y-x)\vert^2\hat{\rho}_r),\nabla_{\mathcal{T}_r}^xe_{\alpha,x,r}\rangle.
        \end{aligned}\nonumber
    \end{equation}
    We can estimate 
    \begin{equation}
        \begin{aligned}
            &\left\vert\int\int_{\mathcal{T}_r}\langle\nabla^x_{\mathcal{T}_r}\psi_{\mathcal{T}},\nabla^x_{\mathcal{T}}(\vert\Pi_{x,r}^{\perp}(y-x)\vert^2\hat{\rho}_r)e_{\alpha,x,r}\dif x\dif y\right\vert\\\leqslant&C(m,R)\int_{\mathcal{T}_r}r\vert\nabla\psi_{\mathcal{T}}\vert\bar{\Theta}(x,2r;\mathcal{L}_{x,r}^{\perp})\leqslant C(m,R,\varepsilon_0)\sqrt{\delta}\int_{\mathcal{T}_r}r\vert\nabla\psi_{\mathcal{T}}\vert=\mathfrak{E}_1(r).
        \end{aligned}\nonumber
    \end{equation}
    By Lemma \ref{angular technical lemma 3} above 
    \begin{equation}
        \begin{aligned}
            &\left\vert\int_{\mathcal{T}_r}\psi_{\mathcal{T}}\int\psi_{\mathcal{T}}\langle\nabla^x_{\mathcal{T}}(\vert\Pi_{x,r}^{\perp}(y-x)\vert^2\hat{\rho}_r),\nabla_{\mathcal{T}_r}^xe_{\alpha,x,r}\rangle\right\vert\\\leqslant&C(m,R,\varepsilon_0)\sqrt{\delta}\int_{\mathcal{T}_r}\psi_{\mathcal{T}}(\bar{\Theta}(x,2r;\mathcal{L}_{x,r})+\bar{\Theta}(x,2r;\mathcal{L}_{x,r}^{\perp}))=\mathfrak{E}_1(r).
        \end{aligned}\nonumber
    \end{equation}
    We have got the desired estimates.
\end{proof}
Now we can finish the proof of Proposition \ref{angular derivative superconvexity} and the main theorem in this section.
\begin{proof}[Proof of Proposition \ref{angular derivative superconvexity}]
    If we set $\mathfrak{E}_2(r)=\int_{\mathcal{T}_r}\psi_{\mathcal{T}}\int\hat{\rho}_r\nabla^y_{\Pi_{x,r}^{\perp}(y-x)}e_{\alpha,x,r}$, then Lemma \ref{super convexity angular} and \ref{computation angular} yields
    \begin{equation}
        \begin{aligned}
            \mathfrak{E}_1(r)+r\frac{\dif}{\dif r}\mathfrak{E}_2(r)&=\int_{\mathcal{T}_r}\psi_{\mathcal{T}}\int\hat{\rho}_r\left(\left(s_{x,r}^{\perp2}\frac{\p}{\p s_{x,r}^{\perp}}\right)^2+s_{x,r}^{\perp^2}\Delta_{\mathcal{L}_{x,r}}\right)e_{\alpha,x,r}\\&\geqslant \left(\frac{3}{2}-C(m,R,K_N)\sqrt{\delta}\right)\int_{\mathcal{T}_r}\int\psi_{\mathcal{T}}\int\hat{\rho}_re_{\alpha,x,r}\geqslant\hat{\mathcal{E}}_{\alpha}(\mathcal{T}_r).
        \end{aligned}\nonumber
    \end{equation}
    provided $\delta$ is small enough.
\end{proof}
\begin{proof}[Proof of Theorem \ref{angular energy estimate}]
    The proof follows easily from Proposition \ref{angular derivative superconvexity}. Indeed 
    $$\mathcal{E}_{\alpha}(\mathcal{A})=\int\mathcal{E}_{\alpha}(\mathcal{T}_r)\frac{\dif r}{r}=\int\left(\mathfrak{E}_1(r)+r\frac{\dif}{\dif r}\mathfrak{E}_2(r)\right)\frac{\dif r}{r}=\int\mathfrak{E}_1(r)\frac{\dif r}{r}\leqslant C(m,R,K_N,\varepsilon_0,\Lambda)\sqrt{\delta}.$$
\end{proof}

\section{Smallness of energy in annular regions: tangential and radial energy}\label{s:radial energy}
We finally arrive at the last piece of our theorem. Namely, we wish to show how to control the $\mathcal{L}$-energy and radial energy in our annular region. We have the following
\begin{thm}[{cf. \cite[Lemma 10.1]{naber2024energyidentitystationaryharmonic}}]\label{radial and tangential energy estimate}
    Let $u_{\epsilon}:B_{2R}(p)\rightarrow\R^L$ be a solution to (\ref{*}) with $R^{2-m}\int_{B_{2R}(p)}e_{\epsilon}(u_{\epsilon})\leqslant\Lambda$. Suppose the $(m-2)$-submanifold $\mathcal{T}\subset B_2(p)$ is such that $\mathcal{A}=B_2(p)\setminus\overline{B_{\mathfrak{r}_x}(\mathcal{T})}$ is a $\delta$-annular region. For each $\sigma>0$ if $R\geqslant R(m,\Lambda,\sigma),\delta<\delta(m,K_N,\Lambda,\sigma),\epsilon<\epsilon(m,K_N,\Lambda,\sigma)$, then 
    $$\mathcal{E}_{\mathcal{L}}(\mathcal{A})+\mathcal{E}_m(\mathcal{A})\leqslant\sigma.$$
\end{thm}
The proof of Theorem \ref{radial and tangential energy estimate} above is done by analyzing the $\mathcal{L}$-energy and radial energy on best approximating submanifold $\mathcal{T}_r$. Namely,
\begin{prop}[{cf. \cite[Lemma 10.3]{naber2024energyidentitystationaryharmonic}}]\label{local estimate for tangential and radial energy}
    Let $u_{\epsilon}:B_{2R}(p)\rightarrow\R^L$ be a solution to (\ref{*}) with $R^{2-m}\int_{B_{2R}(p)}e_{\epsilon}(u_{\epsilon})\leqslant\Lambda$. Suppose the $(m-2)$-submanifold $\mathcal{T}\subset B_2(p)$ is such that $\mathcal{A}=B_2(p)\setminus\overline{B_{\mathfrak{r}_x}(\mathcal{T})}$ is a $\delta$-annular region. If $\delta<\delta(m,R,K_N,\Lambda,\varepsilon_0)$, then the following holds $$\hat{\mathcal{E}}_n(\mathcal{T}_r)+2\mathcal{E}_{\mathcal{L}}(\mathcal{T}_r)\leqslant (m-1)r\frac{\dif}{\dif r}\mathcal{E}_{\mathcal{L}}(\mathcal{T}_r)+\hat{\mathcal{E}}_{\alpha}(\mathcal{T}_r)+\mathfrak{F}(r).$$
    where $\mathfrak{F}(r)$ satisfies 
    \[
    \begin{split}
        \vert\mathfrak{F}(r)\vert & \leqslant C(m,K_N,\Lambda,R)\bigg( \left(\sqrt{\delta}+e^{-R/2}\right)\int_{\mathcal{T}_r}\psi_{\mathcal{T}}\left(\bar{\Theta}_{\mathcal{L}}(x,2r)+\frac{\p}{\p r}\bar{\Theta}(x,2r)\right) \\
        & +\sqrt{\delta}\int_{\mathcal{T}_r}r\left\vert\frac{\p}{\p r}\psi_{\mathcal{T}}\right\vert + r \vert\nabla\psi_{\mathcal{T}}\vert\bigg).
    \end{split}
    \]
    It is worth pointing out that the dependence of $R$ in the large constant $C$ is a certain power, and hence can be controlled by the exponential $e^{-R/2}$.
\end{prop}
\begin{rmk}
    Use the argument in \cite[Section 4]{naber2024energyidentitystationaryharmonic}, it is possible to show the estimate of the form
    $$\mathcal{E}_{\mathcal{L}}(\mathcal{T}_r)\leqslant C(m)\min\lbrace \delta,\sigma/\vert\ln r\vert\rbrace.$$
    However, we do not need such a fact here.
\end{rmk}
Let us recall the discussion in the introduction of this paper. In particular, let us recall the equation (\ref{almost conformality}). We begin by choosing $\xi=\phi(y-x)\Pi_{L^{\perp}}(y-x)$ in the stationary equation, where $L$ is an arbitrary $m-2$ dimensional subspace of $\R^m$ and $\phi$ is a $C_c^{\infty}$ function. We get 
$$\int\phi\left(\vert\Pi_L\nabla u_{\epsilon}\vert^2+2\frac{F(u_{\epsilon})}{\epsilon^2}\right)+e_{\epsilon}(u_{\epsilon})\langle\nabla\phi(y-x),\Pi_{L^{\perp}}(y-x)\rangle=\int\langle \nabla_{\nabla\phi}u_{\epsilon},\nabla_{\Pi_{L^{\perp}(y-x)}}u_{\epsilon}\rangle.$$
In order to see the $L^{\perp}$ derivatives, take $\phi$ to be a function with $\Pi_{L^{\perp}}\nabla\phi(y)=\varphi(y)\Pi_{L^{\perp}}(y)$, we get 
\begin{equation}
    \begin{aligned}
        &\int\phi(y-x)\left(\vert\Pi_L\nabla u_{\epsilon}\vert^2+2\frac{F(u_{\epsilon})}{\epsilon^2}\right)+\varphi(y-x)\frac{1}{2}\left(\vert\Pi_L\nabla u_{\epsilon}\vert^2+2\frac{F(u_{\epsilon})}{\epsilon^2}\right)\\=&\frac{1}{2}\int\varphi(y-x)\left\vert\Pi_{L^{\perp}}(y-x)\right\vert^2(\vert\nabla_{n_{L^{\perp}}}u_{\epsilon}\vert^2-\vert\nabla_{\alpha_{L^{\perp}}}u_{\epsilon}\vert^2)+\langle\nabla_{\Pi_L\nabla\phi}u_{\epsilon},\nabla_{\Pi_{L^{\perp}}(y-x)}u_{\epsilon}\rangle.
    \end{aligned}\nonumber
\end{equation}
In above calculation, we see that it is possible to control the radial energy using the angular energy. In order to see the heat mollified energy along $\mathcal{T}_r$, we set $\tilde{\rho}_r(y,L)$ to be a compactly supported function with $$\Pi_{L^{\perp}}\nabla\tilde{\rho}_r(y,L)=-\frac{\Pi_{L^{\perp}}(y)}{r^2}\hat{\rho}_r(y;L).$$
Note that by the compactly supportedness of $\hat{\rho}_r$, we can explicitly write $\tilde{\rho}_r$ as \begin{equation}
    \tilde{\rho}_r(y;L)=\frac{1}{r^m}\int_{\Pi_{L^{\perp}}(y)}^{\infty}\psi\left(\frac{s^2}{2r^2}\right)\rho\left(\frac{s^2+\left\vert\Pi_L(y)\right\vert^2}{2r^2}\right)\frac{s\dif s}{r^2}.\label{expression of intergration of L-cut off}
\end{equation}
To see (\ref{expression of intergration of L-cut off}), simply take some $y'$ such that $\vert y'\vert$ is large enough, $y'-y\in L^{\perp}$ with $\langle y'-y,y\rangle=0$ and then use the integrate formula  $$\tilde{\rho}_r(y;L)-\tilde{\rho}_r(y';L)=-\int_0^1\frac{\dif}{\dif t}\tilde{\rho}_r(ty'+(1-t)y;L)\dif t.$$
Using formula (\ref{expression of intergration of L-cut off}), we can easily get the following estimates for $\tilde{\rho}_r$.
\begin{lem}[{cf. \cite[Lemma 2.32]{naber2024energyidentitystationaryharmonic}}]\label{estimate for tide rho}
    We have $\nabla_{L^{\perp}}\tilde{\rho}_r(y;L)=0$ for each $y\in B_{re^{-R/2}}(L)$. Moreover, the difference between $\rho_r$ and $\tilde{\rho}_r$, $e_r(y;L)=\rho_r(y)-\tilde{\rho}_r(y;L)$ satisfies 
    $$\vert e_r(y;L)\vert+r\vert\nabla e_r(y;L)\vert\leqslant C(m)e^{-R/2}\rho_{2r}(y).$$
\end{lem}
\begin{proof}
    The first claim is clear by definition. For the estimates of $e_r$, we write it as 
    \begin{equation}
        \begin{aligned}
            e_r(y;L)&=\frac{1}{r^m}\int_{\vert\Pi_{L^{\perp}}(y)\vert}^{2re^{-R}}\left(1-\psi\left(\frac{s^2}{2r^2}\right)\right)\rho\left(\frac{s^2+\vert\Pi_L(y)\vert^2}{2r^2}\right)\frac{s\dif s}{r^2}\\&-\frac{1}{r^m}\int_{\vert y\vert}^\infty\left(\rho\left(\frac{s^2}{2r^2}\right)+\dot{\rho}\left(\frac{s^2}{2r^2}\right)\right)\frac{s\dif s}{r^2}\\&:=f_1(y)+f_2(y).
        \end{aligned}\nonumber
    \end{equation}
    Note that $f_1(y)$ satisfies $\vert f_1(y)\vert\leqslant 4r^{-m}e^{-2R}\chi_{B_{2rR}}$. Hence it is clear that $$\vert f_1(y)\vert\leqslant C(m)e^{-R}\rho_{2r}(y)$$ by Lemma \ref{properties of heat mollifier}. Also 
    $$\vert f_2(y)\vert\leqslant\frac{C(m)}{r^me^{-R/2}}\int_{\vert y\vert}^{\infty}\rho\left(\frac{s^2}{8r^2}\right)\frac{s\dif s}{r^2}\leqslant-\frac{C(m)}{r^me^{-R/2}}\int_{\vert y\vert}^{\infty}\dot{\rho}\left(\frac{s^2}{8r^2}\right)\frac{s\dif s}{r^2}= C(m)e^{-R/2}\rho_{2r}(y).$$
    For the derivative, we have 
    $$\nabla_{L}f_1(y)=\frac{\Pi_L(y)}{r^m}\int_{\vert\Pi_{L^{\perp}}(y)\vert}^{2re^{-R}}\left(1-\psi\left(\frac{s^2}{2r^2}\right)\right)\dot{\rho}\left(\frac{s^2+\vert\Pi_L(y)\vert^2}{2r^2}\right)\frac{s\dif s}{r^4}.$$
    For a similar reason to that of $\vert f_1(y)\vert$, we have 
    $$r\vert\nabla_L f_1(y)\vert\leqslant C(m)Re^{-R}\rho_{2r}(y)\leqslant C(m)e^{-R/2}\rho_{2r}(y).$$
    On the other hand 
    $$\nabla_{L^{\perp}}f_1(y)=-\frac{\Pi_{L^{\perp}}(y)}{r^{m+2}}\left(1-\psi\left(\frac{\vert\Pi_{L^{\perp}}(y)\vert^2}{2r^2}\right)\right)\rho\left(\frac{\vert y\vert^2}{2r^2}\right).$$
    $f_1(y)$ vanishes for $\vert\Pi_{L^{\perp}}(y)\vert\geqslant2re^{-R}$. Hence,
    $$r\vert\nabla_{L^{\perp}}f_1(y)\vert\leqslant 2e^{-R}\rho_r(y)\leqslant 2e^{-R}\rho_{2r}(y).$$
    Finally,
    $$\nabla f_2(y)=-r^{-2}(\rho_r(y)+\dot{\rho}_r(y))y,r\vert\nabla f_2(y)\vert\leqslant C(m)Re^{-R}\rho_{2r}(y)\leqslant C(m)e^{-R/2}\rho_{2r}(y).$$
\end{proof}
Now, let us proceed the calculation for the stationary vector field. Take $\xi=\tilde{\rho}_r(y-x;\mathcal{L}_{x,r})\Pi_{L^{\perp}}(y-x)$ in the stationary equation (\ref{**}), we get 
\begin{equation}
    \begin{aligned}
        &\int r^2\tilde{\rho}_r\left(\vert\Pi_{x,r}\nabla u_{\epsilon}\vert^2+2\frac{F(u_{\epsilon})}{\epsilon^2}\right)-\frac{1}{2}\hat{\rho}_r\vert\Pi_{x,r}^{\perp}(y-x)\vert^2\left(\vert\Pi_{x,r}\nabla u_{\epsilon}\vert^2+2\frac{F(u_{\epsilon})}{\epsilon^2}\right)\\=&\frac{1}{2}\int\hat{\rho}_r\left\vert\Pi_{x,r}^{\perp}(y-x)\right\vert^2\left(\vert\nabla_{\alpha_{L^{\perp}}}u_{\epsilon}\vert^2-\vert\nabla_{n_{L^{\perp}}}u_{\epsilon}\vert^2\right)+r^2\left\langle\nabla_{\Pi_{x,r}\nabla\tilde{\rho}_r}u_{\epsilon},\nabla_{\Pi_{x,r}^{\perp}(y-x)}u_{\epsilon}\right\rangle.
    \end{aligned}\label{identity for tangential and radial derivative}
\end{equation}
Multiply the both sides of (\ref{identity for tangential and radial derivative}) by $\psi_{\mathcal{T}}(x,r)$ and integrate over $\mathcal{T}_r$, we see that 
\begin{lem}[{cf. \cite[Lemma 10.5]{naber2024energyidentitystationaryharmonic}}]\label{computation 1 radial energy}
    \begin{equation}
        \begin{aligned}
            &2\mathcal{E}_{\mathcal{L}}(\mathcal{T}_r)+\mathcal{E}_n(\mathcal{T}_r)\\=&\mathcal{E}_{\alpha}(\mathcal{T}_r)+\int_{\mathcal{T}_r}\psi_{\mathcal{T}}\int\rho_r(y-x)\vert\Pi_{x,r}^{\perp}(y-x)\vert^2\left(\vert\Pi_{x,r}\nabla u_{\epsilon}\vert^2+2\frac{F(u_{\epsilon})}{\epsilon^2}\right)+I+\mathfrak{F}(r),
        \end{aligned}\nonumber
    \end{equation}
    where $$I=2\int_{\mathcal{T}_r}\psi_{\mathcal{T}}(x,r)\int\dot{\rho}_r(y-x)\left\langle\nabla_{\Pi_{x,r}(y-x)}u_{\epsilon},\nabla_{\Pi_{x,r}^{\perp}(y-x)}u_{\epsilon}\right\rangle.$$
\end{lem}
\begin{proof}
    We estimate the identity term by term. The relationship of first term in left side of (\ref{identity for tangential and radial derivative}) with tangential energy is controlled by Lemma \ref{estimate for tide rho}
    \begin{equation}
        \begin{aligned}
            &\left\vert r^2\int_{\mathcal{T}_r}\psi_{\mathcal{T}}(x,r)\int\tilde{\rho}_r(y-x;\mathcal{L}_{x,r})\left(\vert\Pi_{x,r}\nabla u_{\epsilon}\vert^2+2\frac{F(u_{\epsilon})}{\epsilon^2}\right)-\mathcal{E}_{\mathcal{L}}(\mathcal{T}_r)\right\vert\\\leqslant&r^2\int_{\mathcal{T}_r}\psi_{\mathcal{T}}(x,r)\int\vert e_r(y-x;\mathcal{L}_{x.r})\vert\left(\vert\Pi_{x,r}\nabla u_{\epsilon}\vert^2+2\frac{F(u_{\epsilon})}{\epsilon^2}\right)\\\leqslant&C(m)e^{-R/2}r^2\int_{\mathcal{T}_r}\psi_{\mathcal{T}}(x,r)\int\rho_{2r}(y-x)\left(\vert\Pi_{x,r}\nabla u_{\epsilon}\vert^2+2\frac{F(u_{\epsilon})}{\epsilon^2}\right)\\\leqslant&C(m)e^{-R/2}\int_{\mathcal{T}_r}\psi_{\mathcal{T}}(x,r)\bar{\Theta}_{\mathcal{L}}(x,2r)=\mathfrak{F}(r).
        \end{aligned}\nonumber
    \end{equation}
    Next, since on the support of $\rho_r-\hat{\rho}_r(\cdot,\mathcal{L}_{x,r})$, we have $\vert\Pi_{x,r}^{\perp}(y)\vert\leqslant 2re^{-R}$. We can substitute the second term in left side of (\ref{identity for tangential and radial derivative}) by
    \begin{equation}
        \begin{aligned}
            &\left\vert\int_{\mathcal{T}_r}\psi_{\mathcal{T}}(x,r)\int(\hat{\rho}_r(y-x;\mathcal{L}_{x,r})-\rho_r(y-x))\vert\Pi_{x,r}^{\perp}(y-x)\vert^2\left(\vert\Pi_{x,r}\nabla u_{\epsilon}\vert^2+2\frac{F(u_{\epsilon})}{\epsilon^2}\right)\right\vert\\\leqslant&C(m)e^{-2R}\int_{\mathcal{T}_r}\psi_{\mathcal{T}}(x,r)\bar{\Theta}_{\mathcal{L}}(x,2r)=\mathfrak{F}(r).
        \end{aligned}\nonumber
    \end{equation}
    The cross derivative term in right side of (\ref{identity for tangential and radial derivative}) is replaced by
    \begin{equation}
        \begin{aligned}
            &r^2\left\vert\int_{\mathcal{T}_r}\psi_{\mathcal{T}}(x,r)\int \left\langle \nabla_{\Pi_{x,r}\nabla\tilde{\rho}_r}u_{\epsilon}-\rho_r(y-x)\nabla_{\Pi_{x,r}(y-x)}u_{\epsilon},\nabla_{\Pi_{x,r}^{\perp}(y-x)}u_{\epsilon}\right\rangle\right\vert\\\leqslant&r^2\int_{\mathcal{T}_r}\psi_{\mathcal{T}}(x,r)\int\vert\Pi_{x,r}\nabla e_r(y-x;\mathcal{L}_{x,r})\vert\left\vert\Pi_{x,r}\nabla u_{\epsilon}\right\vert\left\vert\nabla_{\Pi_{x,r}^{\perp}(y-x)}u_{\epsilon}\right\vert\\\leqslant&C(m)e^{-R/2}\int_{\mathcal{T}_r}\psi_{\mathcal{T}}(x,r)\int\rho_{2r}(y-x)\left(r^2\vert\Pi_{x,r}\nabla u_{\epsilon}\vert^2+\vert\nabla_{y-x}u_{\epsilon}\vert^2\right)=\mathfrak{F}(r).
        \end{aligned}\nonumber
    \end{equation}
    Finally, we exchange $\rho_r$ and $\dot{\rho}_r$ using Lemma \ref{properties of heat mollifier}.
    \begin{equation}
        \begin{aligned}
            &\left\vert\int_{\mathcal{T}_r}\psi_{\mathcal{T}}(x,r)\int(\rho_r(y-x)+\dot{\rho}_r(y-x)\left\langle\nabla_{\Pi_{x,r}(y-x)}u_{\epsilon},\nabla_{\Pi_{x,r}^{\perp}(y-x)}u_{\epsilon}\right\rangle\right\vert\\=&\left\vert\int_{\mathcal{T}_r}\psi_{\mathcal{T}}(x,r)\int(\rho_r(y-x)+\dot{\rho}_r(y-x)\left(\left\langle\nabla_{\Pi_{x,r}(y-x)}u_{\epsilon},\nabla_{y-x}u_{\epsilon}\right\rangle-\left\vert\nabla_{\Pi_{x,r}(y-x)}u_{\epsilon}\right\vert^2\right)\right\vert\\\leqslant&C(m)e^{-R/2}\int_{\mathcal{T}_r}\psi_{\mathcal{T}}(x,r)\int\rho_{2r}(y-x)\left(r^2\vert\Pi_{x,r}\nabla u_{\epsilon}\vert^2+\vert\nabla_{y-x}u_{\epsilon}\vert^2\right)=\mathfrak{F}(r).
        \end{aligned}\nonumber
    \end{equation}
    Summing up, we have completed our proof.
\end{proof}
As discussed in the introduction, a priori, $I$ in Lemma \ref{computation 1 radial energy} is not a small term. We need to use the integrate by parts to rewrite it as better ones.
\begin{lem}[{cf. \cite[Lemma 10.6]{naber2024energyidentitystationaryharmonic}}]\label{computation 2 radial energy}
    $I=II+III+\mathfrak{F}(r)$, where $$II=-2\int_{\mathcal{T}_r}\psi_{\mathcal{T}}(x,r)\int\rho_r(y-x)\left\langle\nabla u_{\epsilon},\sff_{\mathcal{T}_r}(\Pi_{\mathcal{T}_r}\nabla u_{\epsilon},\Pi_{\mathcal{T}_r}(y-x))\right\rangle,$$ $$III=-2r^2\int_{\mathcal{T}_r}\psi_{\mathcal{T}}(x,r)\int\rho_r(y-x)\langle\nabla_{H_{\mathcal{T}_r}} u_{\epsilon},\nabla_{\mathcal{T}_r(y-x)}u_{\epsilon}\rangle.$$
\end{lem}
\begin{proof}
    In order to do the integration by parts, we first exchange $\Pi_{x,r}$ with $\Pi_{\mathcal{T}_r}$ using Theorem \ref{apporximating submanifold} (6). We have
    \begin{equation}
        \begin{aligned}
            &\left\vert\int_{\mathcal{T}_r}\psi_{\mathcal{T}}\int\dot{\rho}_r(y-x)\left\langle\nabla_{(\Pi_{x,r}-\Pi_{\mathcal{T}_r})(y-x)}u_{\epsilon},\nabla_{\Pi_{x,r}^{\perp}(y-x)}u_{\epsilon}\right\rangle\right\vert\\\leqslant&C(m,R)\int_{\mathcal{T}_r}\psi_{\mathcal{T}}\Vert\Pi_{x,r}-\Pi_{\mathcal{T}_r}\Vert\sqrt{\bar{\Theta}(x,2r)}\sqrt{\frac{\p}{\p r}\bar{\Theta}(x,2r)}\\\leqslant&C(m,R)\int_{\mathcal{T}_r}\psi_{\mathcal{T}}\left(\sqrt{\frac{\p}{\p r}\bar{\Theta}(x,2r)}+e^{-R/2}\right)\sqrt{\bar{\Theta}_{\mathcal{L}}(x,2r)}\sqrt{\frac{\p}{\p r}\bar{\Theta}(x,2r)}\\\leqslant&C(m,R)(\sqrt{\delta}+e^{-R/2})\int_{\mathcal{T}_r}\psi_{\mathcal{T}}\sqrt{\bar{\Theta}_{\mathcal{L}}(x,2r)}\sqrt{\frac{\p}{\p r}\bar{\Theta}(x,2r)}=\mathfrak{F}(r).
        \end{aligned}\nonumber
    \end{equation}
    Similarly, we can exchange $\Pi_{x,r}^{\perp}$ with $\Pi_{\mathcal{T}_r}^{\perp}$ and a small error. Hence, we deduce that 
    $$I=2\int_{\mathcal{T}_r}\psi_{\mathcal{T}}\int\dot{\rho}_r(y-x)\left\langle\nabla_{\Pi_{\mathcal{T}_r}(y-x)}u_{\epsilon},\nabla_{\Pi_{\mathcal{T}_r}^{\perp}(y-x)}u_{\epsilon}\right\rangle+\mathfrak{F}(r).$$
    Next, let us integrate by parts 
    \begin{equation}
        \begin{aligned}
            &\int\dif y\int_{\mathcal{T}_r}\psi_{\mathcal{T}}(x,r)\dot{\rho}_r(y-x)\left\langle\nabla_{\Pi_{\mathcal{T}_r}(y-x)}u_{\epsilon},\nabla_{\Pi_{\mathcal{T}_r}^{\perp}(y-x)}u_{\epsilon}\right\rangle\dif x\\=&-\sum_{i=1}^Jr^2\int\dif y\int_{\mathcal{T}_r}\psi_{\mathcal{T}}\left\langle\nabla^y u_{\epsilon}^i,\nabla^x_{\mathcal{T}_r}\rho_r(y-x)\right\rangle\left\langle\nabla^yu_{\epsilon}^i,\Pi_{\mathcal{T}_r}^{\perp}(y-x)\right\rangle\\=&\sum_{i=1}^Jr^2\int\dif y\int_{\mathcal{T}_r}\rho_r(y-x)\left\langle\nabla^y u_{\epsilon}^i,\nabla^x_{\mathcal{T}_r}\psi_{\mathcal{T}}\right\rangle\left\langle\nabla^yu_{\epsilon}^i,\Pi_{\mathcal{T}_r}^{\perp}(y-x)\right\rangle\\+&r^2\int\dif y\int_{\mathcal{T}_r}\rho_r(y-x)\psi_{\mathcal{T}}\left\langle\nabla^y u_{\epsilon}^i,\nabla^x_{\mathcal{T}_r}\left\langle\nabla^yu_{\epsilon}^i,\Pi_{\mathcal{T}_r}^{\perp}(y-x)\right\rangle\right\rangle\\+&r^2\int\dif y\int_{\mathcal{T}_r}\psi_{\mathcal{T}}\rho_r(y-x)\textup{div}_{\mathcal{T}_r}\Pi_{\mathcal{T}_r}\nabla^yu_{\epsilon}^i\langle\nabla^yu_{\epsilon}^i,\Pi_{\mathcal{T}_r}^{\perp}(y-x)\rangle.
        \end{aligned}\label{radial calculation in I}
    \end{equation}
    As we have seen before, the first term in (\ref{radial calculation in I}) is a $\mathfrak{F}(r)$ term. The last term can be rephrased by ($H_{\mathcal{T}_r}$ denotes the mean curvature of $\mathcal{T}_r$) $$\textup{div}_{\mathcal{T}_r}\Pi_{\mathcal{T}_r}\nabla^yu_{\epsilon}^i=\langle\nabla^yu_{\epsilon},H_{\mathcal{T}_r}\rangle,$$
    and using the definition equation of $\mathcal{T}_r$, Theorem \ref{apporximating submanifold} (5), we have $$r^2\int_{\mathcal{T}_r}\psi_{\mathcal{T}}\int\rho_r\langle\nabla_{H_{\mathcal{T}_r}}u_{\epsilon},\nabla_{\Pi_{\mathcal{T}_r}^{\perp}(y-x)}u_{\epsilon}\rangle=-r^2\int_{\mathcal{T}_r}\psi_{\mathcal{T}}\int\rho_r\langle\nabla_{H_{\mathcal{T}_r}}u_{\epsilon},\nabla_{\Pi_{\mathcal{T}_r}(y-x)}u_{\epsilon}\rangle=III.$$
    In order to calculate the second term in (\ref{radial calculation in I}), let us choose a geodesic frame $\lbrace E_{\alpha}\rbrace_{\alpha=1}^{m-2}$ at $x$ with $\Pi_{\mathcal{T}_r}\nabla_{E_{\alpha}}E_{\beta}=0$ at $x$. We have 
    \begin{equation}
        \begin{aligned}
            &\sum_{i=1}^J\left\langle\nabla^y u_{\epsilon}^i,\nabla^x_{\mathcal{T}_r}\left\langle\nabla^yu_{\epsilon}^i,\Pi_{\mathcal{T}_r}^{\perp}(y-x)\right\rangle\right\rangle\\=&\sum_{i,\alpha,\beta}\left\langle\nabla^yu_{\epsilon}^i,E_{\alpha}\right\rangle\nabla^x_{E_{\alpha}}\left\langle\nabla^y u_{\epsilon}^i,y-x-\langle y-x,E_{\beta}\rangle E_{\beta}\right\rangle\\=&-\sum_{i,\alpha,\beta}\langle\nabla^yu_{\epsilon}^i,E_{\alpha}\rangle\left(\langle\nabla^yu_{\epsilon}^i,E_{\beta}\rangle\langle y-x,\nabla^x_{E_{\alpha}}E_{\beta}\rangle+\langle y-x,E_{\beta}\rangle\langle\nabla^yu_{\epsilon}^i,\nabla_{E_{\alpha}}E_{\beta}\rangle\right).
        \end{aligned}\nonumber
    \end{equation}
    At $x$, since $\Pi_{\mathcal{T}_r}\nabla_{E_{\alpha}}E_{\beta}=0$ for all $\alpha,\beta$, we have $\nabla_{E_{\alpha}}E_{\beta}=\Pi_{\mathcal{T}_r}^{\perp}\nabla_{E_{\alpha}}E_{\beta}=\sff_{\mathcal{T}_r}(E_{\alpha},E_{\beta})$. Hence at $x$
    \begin{equation}
        \begin{aligned}
            &\sum_{i=1}^L\left\langle\nabla^y u_{\epsilon}^i,\nabla^x_{\mathcal{T}_r}\left\langle\nabla^yu_{\epsilon}^i,\Pi_{\mathcal{T}_r}^{\perp}(y-x)\right\rangle\right\rangle\\=&-\sum_{i,\alpha,\beta}\langle\nabla^yu_{\epsilon}^i,E_{\alpha}\rangle\left(\langle\nabla^yu_{\epsilon}^i,E_{\beta}\rangle\langle y-x,\sff_{\mathcal{T}_r}(E_{\alpha},E_{\beta})\rangle+\langle y-x,E_{\beta}\rangle\langle\nabla^yu_{\epsilon}^i,\sff_{\mathcal{T}_r}(E_{\alpha},E_{\beta})\rangle\right)\\=&-\left\langle y-x,\sff_{\mathcal{T}_r}\left(\Pi_{\mathcal{T}_r}\nabla^y u_{\epsilon},\Pi_{\mathcal{T}_r}\nabla^yu_{\epsilon}\right)\right\rangle-\left\langle\nabla^yu_{\epsilon},\sff_{\mathcal{T}_r}(\Pi_{\mathcal{T}_r}\nabla^yu_{\epsilon},\Pi_{\mathcal{T}_r}(y-x))\right\rangle.
        \end{aligned}\nonumber
    \end{equation}
    Note that, by the estimate of the second fundamental form of $\mathcal{T}_r$, Theorem \ref{apporximating submanifold} (1), we can exchange again $\Pi_{\mathcal{T}_r}$ and $\Pi_{x,r}$ to get 
    \begin{equation}
        \begin{aligned}
            &\left\vert r^2\int_{\mathcal{T}_r}\psi_{\mathcal{T}}\int\langle y-x,\sff_{\mathcal{T}_r}(\Pi_{\mathcal{T}_r}\nabla^y u_{\epsilon},\Pi_{\mathcal{T}_r}\nabla^y u_{\epsilon})\rangle\right\vert\\\leqslant& C(m,\varepsilon_0)\sqrt{\delta}\int_{\mathcal{T}_r}\psi_{\mathcal{T}}(x,r)\bar{\Theta}_{\mathcal{L}}(x,r)+\mathfrak{F}(r)=\mathfrak{F}(r).
        \end{aligned}
    \end{equation}
    Combining above calculations, we have got the desired estimate.
\end{proof}
We will see later that $II$ above is a good term as it naturally shows in the $r$-derivative of $\mathcal{E}_{\mathcal{L}}(\mathcal{T}_r)$. Now, let us do the integration by parts once more to rewrite $III$ as
\begin{lem}[{cf. \cite[Lemma 10.7(2)]{naber2024energyidentitystationaryharmonic}}]\label{computation 3 radial energy}
    $$III=2r^4\int_{\mathcal{T}_r}\psi_{\mathcal{T}}(x,r)\int\rho_r(y-x)\left\vert\nabla_{H_{\mathcal{T}_r}}u_{\epsilon}\right\vert^2+\mathfrak{F}(r).$$
\end{lem}
\begin{proof}
    Substituting $\Pi_{\mathcal{T}_r}^{\perp}(y-x)$ with $\Pi_{\mathcal{T}_r}(y-x)$ in $III$ gives us another chance to do the integration by parts. We first replace $\rho_r$ with $\dot{\rho}_r$: by the second fundamental form estimate for $\mathcal{T}_r$, Theorem \ref{apporximating submanifold} (2),
    $$\left\vert r^2\int_{\mathcal{T}_r}\psi_{\mathcal{T}}\int(\rho_r+\dot{\rho}_r)\langle\nabla_{H_{\mathcal{T}_r}}u_{\epsilon},\nabla_{\Pi_{\mathcal{T}_r}(y-x)}u_{\epsilon}\rangle\right\vert\leqslant C(m,R,\varepsilon_0)e^{-R/2}\int_{\mathcal{T}_r}\psi_{\mathcal{T}}\bar{\Theta}_{\mathcal{L}}(x,2r)=\mathfrak{F}(r).$$
    Hence use same idea as in Lemma \ref{computation 2 radial energy}, we can calculate
    \begin{equation}
        \begin{aligned}                                                     III=&-2\sum_ir^4\int_{\mathcal{T}_r}\psi_{\mathcal{T}}\int\langle\nabla^yu_{\epsilon}^i,H_{\mathcal{T}_r}\rangle\langle\nabla^yu_{\epsilon}^i,\nabla^x\rho_r(y-x)\rangle+\mathfrak{F}(r)\\=&2r^4\int_{\mathcal{T}_r}\psi_{\mathcal{T}}\int\rho_r\left\vert\nabla_{H_{\mathcal{T}_r}}u_{\epsilon}\right\vert^2+2r^4\sum_i\int_{\mathcal{T}_r}\psi_{\mathcal{T}}\int\rho_r\langle\nabla^yu_{\epsilon}^i,\nabla^x_{\mathcal{T}_r}\langle\nabla^yu_{\epsilon}^i,H_{\mathcal{T}_r}\rangle\rangle+\mathfrak{F}(r).
        \end{aligned}\nonumber
    \end{equation}
    The first term is exactly what we need. To show that the second term is a $\mathfrak{F}(r)$, we compute in the frame $\lbrace E_{\alpha}\rbrace$
    \begin{equation}
        \langle\nabla^yu_{\epsilon}^i,\nabla^x_{\mathcal{T}_r}\langle\nabla^yu_{\epsilon}^i,H_{\mathcal{T}_r}\rangle\rangle=\sum_{\alpha}\langle\nabla^yu_{\epsilon}^i,E_{\alpha}\rangle\langle\nabla^yu_{\epsilon}^i,\nabla_{E_{\alpha}}^xH_{\mathcal{T}_r}\rangle.\nonumber
    \end{equation}
    We split into $\nabla_{E_{\alpha}}^xH_{\mathcal{T}_r}$ into tangential and perpendicular parts, $\nabla_{E_{\alpha}}^xH_{\mathcal{T}_r}=\Pi_{x,r}\nabla_{E_{\alpha}}^xH_{\mathcal{T}_r}+\Pi_{x,r}^{\perp}\nabla_{E_{\alpha}}^xH_{\mathcal{T}_r}$. The tangential part is a priori small 
    \begin{equation}
        \begin{aligned}
            &r^4\left\vert\int_{\mathcal{T}_r}\psi_{\mathcal{T}}\int\rho_r\langle\nabla^yu_{\epsilon}^i,E_{\alpha}\rangle\langle\nabla^yu_{\epsilon}^i,\Pi_{x,r}\nabla_{E_{\alpha}}^xH_{\mathcal{T}_r}\rangle\right\vert\\\leqslant &C(m,\varepsilon_0)\sqrt{\delta}\int_{\mathcal{T}_r}\psi_{\mathcal{T}}\bar{\Theta}_{\mathcal{L}}(x,2r)+\mathfrak{F}(r)=\mathfrak{F}(r),
        \end{aligned}\nonumber
    \end{equation}
    while since $\Pi_{x,r}^{\perp}\nabla_{E_{\alpha}}^xH_{\mathcal{T}_r}\nabla_{E_{\alpha}}^xH_{\mathcal{T}_r}$ and $E_{\alpha}$ are independent of $y$ and lies in $\mathcal{L}_{x,r}^{\perp},\mathcal{L}_{x,r}$ respectively, by the definition of $\mathcal{L}_{x,r}$, we have $$\sum_i\int\rho_r(y-x)\langle\nabla u_{\epsilon}^i,E_{\alpha}\rangle\langle\nabla u_{\epsilon}^i,\Pi_{\mathcal{T}_r}^{\perp}\nabla_{E_{\alpha}}^xH_{\mathcal{T}_r}\rangle=Q(x,r)[\Pi_{x,r}E_{\alpha},\Pi_{x,r}^{\perp}\nabla_{E_{\alpha}}^xH_{\mathcal{T}_r}]+\mathfrak{F}(r)=\mathfrak{F}(r).$$
    Hence we have completed our proof.
\end{proof}
For the purpose of estimate, we rewrite $II$ to be a positive term.
\begin{lem}[{cf. \cite[Lemma 10.7(1)]{naber2024energyidentitystationaryharmonic}}]\label{computation 4 radial energy}
    $$II=2\int_{\mathcal{T}_r}\psi_{\mathcal{T}}(x,r)\int\rho_r(y-x)\langle\nabla u_{\epsilon},\sff_{\mathcal{T}_r}\rangle^2+\mathfrak{F}(r).$$
\end{lem}
\begin{proof}
    With the similar calculation as in the previous paragraphs, we have
    \begin{equation}
        \begin{aligned}
            &-\int_{\mathcal{T}_r}\psi_{\mathcal{T}}\int\rho_r\langle\nabla u_{\epsilon},\sff_{\mathcal{T}_r}(\Pi_{\mathcal{T}_r}\nabla u_{\epsilon},\Pi_{\mathcal{T}_r}(y-x))\rangle\\=&\int_{\mathcal{T}_r}\psi_{\mathcal{T}}\int\dot{\rho}_r\langle\nabla u_{\epsilon},\sff_{\mathcal{T}_r}(\Pi_{\mathcal{T}_r}\nabla u_{\epsilon},\Pi_{\mathcal{T}_r}(y-x))\rangle+\mathfrak{F}(r)\\=&\int_{\mathcal{T}_r}\psi_{\mathcal{T}}\int\rho_r\langle\nabla u_{\epsilon},\sff_{\mathcal{T}_r}\rangle^2+\sum_{i,\alpha}\int_{\mathcal{T}_r}\psi_{\mathcal{T}}\int\rho_r\langle\nabla^yu_{\epsilon}^i,E_{\alpha}\rangle\langle\nabla^yu_{\epsilon}^i,\nabla_{E_{\alpha}}^xH_{\mathcal{T}_r}\rangle+\mathfrak{F}(r)\\=&\int_{\mathcal{T}_r}\psi_{\mathcal{T}}\int\rho_r\langle\nabla u_{\epsilon},\sff_{\mathcal{T}_r}\rangle^2+\mathfrak{F}(r).
        \end{aligned}\nonumber
    \end{equation}
\end{proof}
\begin{cor}
    $$II+III\leqslant2(m-1)\int_{\mathcal{T}_r}\psi_{\mathcal{T}}\int\rho_r(y-x)\langle\nabla u_{\epsilon},\sff_{\mathcal{T}_r}\rangle^2+\mathfrak{F}(r).$$
\end{cor}
The proof of Proposition \ref{local estimate for tangential and radial energy} is a consequence of Lemma \ref{computation 1 radial energy},\ref{computation 2 radial energy},\ref{computation 3 radial energy},\ref{computation 4 radial energy} and the following calculation of $\frac{\dif}{\dif r}\mathcal{E}_\mathcal{L}(\mathcal{T}_r)$.
\begin{lem}[{cf. \cite[Lemma 10.8]{naber2024energyidentitystationaryharmonic}}]
    \begin{equation}
        \begin{aligned}
            &r\frac{\dif}{\dif r}\mathcal{E}_{\mathcal{L}}(\mathcal{T}_r)\\=&\int_{\mathcal{T}_r}\psi_{\mathcal{T}}(x,r)\int\rho_r(y-x)\vert\Pi_{x,r}^{\perp}(y-x)\vert^2\left(\vert\Pi_{x,r}\nabla u_{\epsilon}\vert^2+2\frac{F(u_{\epsilon})}{\epsilon^2}\right)+II+\mathfrak{F}(r).
        \end{aligned}\nonumber
    \end{equation}
\end{lem}
\begin{proof}
    Direct calculation shows
    \begin{equation}
        \begin{aligned}
            r\frac{\dif}{\dif r}\mathcal{E}_{\mathcal{L}}(\mathcal{T}_r)&=\int_{\mathcal{T}_r}\left(r\frac{\p}{\p r}\psi_{\mathcal{T}}+\psi_{\mathcal{T}}\left\langle H_{\mathcal{T}_r},r\frac{\p}{\p r}\mathcal{T}_r\right\rangle\right)\bar{\Theta}_{\mathcal{L}}(x,r)-(m-2)\mathcal{E}_{\mathcal{L}}(\mathcal{T}_r)\\&-\int_{\mathcal{T}_r}\psi_{\mathcal{T}}\int \dot{\rho}_r(y-x)\vert y-x\vert^2\left(\vert\Pi_{x,r}\nabla u_{\epsilon}\vert^2+2\frac{F(u_{\epsilon})}{\epsilon^2}\right)\\&+\int_{\mathcal{T}_r}\psi_{\mathcal{T}}\int r^2\rho_r(y-x)r\frac{\p}{\p r}\vert\Pi_{x,r}\nabla u_{\epsilon}\vert^2.
        \end{aligned}\nonumber
    \end{equation}
    Using estimates similar to before, the first and the last term are $\mathfrak{F}(r)$ terms. On the other hand, we can integrate by parts to deal with the rest terms. First, we exchange $\dot{\rho}_r$ with $\rho_r$ and a small error
    $$\int_{\mathcal{T}_r}\psi_{\mathcal{T}}\int (\dot{\rho}_r(y-x)+\rho_r(y-x))\vert\Pi_{x,r}^{\perp}(y-x)\vert^2\left(\vert\Pi_{x,r}\nabla u_{\epsilon}\vert^2+2\frac{F(u_{\epsilon})}{\epsilon^2}\right)=\mathfrak{F}(r).$$
    Then
    \begin{equation}
        \begin{aligned}
            &\int_{\mathcal{T}_r}\psi_{\mathcal{T}}\int\dot{\rho}_r(y-x)\vert\Pi_{x,r}(y-x)\vert^2\left(\vert\Pi_{x,r}\nabla u_{\epsilon}\vert^2+2\frac{F(u_{\epsilon})}{\epsilon^2}\right)\\=&\int_{\mathcal{T}_r}\psi_{\mathcal{T}}\int\dot{\rho}_r(y-x)\vert\Pi_{\mathcal{T}_r}(y-x)\vert^2\left(\vert\Pi_{\mathcal{T}_r}\nabla u_{\epsilon}\vert^2+2\frac{F(u_{\epsilon})}{\epsilon^2}\right)+\mathfrak{F}(r)\\=&\int_{\mathcal{T}_r}\psi_{\mathcal{T}}\int r^2\left\langle\nabla_{\mathcal{T}_r}^x{\rho}_r(y-x),\Pi_{\mathcal{T}_r}(x-y)\right\rangle\left(\vert\Pi_{\mathcal{T}_r}\nabla u_{\epsilon}\vert^2+2\frac{F(u_{\epsilon})}{\epsilon^2}\right)+\mathfrak{F}(r)\\=&-(m-2)\mathcal{E}_{\mathcal{L}}(\mathcal{T}_r)-\int_{\mathcal{T}_r}\int r^2\left\langle\nabla^x\psi_{\mathcal{T}},\Pi_{\mathcal{T}_r}(x-y)\right\rangle\rho_r(y-x)\left(\vert\Pi_{\mathcal{T}_r}\nabla u_{\epsilon}\vert^2+2\frac{F(u_{\epsilon})}{\epsilon^2}\right)\\-&\int_{\mathcal{T}_r}\psi_{\mathcal{T}}\int r^2\rho_r(y-x)\left\langle\nabla^x\vert\Pi_{\mathcal{T}_r}\nabla u_{\epsilon}\vert^2,\Pi_{\mathcal{T}_r}(x-y)\right\rangle+\mathfrak{F}(r).
        \end{aligned}\nonumber
    \end{equation}
    By previous properties, it is easy to bound
    \begin{equation}
        \begin{aligned}
            &\left\vert\int_{\mathcal{T}_r}\int r^2\left\langle\nabla^x\psi_{\mathcal{T}},\Pi_{\mathcal{T}_r(x)}(x-y)\right\rangle\rho_r(y-x)\left(\vert\Pi_{x,r}\nabla u_{\epsilon}\vert^2+2\frac{F(u_{\epsilon})}{\epsilon^2}\right)\right\vert\\\leqslant& C(m,R)\sqrt{\delta}\int_{\mathcal{T}_r}r\vert\nabla\psi_{\mathcal{T}}\vert=\mathfrak{F}(r).
        \end{aligned}\nonumber
    \end{equation}
    On the other hand, we can calculate
    \begin{equation}
        \begin{aligned}
           &\int_{\mathcal{T}_r}\psi_{\mathcal{T}}\int r^2\rho_r(y-x)\left\langle\nabla^x\vert\Pi_{x,r}\nabla u_{\epsilon}\vert^2,\Pi_{\mathcal{T}_r(x)}(y-x)\right\rangle\\=&-2\int_{\mathcal{T}_r}\psi_{\mathcal{T}}\int\rho_r(y-x)\langle\nabla u_{\epsilon},\sff_{\mathcal{T}_r}(\Pi_{\mathcal{T}_r}\nabla u_{\epsilon},\Pi_{\mathcal{T}_r}(y-x)\rangle=II.
        \end{aligned}\nonumber
    \end{equation}
    Hence, we have completed our proof.
\end{proof}
\begin{proof}[Proof of Theorem \ref{radial and tangential energy estimate}]
    Integration the both sides of Proposition \ref{local estimate for tangential and radial energy} yields
    \begin{equation}
        \begin{aligned}
            &\int(\hat{\mathcal{E}}_{\alpha}(\mathcal{T}_r)+2\mathcal{E}_{\mathcal{L}}(\mathcal{T}_r))\frac{\dif r}{r}\\\leqslant&C(m,K_N,\Lambda,R)\left((\sqrt{\delta}+e^{-R/2})\left(\int\mathcal{E}_{\mathcal{L}}(\mathcal{T}_r)\frac{\dif r}{r}+\Lambda)\right)+\int\mathcal{E}_{\alpha}(\mathcal{T}_r)\frac{\dif r}{r}+\sqrt{\delta}\right)\\\leqslant&\int\mathcal{E}_{\mathcal{L}}(\mathcal{T}_r)\frac{\dif r}{r}+\sigma.
        \end{aligned}\nonumber
    \end{equation}
    provided we choose $R\geqslant R(m,K_N,\Lambda,\sigma)$ and then choose $\delta\leqslant\delta(m,K_N,\Lambda,\sigma)$. This completes the proof.
\end{proof}

\bibliography{reference}
\bibliographystyle{plain}
\end{document}